\documentclass[10pt]{amsart}

\usepackage{anysize}
\usepackage[utf8x]{inputenc}
\usepackage[english]{babel}
\usepackage{amssymb, amsmath, amsthm}
\usepackage{enumitem}
\usepackage{mathtools}
\usepackage{url}
\usepackage[all,cmtip]{xy}
\usepackage[hidelinks,bookmarksnumbered]{hyperref}
\usepackage{tikz-cd}

\usepackage{cleveref}

\usepackage{mathtools}
\usepackage{url}
\usepackage[all,cmtip]{xy}
\usepackage{tikz}
\usepackage{tikz-cd}
\usepackage{cleveref}
	\usetikzlibrary{matrix,arrows,decorations,decorations.pathmorphing,positioning,decorations.pathreplacing,shapes}
	\tikzset{commutative diagrams/.cd, 
		mysymbol/.style = {start anchor=center, end anchor = center, draw = none}}

	\newcommand{\commutes}[2][\circ]{\arrow[mysymbol]{#2}[description]{#1}}

\newtheorem{theorem}{Theorem}[section]
\newtheorem{proposition}[theorem]{Proposition}
\newtheorem{lemma}[theorem]{Lemma}
\newtheorem{corollary}[theorem]{Corollary}

\theoremstyle{definition}
\newtheorem{definition}[theorem]{Definition}
\newtheorem{remark}[theorem]{Remark}
\newtheorem{example}[theorem]{Example}

\newtheorem{question}[theorem]{Question}
\newtheorem{notation}[theorem]{Notation}

\let\cal\mathcal
\renewcommand\AA{{\cal A}}

\newcommand\CC{{\cal C}}
\newcommand\DD{{\cal D}}
\newcommand\EE{{\cal E}}
\newcommand\FF{{\cal F}}

\newcommand\II{{\cal I}}
\newcommand\JJ{{\cal J}}

\newcommand\NN{{\cal N}}

\newcommand\TT{{\cal T}}

\let\blb\mathbb
\newcommand\bC{{\blb C}}

\newcommand\bZ{{\blb Z}}

\newcommand\bN{{\blb N}}

\let\frak\mathfrak

\newcommand\ff{\frak{f}}

\newcommand\fE{\frak{E}} 
\newcommand\fM{\frak{M}}

\newcommand{\Ab}{\mathsf{\mathbf{Ab}}}

\DeclareMathOperator{\inflation}{\rightarrowtail}
\DeclareMathOperator{\deflation}{\twoheadrightarrow}
\DeclareMathOperator{\LB}{\mathsf{LB}}
\DeclareMathOperator{\LH}{LH}

\DeclareMathOperator{\im}{im}
\DeclareMathOperator{\coim}{coim}
\DeclareMathOperator{\coker}{coker}
\DeclareMathOperator{\cone}{cone}

\DeclareMathOperator{\Hom}{Hom}

\DeclareMathOperator{\Mor}{Mor}
\DeclareMathOperator{\Mon}{Mon}
\DeclareMathOperator{\Inf}{Inf}
\DeclareMathOperator{\WInf}{WInf}
\DeclareMathOperator{\hMon}{hMon}
\DeclareMathOperator{\hInf}{hInf}
\DeclareMathOperator{\hWInf}{hWInf}

\DeclareMathOperator{\Glid}{Glid}
\DeclareMathOperator{\Preglid}{Preglid}

\DeclareMathOperator{\Mod}{Mod}
\DeclareMathOperator{\smod}{mod}
\renewcommand{\mod}{\operatorname{mod}}
\DeclareMathOperator{\modone}{mod^1}
\DeclareMathOperator{\modoneinf}{mod^1_{\mathsf{w}.\mathsf{adm}}}
\DeclareMathOperator{\smodad}{mod_{\mathsf{adm}}}
\DeclareMathOperator{\modad}{mod_{\mathsf{adm}}}
\DeclareMathOperator{\modadone}{mod_{\mathsf{adm}}^1}
\DeclareMathOperator{\smodinf}{mod_{\mathsf{w}.\mathsf{adm}}}
\DeclareMathOperator{\modinf}{mod_{\mathsf{w}.\mathsf{adm}}}
\DeclareMathOperator{\modinfone}{mod_{\mathsf{w}.\mathsf{adm}}^1}

\DeclareMathOperator{\eff}{eff}

\DeclareMathOperator{\Fun}{Fun}
\DeclareMathOperator{\Ob}{Ob}

\DeclareMathOperator{\FFinf}{mod_{\mathsf{w}.\mathsf{adm}}^1}
\DeclareMathOperator{\smodadinf}{mod_{\mathsf{w}.\mathsf{adm}}}
\DeclareMathOperator{\resdim}{res.dim}

\DeclareMathOperator{\C}{\mathbf{C}}
\DeclareMathOperator{\D}{\mathbf{D}}
\DeclareMathOperator{\K}{\mathbf{K}}
\DeclareMathOperator{\Ac}{\mathbf{Ac}}
\DeclareMathOperator{\AcC}{\mathbf{Ac}_{\mathbf{C}}}
\DeclareMathOperator{\Acb}{\mathbf{Ac}^b}
\DeclareMathOperator{\AcbC}{\mathbf{Ac}^b_{\mathbf{C}}}

\DeclareMathOperator{\AcK}{\mathbf{Ac}_{\mathbf{K}}}

\DeclareMathOperator{\Rex}{Rex}

\DeclareMathOperator{\Dstar}{\mathbf{D}^*}
\DeclareMathOperator{\Kstar}{\mathbf{K}^*}

\DeclareMathOperator{\Dm}{\mathbf{D}^-}

\DeclareMathOperator{\Db}{\mathbf{D}^b}
\DeclareMathOperator{\Kb}{\mathbf{K}^b}

\DeclareMathOperator{\DAb}{\mathbf{D}_{\mathcal{A}}^b}

\newcommand{\PLS}{\mathsf{PLS}}
\newcommand{\PLSW}{\mathsf{PLS_W}}
\newcommand{\PLN}{\mathsf{PLN}}

\newcommand{\ex}[1]{{#1}^{\text{ex}}}
\newcommand{\dqq}{{/\mkern-6mu/}}
\renewcommand{\emptyset}{\varnothing}

\DeclareMathOperator{\Funex}{\mathsf{Fun}_{\text{ex}}}

\newcommand{\Yoneda}{\mathbb{Y}}
\renewcommand{\Upsilon}{\Yoneda}

\numberwithin{equation}{section}

\begin{document}

\title{The left heart and exact hull of an additive regular category}



\author{Ruben Henrard}
\address{Ruben Henrard, Hasselt University, Campus Diepenbeek, Department WNI, 3590 Diepenbeek, Belgium}
\email{ruben.henrard@uhasselt.be}

\author{Sondre Kvamme}
\address{Sondre Kvamme, Department of Mathematical Sciences, Norwegian University of Science and Technology, 7491 Trondheim, Norway}
\email{sondre.kvamme@ntnu.no}

\author{Adam-Christiaan van Roosmalen}
\address{Adam-Christiaan van Roosmalen, Xi’an Jiaotong-Liverpool University, Suzhou, 215123, P. R. China}
\email{adamchristiaan.vanroosmalen@uhasselt.be}

\author{Sven-Ake Wegner}
\address{Sven-Ake Wegner, University of Hamburg, Bundesstra{\ss}e 55, 20146 Hamburg, Germany} \email{sven.wegner@uni-hamburg.de}

\subjclass[18E40, 18G80, 46A13, 46M18]{18E05, 18E08, 18E20, 18E35}

\makeatletter
\newcommand{\myitem}[1]{%
\item[#1]\protected@edef\@currentlabel{#1}%
}
\makeatother


\begin{abstract}
	Quasi-abelian categories are abundant in functional analysis and representation theory.  It is known that a quasi-abelian category $\EE$ is a cotilting torsionfree class of an abelian category.  In fact, this property characterizes quasi-abelian categories.  This ambient abelian category is derived equivalent to the category $\EE$, and can be constructed as the heart $\mathcal{LH}(\EE)$ of a $t$-structure on the bounded derived category $\Db(\EE)$ or as the localization of the category of monomorphisms in $\EE.$
	
	However, there are natural examples of categories in functional analysis which are not quasi-abelian, but merely one-sided quasi-abelian or even weaker.  Examples are the category of $\operatorname{LB}$-spaces or the category of complete Hausdorff locally convex spaces.  In this paper, we consider additive regular categories as a generalization of quasi-abelian categories that covers the aforementioned examples.  Additive regular categories can be characterized as those subcategories of abelian categories which are closed under subobjects.
	
	As for quasi-abelian categories, we show that such an ambient abelian category of an additive regular category $\EE$ can be found as the heart of a $\operatorname{t}$-structure on the bounded derived category $\Db(\EE)$, or as the localization of the category of monomorphisms of $\EE$.  In our proof of this last construction, we formulate and prove a version of Auslander's formula for additive regular categories.  
	
	Whereas a quasi-abelian category is an exact category in a natural way, an additive regular category has a natural one-sided exact structure.  Such a one-sided exact category can be 2-universally embedded into its exact hull.  We show that the exact hull of an additive regular category is again an additive regular category.
\end{abstract}

\maketitle


\section{Introduction}

Quasi-abelian categories are a well-behaved class of additive categories, generalizing the notion of an abelian category.  They are preabelian categories such that the class of all kernel-cokernel pairs satisfies the axioms of a Quillen exact category.  Quasi-abelian categories occur often in functional analysis, and motivating examples include the categories of Banach spaces and of Fr\'{e}chet spaces \cite{Prosmans00}.

In \cite{BondalvandenBergh03,Rump01,Schneiders99}, a characterization of a quasi-abelian category is given as follows: quasi-abelian categories are precisely those categories which occur as a cotilting torsionfree class in an abelian category.  For a quasi-abelian category $\EE$, such an ambient abelian category $\AA$ is (essentially) unique.  A construction is given in \cite[\S~1.2]{Schneiders99}: one can obtain the category $\AA$ as the heart of a $\operatorname{t}$-structure on the bounded derived category $\Db(\EE).$  This $\operatorname{t}$-structure is called the left $\operatorname{t}$-structure by Schneiders in \cite{Schneiders99} and the associated heart is then called the left heart, denoted by $\mathcal{LH}(\EE)$.  Furthermore, Schneiders shows that the embedding $\EE\to \mathcal{LH}(\EE)$ lifts to a triangle equivalence $\D(\EE)\stackrel{\simeq}{\rightarrow} \D(\mathcal{LH}(\EE))$, essentially reducing homological properties of the quasi-abelian category $\EE$ to those of the abelian category $\mathcal{LH}(\EE)$. Schneiders also shows that $\mathcal{LH}(\EE)$ is equivalent to a localization of the monomorphism category of $\EE$ with respect to the bicartesian squares (see \cite[Corollary~1.2.21]{Schneiders99}).

For several quasi-abelian categories arising in functional analysis, such as the category of Banach spaces, this last construction of the left heart can already be found in \cite{Waelbroeck71}, before the introduction of $\operatorname{t}$-structures in \cite{BeilinsonBernsteinDeligne82}.  Indeed, it is noted in \cite[Exemple 1.3.24]{BeilinsonBernsteinDeligne82} that Waelbroeck's construction can be recovered using $\operatorname{t}$-structures.

Waelbroeck's approach nonetheless suggests similar ambient abelian categories could also be found in non quasi-abelian settings.  Indeed, in \cite{Wegner17} this was done for what was there called \emph{Waelbroeck categories}.  These include categories such as the non quasi-abelian category $\LB$ of $\operatorname{LB}$-spaces (see section \ref{Section:LBSpaces}).  

This leads to the following natural question: how similar is the situation to that of quasi-abelian categories.  Specifically, one can ask the following questions.

\begin{question}\label{Question:Intro}
\begin{enumerate}
	\item Can these ambient abelian categories be obtained as the heart of a $\operatorname{t}$-structure on some natural triangulated category?
	\item What characterizes the embedding $\EE \to \mathcal{LH}(\EE)$?
\end{enumerate}
\end{question}

When trying to solve the above question for $\LB$, one might be tempted to search an appropriate exact structure on $\LB$ (so that the derived category is well-defined) such that the ambient abelian category is obtained as the heart of a natural $\operatorname{t}$-structure. In fact, the category $\LB$ has several natural exact structures; we will recall some of these in \Cref{Section:LBSpaces}.  However, we show in \Cref{theorem:NoDerivedEquivalenceForExact} that this approach cannot be successful: none of these exact structures yield a well-suited derived category.  Instead, we relax the conditions of an exact category and take the derived category of such a weaker structure.  Our starting point is the recent observation in \cite{HassounShahWegner20} that the category $\LB$ is left quasi-abelian and, as such, carries a natural one-sided exact structure.  It is possible to construct the derived category of $\LB$ with respect to this one-sided exact structure.  In this paper, we show that this derived category provides a good framework to answer the above questions.

Before addressing the above questions or describing the results in this paper, we sketch the setting more accurately.  We work with a slight generalization of a left quasi-abelian category, namely with additive regular categories.  An \emph{additive regular} category is an additive category where (i) every morphism has a cokernel-monomorphism factorization, and (ii) cokernels have pullbacks and the pullback of a cokernel is again a cokernel.  The difference between a left quasi-abelian category and an additive regular category is that the latter need not have cokernels.  An additive regular category is an additive category which is regular in the sense of \cite{BarrGrilletVanOsdol71, BorceuxBourn04}.

Whereas a quasi-abelian category is an exact category, an additive regular category (or even a left quasi-abelian category) need only satisfy those axioms of an exact category that pertain to the cokernel-side of the exact sequences: it is a deflation-exact category (see \Cref{subsection:PrelimOneSided} for a detailed definition).  Even though the axioms of a deflation-exact category are weaker than those of an exact category, deflation-exact categories still satisfy many attractive homological properties similar to those of an exact category, such as the Short Five Lemma, the Snake Lemma, and the Nine Lemma (see \cite{BazzoniCrivei13,HenrardvanRoosmalen20Obscure}).  One possible explanation for this nice behavior is given by the existence of the exact hull (\cite{HenrardvanRoosmalen19b, Rosenberg11}): a deflation-exact category $\EE$ can be embedded in a 2-universal way in an exact category $\ex{\EE}$; this embedding lifts to a derived equivalence $\Db(\EE) \to \Db(\ex{\EE})$.

The exact hull of a left quasi-abelian category need not be quasi-abelian (or even pre-abelian).  In contrast, we show that the exact hull of an additive regular category is again additive regular.  Similarly, we show that additive regular categories are stable under taking quotients (in the sense of \cite{HenrardvanRoosmalen19a}).  Furthermore, the following proposition (see \Cref{Proposition:ClosedUnderSubjectsInheritsDeflationExactAndAdmissibeKernels} in the text) gives a straightforward source of examples.

\begin{proposition}\label{Introduction:Proposition:ClosedUnderSubobjectsInheritsHavingAdmissibleKernels}
Let $\EE$ be an additive regular category.  Any full subcategory $\EE' \subseteq \EE$ which is closed under subobjects, is also additive regular.
\end{proposition}

As abelian categories are additive regular, the previous proposition gives an easy way to find additive regular categories inside an abelian category.  In fact, it follows from \Cref{theorem:MainTheoremIntroduction} below that every additive regular category occurs in this way.

We mentioned that an additive regular category $\EE$ has a natural structure of a deflation-exact category.  As such, one can consider the bounded derived category $\Db(\EE)$; we recall the construction in \S\ref{subsection:DerivedCategory}.  In this setting, the construction of the left heart for quasi-abelian categories given in \cite{Schneiders99} generalizes to the setting of additive regular categories.  We obtain the following theorem directly generalizing the properties we mentioned for quasi-abelian categories.

\begin{theorem}\label{theorem:MainTheoremIntroduction}
Let $\EE$ be an additive regular category.  There is an embedding of $\EE$ into an abelian category $\mathcal{LH}(\EE)$, characterized by the following properties:
\begin{enumerate}
	\item $\EE$ is closed under subobjects in $\mathcal{LH}(\EE)$,
	\item every object in $\mathcal{LH}(\EE)$ is a quotient of an object in $\EE$.
\end{enumerate}
The embedding $\EE \to \mathcal{LH}(\EE)$ lifts to a triangle equivalence $\Dstar(\EE) \to \Dstar(\mathcal{LH}(\EE))$, for $\ast \in \{\varnothing, -, b\}.$
\end{theorem}

The characterization in this theorem follows from combining Propositions \ref{Proposition:TheEmbeddingToHeartProperties} and \ref{proposition:NewEmbedding}; the last statement follows from Proposition \ref{Proposition:PhiFactorsThroughHull}.  

In section \ref{Section:AsLocalizationOfMonomorphisms}, we show that the left heart $\mathcal{LH}(\EE)$ of an additive regular category $\EE$ can be obtained by localizing the category of monomorphisms $\hMon(\EE)$ in $\EE$ (up to homotopy) at the bicartesian squares. This recovers Waelbroeck's construction as well as the construction of the left heart of the $\LB$-spaces as in \cite{Wegner17} (see section \ref{Section:LBSpaces} for more details on the latter).  The following theorem provides a construction of the left heart $\mathcal{LH}(\EE)$ that does not refer to the derived category $\Db(\EE).$

\begin{theorem}\label{theorem:MonoIntroduction}
Let $\hMon(\EE)$ be the category whose objects are monomorphisms $\delta_E\colon E^{-1} \hookrightarrow E^0$ in $\EE$ and whose morphisms are commutative squares
	\[\begin{tikzcd}
		{E^{-1}}\arrow[r, hookrightarrow, "\delta_E"] \arrow[d, "u_{-1}"] & {E^0} \arrow[d,"{u_0}"]\\
		F^{-1}\arrow[r, hookrightarrow, "\delta_F"] & F^0
		\end{tikzcd}\]
	up to homotopy (meaning that there is a morphism $t\colon E^0 \to F^{-1}$ in $\EE$ such that $u_{-1} = t \circ \delta_E$ and $u_0 = \delta_F \circ t$).
	\begin{enumerate}
		\item The set $S$ of all morphisms which are bicartesian squares is a multiplicative system in $\hMon(\EE).$
		\item The localization $\hMon(\EE)[S^{-1}]$ is equivalent to the left heart $\mathcal{LH}(\EE).$
	\end{enumerate}
\end{theorem}

Our proof of \Cref{theorem:MonoIntroduction} is based on Auslander's Formula (see \Cref{Section:ConstructionOfTheLeftHeart} and \Cref{Section:AsLocalizationOfMonomorphisms}) and follows \cite{Rump01}.  We consider the category $\mod \EE$ of finitely presented functors on $\EE$.  As $\EE$ has kernels, $\mod \EE$ is an abelian category.  We show that the subcategory $\eff \EE$ of effaceable functors is a hereditary torsion class in $\mod \EE$; the corresponding torsionfree class is the category $\modone(\EE)$ of objects of projective dimension at most one.  Using the Yoneda embedding $\EE \to \mod \EE$, it is straightforward to show that $\hMon(\EE) \simeq \modone(\EE).$  The proof of \Cref{theorem:MonoIntroduction} then follows from studying the composition $\hMon(\EE) \simeq \modone(\EE) \to \mod(\EE) \to \mod(\EE) / \eff(\EE) \simeq \mathcal{LH}(\EE).$

\textbf{Acknowledgments.} The authors thank Luisa Fiorot and Michel Van den Bergh for helpful discussions.  The second author would like to thank the Hausdorff Institute for Mathematics in Bonn, since parts of the paper were written during his stay at the junior trimester program “New Trends in Representation Theory”. The third author gratefully acknowledges the support received from the Research Foundation-Flanders (FWO), 12.M33.16N.
\section{Preliminaries}\label{section:Preliminaries}

This section is preliminary in nature.  We summarize some results of \cite{BazzoniCrivei13,HenrardvanRoosmalen19b,HenrardvanRoosmalen19a} in a convenient form. Throughout the paper, all categories are assumed essentially small.  Furthermore, all categories and functors will be additive.

\subsection{The category of finitely presented functors} Let $\EE$ be an additive category.  We denote by $\Mod(\EE)$ the category $\Fun(\EE^\circ, \Ab)$ of contravariant additive functors $\EE \to \Ab$.  We write $\Yoneda\colon \EE \to \Mod \EE$ for the contravariant Yoneda functor $E \mapsto \Yoneda(E) = \Hom(-,E).$

We say that $M$ is \emph{finitely presented} if $M \cong \coker \Yoneda(f)$ where $f$ is a morphism in $\EE$.  We write $\smod(\EE)$ for the category of finitely presented objects in $\Mod(\EE)$.  If $\EE$ has weak kernels, then $\mod(\EE)$ is abelian.  The category of finitely presented objects satisfies the following universal property (see \cite[Universal Property 2.1]{Krause98}).

\begin{theorem}\label{theorem:UniversalFreyd} Let $F\colon \EE \to \AA$ be a functor between additive categories.  Assume that $\AA$ has cokernels.  There exists, up to a natural equivalence, a unique right exact functor $\overline{F}\colon \mod(\EE) \to \AA$ such that $F = \overline{F} \circ \Yoneda$.
\end{theorem}

\begin{proposition}\label{proposition:UniversalLiftExact}
Let $\AA$ be an abelian category and let $\EE$ be an additive category with kernels.  If a functor $F\colon \EE \to \AA$ commutes with kernels, then the lift $\overline{F}\colon \mod(\EE) \to \AA$ is exact.
\end{proposition}

\begin{proof}
Following \cite[Lemma~2.5]{Krause98}, it suffices to show the following property: for each exact sequence $X \to Y \to Z$ of projective objects in $\mod(\EE)$, the corresponding sequence $F(X) \to F(Y) \to F(Z)$ is exact.

As $\EE$ has kernels (and hence is idempotent complete), every projective in $\mod(\EE)$ is of the form $\Yoneda(E)$, for some $E \in \EE$.  Hence, the sequence $X \to Y \to Z$ is isomorphic to a sequence of the form $\Yoneda(A) \to \Yoneda(B) \to \Yoneda(C)$ (for some sequence $A \xrightarrow{f} B \xrightarrow{g} C$).  Saying that $\Yoneda(A) \to \Yoneda(B) \to \Yoneda(C)$ is exact is equivalent to $\im \Yoneda(f) = \ker \Yoneda(g)$.  As $\EE$ has kernels, we find $\ker \Yoneda(g) \cong \Yoneda(\ker g)$.  In particular, $\ker \Yoneda(g)$ is projective.  Hence, $\Yoneda(A) \to \Yoneda(\ker g)$ is a split epimorphism.

We find a split epimorphism $A \to \ker g$ and hence a split epimorphism $F(A) \to F(\ker g)$.  As $F$ commutes with kernels, we also find an exact sequence $0 \to F(\ker g) \to F(B) \xrightarrow{F(g)} F(C)$.  Combining these sequences, we see that $F(A) \to F(B) \to F(C)$ is exact, as required.
\end{proof}

For any functor $F\colon\EE \to \FF$ between additive categories, there is an obvious restriction functor
\[
- \circ F\colon\Mod(\FF)\to \Mod(\EE)
\]
which sends an $M \in \Mod(\FF)$ to $M \circ F \in \Mod(\EE)$. The restriction functor has a left adjoint 
\[
-\otimes_{\EE} \FF\colon \Mod(\EE)\to \Mod(\FF)
\]
which is the (essentially unique) cocontinuous functor which sends the projective generators $\Yoneda(E)$ of $\Mod(\EE)$ to $\Yoneda(F(E))$.  Note that $-\otimes_\EE \FF\colon \Mod(\EE) \to \Mod(\FF)$ restricts to a functor $\mod(\EE) \to \mod(\FF)$.

Let $\EE$ be an additive category with kernels (in particular, $\mod \EE$ is abelian).  We write $\modone(\EE)$ for the subcategory of $\mod(\EE)$ consisting of all objects of global dimension at most one.  The following description of the objects of $\modone(\EE)$ is standard.

\begin{proposition}\label{proposition:GlobalDimensionAtMostOne}
Let $\EE$ be an additive category with kernels.  The following are equivalent for an object $M \in \mod(\EE)$:
\begin{enumerate}
	\item $M$ has projective dimension at most one,
	\item there is a monomorphism $f$ in $\EE$ such that $M \cong \coker \Yoneda(f)$,
	\item every morphism $f$ in $\EE$ for which $M \cong \coker \Yoneda(f)$ factors as $f = m \circ p$ where $p$ is a retraction and $m$ is a monomorphism.
\end{enumerate}
\end{proposition}

\begin{proof}
Straightforward adaptation of \cite[Proposition~1.1]{AuslanderReiten75}.
\end{proof}

The following proposition (see \cite[Proposition~3.4]{HenrardKvammevanRoosmalen20}) will be used multiple times throughout the text.

\begin{proposition}\label{Lemma:TheFamousDiagramChase}
	Let $\EE$ be an additive category and write $\Upsilon\colon \EE\to \Mod(\EE)$ for the Yoneda embedding. Consider a commutative diagram
	\[\xymatrix{
		A\ar[r]^{\beta}\ar[d]^f & C\ar[d]^{g}\\
		B\ar[r]^{\alpha} & D
	}\] in $\EE$ such that $g$ admits a kernel $k\colon \ker(g)\to C$ and such that the cospan $B\stackrel{\alpha}{\rightarrow}D\stackrel{g}{\leftarrow}C$ admits a pullback $E$. Write $F=\coker(\Upsilon(f))$, $G=\coker(\Upsilon(g))$ and $\eta\colon F\to G$ for the induced map. Consider the commutative diagram
	\begin{equation}\label{Thefamousdiagram}\xymatrix{
		&&\ker(g')\ar[d]^{k'}\ar@{=}[r] & \ker(g)\ar[d]^{k} &&\\
		\ker(g)\oplus A\ar[r]^-{\begin{psmallmatrix}0&1\end{psmallmatrix}}\ar[d]^{\begin{psmallmatrix}k'&\beta''\end{psmallmatrix}} & A\ar[r]^{\beta''}\ar[d]^f & E\ar[r]^{\beta'}\ar[d]^{g'} & C\ar[r]^-{\begin{psmallmatrix}1\\0\end{psmallmatrix}}\ar[d]^{g} & C\oplus B\ar[d]^{\begin{psmallmatrix}g&\alpha\end{psmallmatrix}}\\
		E\ar[r]^{g'} & B\ar@{=}[r] & B\ar[r]^{\alpha} & D\ar@{=}[r] & D
	}\end{equation} where $ECBD$ is a pullback square and $\beta=\beta'\beta''$. Applying $\Upsilon$ and taking the cokernel of the vertical maps induces the epi-mono factorization 
	\[\xymatrix{
		\ker(\eta)\ar@{>->}[r] & F\ar@{->>}[r] & \im(\eta)\ar@{>->}[r] & G\ar@{->>}[r] & \coker(\eta)
	}\] of $\eta\colon F \to G$ in $\Mod(\EE)$.
\end{proposition}

It will be convenient to state the following corollary.

\begin{corollary}\label{proposition:TorsionPart}
Let $\EE$ be an additive category.  Let $f\colon A \to B$ be any morphism in $\EE$ with factorization $f = m \circ p$ where $m$ is a monomorphism.  There is an associated exact sequence in $\mod(\EE)$
\[0 \to \coker \Yoneda(p) \to \coker \Yoneda(f) \to \coker \Yoneda(m) \to 0.\]
\end{corollary}

\begin{proof}
The given factorization gives the following commutative diagram in $\EE:$
\[\begin{tikzcd}
A \arrow[equal]{r} \arrow[d, "p"] & A \arrow{d}{f} \arrow[r, "p"] & E \arrow[hook]{d}{m} \\
E \arrow[hook]{r}{m} & B \arrow[equal]{r} & B
\end{tikzcd}\]
Applying the Yoneda embedding $\Yoneda\colon \EE \to \mod(\EE)$ and then taking the cokernels of the vertical maps, we find a sequence $F\stackrel{\phi}{\to} G \stackrel{\psi}{\to} H$ in $\smod(\EE)$ where $F=\coker(\Upsilon(p))$ and $H=\coker(\Upsilon(m))$.  By \Cref{Lemma:TheFamousDiagramChase} (where $g'=m$ and $\beta''=p$), we find that $0 \to F\stackrel{\phi}{\to} G \stackrel{\psi}{\to} H \to 0$ is a short exact sequence.
\end{proof}

\subsection{One-sided exact categories}\label{subsection:PrelimOneSided}
One sided-exact categories are obtained via a weakening of the axioms of a Quillen exact category \cite{BazzoniCrivei13,Generalov92,Rump11}.

\begin{definition}\label{definition:Conflation}
	A \emph{conflation category} is an additive category $\EE$ together with a chosen class of kernel-cokernel pairs, called \emph{conflations}, such that this class is closed under isomorphisms. The kernel part of a conflation is called an \emph{inflation} and the cokernel part of a conflation is called a \emph{deflation}.  	We depict inflations by the symbol $\inflation$ and deflations by $\deflation$. Moreover, we depict monomorphisms by $\hookrightarrow$.  A morphism $X \to Y$ is called \emph{admissible} if it admits a deflation-inflation factorization $X \deflation Z \inflation Y.$
	
	An additive functor $F\colon \CC\to \DD$ between conflation categories is called \emph{(conflation-)exact} if conflations are mapped to conflations.  We say that $F$ is \emph{left (conflation-)exact }if any conflation $A \stackrel{f}{\inflation} B \stackrel{g}{\deflation} C$ is mapped to a sequence $F(A) \stackrel{F(f)}{\inflation} F(B) \stackrel{F(g)}{\rightarrow} F(C)$ where $F(g)$ is admissible and $F(f) = \ker(F(g)).$
\end{definition}

\begin{definition}\label{Definition:DeflationExact}
  A \emph{deflation-exact category} $\EE$ is a conflation category satisfying the following axioms:
	\begin{enumerate}[label=\textbf{R\arabic*},start=0]
		\item\label{R0} For each $X\in \EE$, the map $X\to 0$ is a deflation.
		\item\label{R1} The composition of two deflations is a deflation.
		\item\label{R2} The pullback of a deflation along any morphism exists and deflations are stable under pullbacks.
	\end{enumerate}
	Dually, an \emph{inflation-exact category} is a conflation category $\EE$ satisfying the following axioms:
		\begin{enumerate}[label=\textbf{L\arabic*},start=0]
		\item\label{L0} For each $X\in \EE$, the map $0\to X$ is an inflation.
		\item\label{L1} The composition of two inflations is an inflation.
		\item\label{L2} The pushout of an inflation along any morphism exists and inflations are stable under pushouts.
	\end{enumerate}
\end{definition}

\begin{definition}\label{Definition:StronglyDeflationExact}
	Let $\EE$ be a conflation category. In addition to the properties listed in \Cref{Definition:DeflationExact}, we will also consider the following axioms:
	\begin{enumerate}[align=left]
		\myitem{\textbf{R3}}\label{R3} \hspace{0.175cm}If $i\colon A\rightarrow B$ and $p\colon B\rightarrow C$ are morphisms in $\EE$ such that $p$ has a kernel and $p\circ i$ is a deflation, then $p$ is a deflation.
		\myitem{\textbf{R3}$^+$}\label{R3+} \hspace{0.175cm}If $i\colon A\rightarrow B$ and $p\colon B\rightarrow C$ are morphisms in $\EE$ such that $p\circ i$ is a deflation, then $p$ is a deflation.
	\end{enumerate}
	The axioms \textbf{L3} and \textbf{L3}$^+$ are defined dually. A deflation-exact category satisfying \ref{R3} is called \emph{strongly deflation-exact}. Dually, an inflation-exact category satisfying axiom \textbf{L3} is called a \emph{strongly inflation-exact category}.
\end{definition}

\begin{remark}\label{Remark:BasicDefinitions}
	\begin{enumerate}
		\item Inflation-exact and deflation-exact categories are called left or right exact categories in the literature. However, as the use of left and right is not consistent, we prefer to use the above terminology to avoid possible confusion.
		\item It follows from \cite[Proposition~2.5]{KaledinLowen15} that a deflation-exact category is an additive single $\Lambda$-topology, where $\Lambda$ is the class of deflations.
		\item Axioms \ref{R0} and \ref{L0} are slightly stronger than their counterparts in \cite{BazzoniCrivei13,HenrardvanRoosmalen19b,HenrardvanRoosmalen19a}. The above definition ensures that all split kernel-cokernel pairs are conflations in a one-sided exact category.
		\item An exact category in the sense of Quillen (see \cite{Quillen73}) is a conflation category $\EE$ satisfying axioms \ref{R0} through \ref{R3} and \ref{L0} through \textbf{L3}. In \cite[Appendix~A]{Keller90}, Keller shows that axioms \ref{R0}, \ref{R1}, \ref{R2}, and \ref{L2} suffice to define an exact category.
		\item Axioms \ref{R3} and \textbf{L3} are sometimes referred to as Quillen's \emph{obscure axioms} (see \cite{Buhler10,ThomasonTrobaugh90}).
		\item For a conflation category, the following implications hold: \ref{R3+} $\Rightarrow$ \ref{R3} . 
		\item For a deflation-exact category $\EE$, axiom \ref{R3+} is equivalent to $\EE$ being weakly idempotent complete and satisfying axiom \ref{R3} (see \cite[Proposition~7.1]{HenrardvanRoosmalen19a}).
	\end{enumerate}
\end{remark}

The following theorem highlights the importance of axioms \ref{R3} and \ref{R3+}.

\begin{theorem}[\protect{\cite[Theorems~1.1 and 1.2]{HenrardvanRoosmalen20Obscure}}]\label{Theorem:ImportanceR3andR3+}
	Let $\EE$ be a deflation-exact category.
		\begin{enumerate}
			\item The category $\EE$ satisfies axiom \ref{R3} if and only if the Nine Lemma holds.
			\item The category $\EE$ satisfies axiom \ref{R3+} if and only if the Snake Lemma holds.
		\end{enumerate}
\end{theorem}

The following observation is essentially contained in \cite{Rump01}.

\begin{proposition}\label{proposition:CriterionDeflationRegular}
Let $\EE$ be a conflation category. If every kernel-cokernel pair is a conflation and $\EE$ satisfies axiom \ref{R2}, then $\EE$ is deflation-exact.
\end{proposition}

\begin{proof}
As all kernel-cokernel pairs are conflations, $\EE$ satisfies axiom \ref{R0}.  That $\EE$ satisfies axiom \ref{R1} follows from \cite[Proposition~5.11]{Kelly69} (in the terminology of \cite{Kelly69} and assuming axiom \ref{R2}, axiom \ref{R1} is equivalent to saying that the composition of totally regular epimorphisms is again a totally regular epimorphism).
\end{proof}

\begin{definition}
We recall that a \emph{pre-abelian} category is an additive category where every morphism has a kernel and a cokernel.  We say that a pre-abelian category is \emph{deflation quasi-abelian} (or \emph{left quasi-abelian}) if the class of all kernel-cokernel pairs endow it with the structure of a deflation-exact category.  Dually, a pre-abelian category is \emph{inflation quasi-abelian} (or \emph{right quasi-abelian}) if the class of all kernel-cokernel pairs endow it with the structure of an inflation-exact category.
\end{definition}

\begin{remark}
\begin{enumerate}
  \item A quasi-abelian category is called an almost abelian category in \cite{Rump01}.
	\item For a pre-abelian category to be deflation quasi-abelian, it suffices that the pullback of a cokernel is a cokernel, see \cite[Proposition 1 and 2]{Rump01}.
	\item A pre-abelian category is idempotent complete.  A deflation quasi-abelian category satisfies axiom \ref{R3+}, see \cite[Proposition 2]{Rump01}.
\end{enumerate}
\end{remark}

%

\subsection{Derived categories of one-sided exact categories}\label{subsection:DerivedCategory}

The derived category of a one-sided exact category was studied in \cite{BazzoniCrivei13,Generalov92,HenrardvanRoosmalen19b}.  We recall the definition of the derived category, starting with the notion of an acyclic complex.

\begin{definition}\label{definition:AcylicComplex}
	Let $\EE$ be a conflation category. A complex $X^{\bullet}\in \C(\EE)$ is called \emph{acylic (or exact) in degree $n$} if $d_X^{n-1}\colon X^{n-1} \to X^n$ factors as
	\[\xymatrix{
	X^{n-1}\ar[rr]^{d_X^{n-1}}\ar@{->>}[rd]^{p^{n-1}} && X^n\\
	 & \ker(d_X^{n})\ar@{>->}[ru]^{i^{n-1}} &
	}\] where the deflation $p^{n-1}$ is the cokernel of $d_X^{n-2}$ and the inflation $i^{n-1}$ is the kernel of $d_X^{n}$.
	
	A complex $X^{\bullet}$ is called \emph{acyclic} or \emph{exact} if it is acylic in each degree.  We write $\AcC(\EE)$ for the full subcategory of $\C(\EE)$ consisting of acyclic complexes.  We write $\AcK(\EE)$ for the full subcategory of $\K(\EE)$ given by those complexes which are homotopy equivalent to an acyclic complex (thus, $\AcK(\EE)$ is the closure of $\AcC (\EE)$ under isomorphisms in $\K(\EE)$).  We simply write $\Ac(\EE)$ for either $\AcC(\EE)$ or $\AcK(\EE)$ if there is no confusion.  The bounded versions are defined by $\mathbf{Ac}_{\mathbf{C}}^*(\EE)=\AcC (\EE)\cap \C^*(\EE)$ and $\mathbf{Ac}_{\mathbf{K}}^*=\AcK (\EE)\cap \K^*(\EE)$.
\end{definition}

The subcategory $\AcC(\EE)$ of $\K(\EE)$ is not \emph{replete}, i.e.~it is not closed under isomorphisms in $\K(\EE)$. Nonetheless, it is a triangulated subcategory of $\K(\EE)$.

\begin{lemma}[\protect{\cite[Lemma~7.2]{BazzoniCrivei13}}]\label{lemma:AcyclicIsTriangulated}
	For each map $f\colon X^{\bullet}\to Y^{\bullet}$ in $\AcC(\EE)$, we have that $\cone(f^{\bullet})\in \AcC(\EE)$.  In particular, the category $\AcC(\EE)$ is a triangulated subcategoy of $\K(\EE)$ which is not necessarily closed under isomorphisms.
\end{lemma}

Analogously to exact categories, one can define the derived category $\D(\EE)$ as the Verdier localization $\K(\EE)/\left\langle\Ac(\EE)\right\rangle_{\text{thick}}$ of the bounded homotopy category by the thick closure of the triangulated subcategory of acyclic complexes. The bounded versions are defined analogously. The following theorem summarizes some useful properties of the derived category.

\begin{theorem}[\protect{\cite[Theorem~1.2]{HenrardvanRoosmalen19b}}]\label{Theorem:BasicPropertiesDerivedCategory}
	Let $\EE$ be a deflation-exact category.
	\begin{enumerate}
		\item The natural embedding $i\colon \EE\to \D(\EE)$ is fully faithful.
		\item For all $X, Y\in \EE$ and $n>0$, $\Hom_{\D(\EE)}(\Sigma^n iX, iY)=0$.
		\item Every conflation $X\inflation Y\deflation Z$ in $\EE$ maps to a triangle $iX\to iY\to iZ\to \Sigma iX$ in $\D(\EE)$.
	\end{enumerate}
\end{theorem}

With regard to the derived category, axioms \ref{R3} and \ref{R3+} have useful interpretations.

\begin{proposition}[\protect{\cite[Propositions~3.11 and 6.2]{HenrardvanRoosmalen19b} and \cite[Theorem~1.1.(4) and 1.2.(2)]{HenrardvanRoosmalen20Obscure}}]\label{Proposition:BasicPropertiesDerivedCategoryNew}
	Let $\EE$ be a deflation-exact category.
	\begin{enumerate}
		\item Axiom \ref{R3} is equivalent to: a sequence $X \to Y \to Z$ in $\EE$ is a conflation if and only if $i(X) \to i(Y) \to i(Z) \to \Sigma i(X)$ is a triangle in $\Db(\EE)$.
		\item If $\EE$ satisfies axiom \ref{R3+}, then $\AcbC(\EE)$ is a thick triangulated subcategory of $\Kb(\EE)$.
		\item If $\EE$ satisfies axiom \ref{R3} and is idempotent complete, then $\AcC(\EE)$ is a thick triangulated subcategory of $\K(\EE)$.
	\end{enumerate}
\end{proposition}

\begin{remark}
\Cref{Proposition:BasicPropertiesDerivedCategoryNew} implies that, if $\EE$ satisfies axiom \ref{R3+}, any complex which is homotopic to an acyclic complex, is acyclic itself.
\end{remark}

\begin{theorem}[Horseshoe lemma]\label{theorem:Horseshoe} Let $\EE$ be a deflation-exact category satisfying axiom \ref{R3}.  Let $X \inflation Y \deflation Z$ be a conflation in $\EE$.  If $P_X^\bullet \to X$ and $P_Z^\bullet \to Z$ are projective resolutions of $X$ and $Z$, respectively, then there is a projective resolution $P_Y^\bullet \to Y$, fitting in a commutative diagram
\[\begin{tikzcd}
\vdots \arrow[d]  & \vdots \arrow[d] & \vdots \arrow[d] \\
P_X^{-1} \arrow[d] \arrow[r] & P_Y^{-1} \arrow[d] \arrow[r] & P_Z^{-1} \arrow[d] \\
P_X^0 \arrow[d] \arrow[r] & P_Y^0 \arrow[d] \arrow[r] & P_Z^0 \arrow[d] \\
X \arrow[r, rightarrowtail] & Y \arrow[r, twoheadrightarrow] & Z.
\end{tikzcd}\]
Moreover, for each $i \leq 0$, the sequence $P_X^i \inflation P_Y^i \deflation P_Z^i$ is a split kernel-cokernel pair.
\end{theorem}

\begin{proof}
The proof of \cite[Theorem~12.8]{Buhler10} for exact categories holds verbatim in the deflation-exact setting.  One can replace the reference to \cite[Corollary~3.2 and Corollary~3.6]{Buhler10} by references to a deflation-exact version \cite[Lemma~4.2.(2) and Theorem~4.1]{HenrardvanRoosmalen20Obscure} (or the duals of \cite[Lemma~5.10 and Proposition~5.11]{BazzoniCrivei13}).%
\end{proof}

\subsection{Preresolving subcategories}

This subsection is a brief summary of \cite{HenrardvanRoosmalen20Preresolving}. We recall the following definition.

\begin{definition}
	Let $\EE$ be a deflation-exact category. A full additive subcategory $\AA\subseteq \EE$ is called \emph{preresolving} if the following two conditions are met:
	\begin{enumerate}[label=\textbf{PR\arabic*},start=1]
		\item\label{PR1} 	For every $E\in \EE$, there exists a deflation $A\deflation E$ with $A\in \EE$.
		\item\label{PR2}	The subcategory $\AA\subseteq \EE$ is \emph{deflation-closed}, i.e,.~for every conflation $X\inflation Y\deflation Z$ in $\EE$ with $Y,Z\in \AA$, we have $X\in \AA$ as well.
	\end{enumerate}
	If $\AA\subseteq \EE$ satisfies \ref{PR1}, we define the $\AA$-resolution dimension of an object $E\in \EE$, denoted by $\resdim_{\AA}(E)$, as the smallest integer $n\geq 0$ for which there exists an exact sequence 
	\[0\to A^{-n} \to A^{-n+1}\to \dots \to A^{-1} \to A^{0}\deflation E\to 0\]
	where all $A^{k}\in \AA$. If such an $n$ does not exist, we write $\resdim_{\AA}(E)=\infty$.\\
	Furthermore, we set $\resdim_{\AA}(\EE)=\sup_{E\in \EE}\resdim_{\AA}(E)$.
	
	A preresolving subcategory $\AA\subseteq \EE$ is called \emph{finitely preresolving} if for all $E\in \EE$, $\resdim_{\AA}(E)<\infty$ and is called \emph{uniformly preresolving} if $\resdim_{\AA}(\EE)<\infty$.
\end{definition}

Deflation-closed subcategories of deflation-exact categories inherit a deflation-exact structure.

\begin{proposition}[\protect{\cite[Proposition~3.6]{HenrardvanRoosmalen20Preresolving}}]\label{Proposition:DeflationClosed}
	Let $\EE$ be a deflation-exact category and let $\AA\subseteq \EE$ be a full additive subcategory. If $\AA\subseteq \EE$ is deflation-closed, then $\AA$ inherits a deflation-exact structure from $\EE$ (the conflations are precisely the conflations in $\EE$ with terms in $\AA$). Furthermore, if $\EE$ satisfies axioms \ref{R3} or \ref{R3+}, then so does $\AA$.
\end{proposition}

The following theorem is an extension of \cite[Lemma~I.4.6]{Hartshorne66}.

\begin{theorem}[\protect{\cite[Theorem~1.1]{HenrardvanRoosmalen20Preresolving}}]\label{Theorem:PreResolvingDerivedEquivalence}
	Let $\EE$ be a deflation-exact category and let $\AA\subseteq \EE$ be a full additive subcategory.
	\begin{enumerate}
		\item If $\AA$ is preresolving, the embedding $\AA\to \EE$ lifts to a triangle equivalence $\Dm(\AA)\xrightarrow{\simeq} \Dm(\EE)$.
		\item If $\AA$ is finitely preresolving, the embedding $\AA\to \EE$ lifts to a triangle equivalence $\Db(\AA)\xrightarrow{\simeq} \Db(\EE)$.
		\item If $\AA$ is uniformly preresolving, the embedding $\AA\to \EE$ lifts to a triangle equivalence $\D(\AA)\xrightarrow{\simeq} \D(\EE)$.
	\end{enumerate}
\end{theorem}

\subsection{The exact hull}\label{subsection:ExactHull}

The following is based on \cite[Section~7]{HenrardvanRoosmalen19b}; the exact hull of a one-sided exact category also appeared in \cite[Proposition~I.7.5]{Rosenberg11}.

\begin{definition}
	Let $\EE$ be a deflation-exact category. The \emph{exact hull} $\EE^{\mathsf{ex}}$ of $\EE$ is the extension closure of $i(\EE)\subseteq \Db(\EE)$.
\end{definition}

The conflation structure on $\ex{\EE}$ is given as follows (based on \cite{Dyer05}): a sequence $A \xrightarrow{f} B \xrightarrow{g} C$ in $\ex{\EE}$ is a conflation if and only if there is a triangle $A \xrightarrow{f} B \xrightarrow{g} C \to \Sigma(A)$ in $\Db(\EE)$.  With this conflation structure, the canonical embedding $j\colon \EE \rightarrow \EE^{\mathsf{ex}}$ is conflation-exact.

\begin{theorem}[\protect{\cite[Section~7]{HenrardvanRoosmalen19b}}]\label{Theorem:ExactHull}
	Let $\EE$ be a deflation-exact category.
	\begin{enumerate}
		\item The embedding $j\colon \EE\hookrightarrow \EE^{\mathsf{ex}}$ is fully faithful, and is $2$-universal among conflation-exact functors to exact categories.
		\item The embedding $j$ lifts to a triangle equivalence $\Db(\EE)\stackrel{\simeq}{\rightarrow}\Db(\EE^{\mathsf{ex}})$.
		\item For every $Z\in \EE^{\mathsf{ex}}$, there is a conflation $X\inflation Y\deflation Z$ in $\EE^{\mathsf{ex}}$ with $X,Y\in i(\EE)$.
	\end{enumerate}
	Furthermore, if $\EE$ satisfies axiom \ref{R3}, then the embedding $j$ reflects conflations.
\end{theorem}

When working with the exact hull, it is often useful to describe objects of the exact hull $\ex{\EE}$ as iterated extensions of objects in $\EE.$  For this, the following notation will be useful.

\begin{notation}\label{notation:ExactHull}
For a deflation-exact category $\EE$, we write $\EE_0$ for the full subcategory of $\Db(\EE)$ consisting of stalk complexes concentrated in degree 0.  The subcategories $\EE_n$ are recursively defined as all objects $B$ which fit into a triangle $A\to B\to C\to \Sigma A$ with $A\in \EE_{n-1}$ and $C\in \EE_0$.
\end{notation}

With this notation, we have $\EE^{\mathsf{ex}}=\bigcup_{n\geq 0}\EE_n$; this uses \cite[Lemme~1.3.10]{BeilinsonBernsteinDeligne82}.

\begin{lemma}\label{Lemma:MonosAndEpisInExactHull}
	Let $\EE$ be a deflation-exact category. Let $f\colon X\to Y$ be a morphism in $\EE$. Then $f$ is a monomorphism (resp. epimorphism) if and only if $j(f)$ is a monomorphism (resp. epimorphism).
\end{lemma}

\begin{proof}
As $j$ is fully faithful, it is clear that $j$ reflects epimorphisms and monomorphisms.  We first show that $j$ preserves monomorphisms.  For this, consider a monomorphism $f\colon X \to Y$ in $\EE$. Let $t\colon T\to X$ be a map in $\EE^{\mathsf{ex}}$ such that $f\circ t=0$ in $\EE^{\mathsf{ex}}$. As $T\in \EE^{\mathsf{ex}}$, there exists an $n$ such that $T\in \EE_n$. We proceed by induction on $n$. If $n=0$, then $t=0$ as $f$ is a monomorphism in $\EE_0$. Assume now that $n\geq 1$. By construction, there is a conflation $A \stackrel{i}{\inflation} T\stackrel{p}{\deflation} B$ in $\EE^{\mathsf{ex}}$ with $A\in \EE_{n-1}$ and $B\in \EE_0$. By the induction hypothesis, we have $t\circ i=0$. It follows that there is a unique map $u\colon B\to X$ such that $u\circ p=t$. Note that $f\circ t=f\circ u\circ p=0$ and thus $f\circ u=0$ as $p$ is a deflation (and hence an epimorphism). The induction hypothesis implies that $u=0$ and thus $t=u\circ p=0$. This shows that $j(f)$ is a monomorphism.

To show that $j$ preserves epimorphisms, consider an epimorphism $f\colon X \to Y$.  Let $t\colon Y\to T$ be a map in $\EE^{\mathsf{ex}}$ such that $t\circ f=0$ in $\EE^{\mathsf{ex}}$.  As $T \in \ex{\EE}$, there exists an $n$ such that $T\in \EE_n$.  Consider a conflation $A \stackrel{i}{\inflation} T\stackrel{p}{\deflation} B$ with $A\in \EE_{n-1}$ and $B\in \EE_0$.  Using that $f$ is an epimorphism in $\EE \simeq \EE_0$, we obtain from $p \circ t \circ f = 0$ that $p \circ t = 0$ and hence $t\colon Y \to T$ factors through $i\colon A \to Y.$  Using an induction argument as before, one can show that $t=0$.  This shows that $j(f)$ is an epimorphism in $\ex{\EE}.$
\end{proof}

If $\EE$ satisfies axiom \ref{R3+}, the embedding $j\colon \EE\to \EE^{\mathsf{ex}}$ satisfies additional properties.

\begin{theorem}[\protect{\cite[Theorem~5.7]{HenrardvanRoosmalen20Preresolving}}]\label{Theorem:DivisiveDeflationExactIsFinitelyPreResolvingInExactHull}
	Let $\EE$ be a deflation-exact category. If $\EE$ satisfies axiom \ref{R3+}, $\EE\subseteq \EE^{\mathsf{ex}}$ is a uniformly preresolving subcategory such that $\resdim_{\EE}(\EE^{\mathsf{ex}})\leq 1$. In particular, the derived equivalences of \Cref{Theorem:PreResolvingDerivedEquivalence} hold.
\end{theorem}

\subsection{\texorpdfstring{$\operatorname{t}$}{t}-Structures and their hearts}

Let $\TT$ be a triangulated category with suspension functor $\Sigma$. A \emph{$\operatorname{t}$-structure} on $\TT$ is a pair $(\TT^{\leq 0},\TT^{\geq 0})$ of full and replete (i.e.~closed under isomorphisms) subcategories satisfying the following properties:
\begin{enumerate}
	\item $\Hom_{\TT}(\TT^{\leq 0}, \Sigma^{-1}\TT^{\geq 0})=0$.
	\item If $X\in \TT^{\leq 0}$, then $\Sigma X\in \TT^{\leq 0}$. Similarly, if $Y\in \TT^{\geq 0}$, then $\Sigma^{-1}Y\in \TT^{\geq 0}$.
	\item For any $C\in \TT$, there exists a triangle $X \to C\to \Sigma^{-1}Y\to \Sigma X$ with $X\in \TT^{\leq 0}$ and $Y\in \TT^{\geq 0}$.
\end{enumerate}

We write $\TT^{\leq i} \coloneqq \Sigma^{-i} \TT^{\leq 0}$ and $\TT^{\geq i} \coloneqq \Sigma^{-i} \TT^{\geq 0}$.  Given a $\operatorname{t}$-structure $(\TT^{\leq 0},\TT^{\geq 0})$ on $\TT$, the \emph{heart} of $\TT$ is defined as the subcategory $\TT^{\heartsuit}=\TT^{\leq 0}\cap \TT^{\geq 0}$. The following proposition is standard (see \cite{BeilinsonBernsteinDeligne82}).

\begin{proposition}
	Given a $\operatorname{t}$-structure $(\TT^{\leq 0},\TT^{\geq 0})$ on a triangulated category $\TT$. The categories $\TT^{\leq 0}$ and $\TT^{\geq 0}$ are closed under extensions and the heart $\TT^{\heartsuit}$ is an abelian subcategory.  Moreover, a sequence $0 \to X \xrightarrow{f} Y \xrightarrow{g} Z \to 0$ is a short exact sequence in $\TT^{\heartsuit}$ if and only if there is a triangle $X \xrightarrow{f} Y \xrightarrow{g} Z \to \Sigma(X)$ in $\TT.$
\end{proposition}

\section{The left \texorpdfstring{$\operatorname{t}$}{t}-structure and left heart}\label{section:LeftHeart}

The left $\operatorname{t}$-structure and the left heart were introduced in \cite{Schneiders99} for quasi-abelian categories.  In this section, we show that these constructions and many of the properties lift to a weaker setting, namely that of a deflation-exact structure on an additive category $\EE$ with kernels.  We follow the same outline as \cite[Section 1.2]{Schneiders99}.

We will assume that the deflation-exact structure is strong (that is, satisfies axiom \ref{R3}).  This is a purely technical condition: as $\EE$ has kernels, one can take the closure of $\EE$ under the axiom \ref{R3} without changing the derived category (see \cite{HenrardvanRoosmalen20Obscure}).

\subsection{\texorpdfstring{A $\operatorname{t}$-structure on $\K(\EE)$}{A t-structure on the homotopy category}}
As in \cite{Schneiders99}, we start by considering a $\operatorname{t}$-structure on the homotopy category $\K(\EE)$.  In this subsection, we only use the additive structure on $\EE$.  We will use the following truncation functors.

\begin{definition}\label{Definition:Truncations}
Let $C^{\bullet}$ be a complex in $\EE$.  As $\EE$ has kernels, every differential $d^{n-1}_C\colon C^{n-1} \to C^n$ factors as 
	\[C^{n-1}\xrightarrow{p^{n-1}} \ker(d^n)\xrightarrow{i^{n-1}} C^n\]
	where $i^{n-1}$ is the kernel of $d^n_C$. The \emph{canonical truncation} $\tau^{\leq n}C^{\bullet}$ is a complex together with a morphism $\tau^{\leq n}C^{\bullet} \to C^\bullet$ given by:
	\begin{equation*}
		\begin{tikzcd}
			\tau^{\leq n}C^{\bullet}\arrow{d}{} & &\cdots\arrow{r} &C^{n-3}\arrow[equals]{d}\arrow{r}&C^{n-2}\arrow[equals]{d}\arrow{r} &C^{n-1}\arrow[equals]{d}\arrow{r}{p^{n-1}} &\ker(d^n_C)\arrow{d}{i^{n-1}}\arrow{r}&0\arrow{d}\arrow{r}&\cdots\\
			C^{\bullet} &  &\cdots\arrow{r} &C^{n-3}\arrow{r}&C^{n-2}\arrow{r} &C^{n-1}\arrow{r} & C^n\arrow{r}&C^{n+1}\arrow{r}&\cdots
		\end{tikzcd}
	\end{equation*}
	and the \emph{canonical truncation} $C^\bullet \to \tau^{\geq n+1}C^{\bullet}$ is similarly defined by:
	\begin{equation*}
		\begin{tikzcd}
			C^{\bullet}\arrow{d}{} & &\cdots\arrow{r} &C^{n-3}\arrow{d}\arrow{r}&C^{n-2}\arrow{d}\arrow{r} &C^{n-1}\arrow{d}{p^{n-1}}\arrow{r} &C^{n}\arrow[equals]{d}\arrow{r}&C^{n+1}\arrow[equals]{d}\arrow{r}&\cdots\\
			\tau^{\geq n+1}C^{\bullet} &  &\cdots\arrow{r} &0\arrow{r}&0\arrow{r} &\ker(d^n_C)\arrow{r}{i^{n-1}} & C^n\arrow{r}&C^{n+1}\arrow{r}&\cdots
		\end{tikzcd}
	\end{equation*}
\end{definition}

The following proposition is \cite[Proposition~3.13]{HenrardvanRoosmalen19b}.

\begin{proposition}\label{Proposition:AcylicComplexAsExtensionOfTruncations}
	Let $\EE$ be an additive category with kernels.  Let $C^{\bullet}\in \C^*(\EE)$ where $*\in \{-,+,b,\varnothing\}$.  For each $n \in \bZ$, the following triangle is a distinguished triangle in $\K^*(\EE)$:
\[\tau^{\leq n}C^{\bullet}\rightarrow C^{\bullet} \rightarrow \tau^{\geq n+1}C^{\bullet}\rightarrow \Sigma \left( \tau^{\leq n}C^{\bullet}\right)\]
In other words, $C^{\bullet}$ is an extension of the canonical truncation $\tau^{\geq n+1}C^{\bullet}$ by $\tau^{\leq n}C^{\bullet}$ in $\K^*(\EE)$.
\end{proposition}

We can now consider the $\operatorname{t}$-structure $(\K^{\leq 0}(\EE), \K^{\geq 0}(\EE))$ on the homotopy category $\K(\EE)$ where
\begin{align*}
\K^{\leq 0}(\EE) &= \{X^\bullet \in \K(\EE) \mid \tau^{\geq 1} X^\bullet \cong 0 \}, \\
\K^{\geq 0}(\EE) &= \{X^\bullet \in \K(\EE) \mid \tau^{\leq -1} X^\bullet \cong 0 \}.
\end{align*}
In other words, $\K^{\leq 0}(\EE)$ is given by those complexes $X^\bullet$ such that $\ker d^{i-1}_X \to X^{i-1} \to \ker d^{i}_X$ is a split kernel-cokernel pair for all $i \geq 1.$  Likewise, $X^\bullet \in \K^{\leq 0}(\EE)$ if and only if $\ker d^{i-1}_X \to X^{i-1} \to \ker d^{i}_X$ is a split kernel-cokernel pair for all $i \leq -1$.

\subsection{Induced \texorpdfstring{$\operatorname{t}$}{t}-structures on the derived category}

In this subsection, we show that the above $\operatorname{t}$-structure on $\K(\EE)$ induces a $\operatorname{t}$-structure on $\D(\EE) = \K(\EE) / \Ac(\EE).$  For this, we use the following statement from \cite[Lemma~1.2.17]{Schneiders99}

\begin{proposition}\label{proposition:ConditionSchneiders}
Let $(\TT^{\leq 0}, \TT^{\geq 0})$ be a $\operatorname{t}$-structure on a triangulated category $\TT$.  Let $\NN \subseteq \TT$ be a thick subcategory and write $Q\colon \TT \to \TT / \NN$ for the corresponding quotient.  The pair $(Q(\TT^{\leq 0}), Q(\TT^{\geq 0}))$ is a $\operatorname{t}$-structure on $\TT / \NN$ if and only if for any triangle $X_1 \to X_0 \to N \to \Sigma X_1$ with $X_1 \in \TT^{\geq 1}$, $X_0 \in \TT^{\leq 0}$, and $N \in \NN$, we have $X_1, X_0 \in \NN$
\end{proposition}

The following proposition is a convenient strengthening of \cite[Lemma~7.2]{BazzoniCrivei13}.

\begin{proposition}\label{proposition:StrengtheningBazzoniCrivei}
Let $f\colon X^{\bullet}\to Y^{\bullet}$ in $\C(\EE)$.  If $X^\bullet$ is acylic in degree $n$ and $Y^\bullet$ is acyclic in degrees $n-1$ and $n$, then $\cone(f^\bullet)$ is acyclic in degree $n-1$.
\end{proposition}

\begin{proof}
The proof follows that of \cite[Lemma~7.2]{BazzoniCrivei13} closely.  As $X^\bullet$ is acyclic in degree $n$, we know that $i_X^{n-1}\colon \ker d^n_X \inflation X^n$ is an inflation.  By \cite[Proposition~4.4]{HenrardvanRoosmalen20Obscure}, this implies that $d^{n}\colon X^n \to X^{n+1}$ has a deflation-mono factorization: $\begin{tikzcd}[cramped, sep = scriptsize] X^n \arrow[r, twoheadrightarrow, "p_X^n"] & \coim d^{n+1}_X \arrow[hook]{r}{m}&X^{n+1}\end{tikzcd}$.  We find the following commutative diagram
\[\begin{tikzcd}[column sep=scriptsize]
\arrow[rr, dotted, no head]  && X^{n-2} \arrow[dr] \arrow[rr, "d_X^{n-2}"] \arrow[ddd, "f^{n-2}"] && X^{n-1} \arrow[dr, twoheadrightarrow, "p_X^{n-1}"] \arrow[rr, "d_X^{n-1}"] \arrow[ddd,  "f^{n-1}"] && X^n \arrow[dr, twoheadrightarrow, "p_X^{n}"] \arrow[rr, "d_X^{n}"] \arrow[ddd,  "f^{n}"] && X^{n+1} \arrow[rr, dotted, no head] \arrow[ddd,  "f^{n+1}"] && {} \\
&&&\ker d_X^{n-1} \ar[d, dashrightarrow, "g^{n-1}"] \arrow[ru, rightarrowtail, "i_X^{n-2}"] &&\ker d^n_X\ar[d, dashrightarrow, "g^{n}"] \arrow[ru, rightarrowtail, "i_X^{n-1}"] && \coim d^{n+1}_X\ar[d, dashrightarrow, "g^{n}"] \arrow[ru, hookrightarrow, "m"] \\
&&&\ker d_Y^{n-1} \arrow[rd, rightarrowtail, "i_Y^{n-2}"'] &&\ker d^n_Y\arrow[rd, rightarrowtail, "i_Y^{n-1}"'] && \coim d^{n+1}_Y\arrow[rd] \\
\arrow[rr, dotted, no head] && Y^{n-2} \arrow[ur, twoheadrightarrow, "p_Y^{n-2}"'] \arrow[rr, "d_Y^{n-2}"'] && Y^{n-1} \arrow[ur, twoheadrightarrow, "p_Y^{n-1}"'] \arrow[rr, "d_Y^{n-1}"'] && Y^n \arrow[ur, twoheadrightarrow, "p_Y^{n}"'] \arrow[rr, "d_Y^{n}"'] && Y^{n+1} \arrow[rr, dotted, no head] && {}
\end{tikzcd}\]
where the morphisms $g^{n-1}$, $g^n$, and $g^{n+1}$ are uniquely determined.  We can apply \cite[Proposition~3.9]{HenrardvanRoosmalen19a} (or the dual of \cite[Proposition~5.2]{BazzoniCrivei13}) to the maps $(g^{n-1}, f^{n-1}, g^n)$ and $(g^{n}, f^{n}, g^{n+1})$ between conflations to obtain the following commutative diagram (the squares marked with BC are bicartesian squares):
\[\begin{tikzcd}[column sep=scriptsize]
\arrow[rr, dotted, no head]  && X^{n-2} \arrow[dr] \arrow[rr, "d_X^{n-2}"] \arrow[dd, "f_1^{n-2}"] && X^{n-1} \arrow[dr, twoheadrightarrow, "p_X^{n-1}"] \arrow[rr, "d_X^{n-1}"] \arrow[dd,  "f_1^{n-1}"] && X^n \arrow[dr, twoheadrightarrow, "p_X^{n}"] \arrow[rr, "d_X^{n}"] \arrow[dd,  "f_1^{n}"] && X^{n+1} \arrow[r, dotted, no head] \arrow[dddd,  "f^{n+1}"] & {} \\
&&&\ker d_X^{n-1} \ar[dd, "g^{n-1}"] \arrow[ru, rightarrowtail, "i_X^{n-2}"] \arrow[dr, phantom, "\text{BC}" description] &&\ker d^n_X\ar[dd, "g^{n}"] \arrow[ru, rightarrowtail, "i_X^{n-1}"] \arrow[dr, phantom, "\text{BC}" description] && \coim d^{n+1}_X\ar[dd, "g^{n}"] \arrow[ru, hookrightarrow, "m"] \\
&& C^{n-2} \arrow[dd, "f_2^{n-2}"] \arrow[ru, twoheadrightarrow, "q^{n-2}"'] \arrow[dr, phantom,"\text{BC}" description] && C^{n-1} \arrow[dd, "f_2^{n-1}"] \arrow[ru, twoheadrightarrow, "q^{n-1}"'] \arrow[dr, phantom,"\text{BC}" description]&& C^{n} \arrow[dd, "f_2^{n}"] \arrow[ru, twoheadrightarrow, "q^{n}"'] \arrow[dr, phantom,"\text{BC}" description]\\
&&&\ker d_Y^{n-1} \arrow[rd, rightarrowtail, "i_Y^{n-2}"'] \arrow[ru, rightarrowtail, "j^{n-2}"'] &&\ker d^n_Y\arrow[rd, rightarrowtail, "i_Y^{n-1}"'] \arrow[ru, rightarrowtail, "j^{n-1}"'] && \coim d^{n+1}_Y\arrow[rd] \\
\arrow[rr, dotted, no head] && Y^{n-2} \arrow[ur, twoheadrightarrow, "p_Y^{n-2}"'] \arrow[rr, "d_Y^{n-2}"'] && Y^{n-1} \arrow[ur, twoheadrightarrow, "p_Y^{n-1}"'] \arrow[rr, "d_Y^{n-1}"'] && Y^n \arrow[ur, twoheadrightarrow, "p_Y^{n}"'] \arrow[rr, "d_Y^{n}"'] && Y^{n+1} \arrow[r, dotted, no head] & {}
\end{tikzcd}\]
Additionally, we have added the pullback of the cospan $\begin{tikzcd} Y^{n-2} \arrow[r, twoheadrightarrow, "p_Y^{n-2}"] & \ker d_Y^{n-1} & \ker d_X^{n-1} \arrow[l, "g^{n-1}"'] \end{tikzcd}$.  By the dual of \cite[Propositions~5.4 and~5.5]{BazzoniCrivei13}, we have the conflations
\[\begin{tikzcd}[column sep = 7em]
C^{n-2} \arrow[r, rightarrowtail, "\begin{pmatrix} i^{n-2}_X q^{n-2} \\ -f_2^{n-2} \end{pmatrix}"] & X^{n-1} \oplus Y^{n-2} \arrow[r, twoheadrightarrow, "\left({ f_1^{n-1} \enspace j^{n-2} p_Y^{n-2} }\right)" ] & C^{n-1}
\end{tikzcd}\]
and
\[\begin{tikzcd}[column sep = 7em]
C^{n-1} \arrow[r, rightarrowtail, "\begin{pmatrix} i^{n-1}_X q^{n-1} \\ -f_2^{n-1} \end{pmatrix}"] & X^{n} \oplus Y^{n-1} \arrow[r, twoheadrightarrow, "\left( {f_1^{n} \enspace j^{n-1} p_Y^{n-1}} \right)" ] & C^{n}
\end{tikzcd}\]
Hence, to show that $\cone(f^\bullet)$ is acyclic in degree $n-1$, we only need to show that $C^{n-1} \inflation X^n \oplus Y^{n-1}$ is the kernel of $X^n \oplus Y^{n-1} \to X^{n+1} \oplus Y^{n}$.  For this, it suffices to show that $\begin{psmallmatrix} m q^n \\ f_2^n \end{psmallmatrix}\colon C^n \to X^{n+1} \oplus Y^n$ is a monomorphism.  To verify this last claim, consider the following commutative diagram
\[\begin{tikzcd}
\ker d_Y^n \arrow[r, rightarrowtail, "j^{n-1}"] \arrow[d, equal] & C^n \arrow[r, twoheadrightarrow, "q^n"] \arrow[d, "f_2^n"] & \coim d^{n+1}_X \arrow[hook]{r}{m} & X^{n+1} \\
\ker d_Y^n \arrow[r, rightarrowtail, "i_Y^{n-1}"] & Y^{n}.
\end{tikzcd}\]
Let $t\colon T \to C^n$ be a morphism for which $\begin{psmallmatrix} m q^n \\ f_2^n \end{psmallmatrix}\circ t = 0$.  As $m$ is a monomorphism, it follows from $m q^n t = 0$ that $t = j^{n-1} \circ t'$.  As $i_Y^{n-1}$ is a monomorphism, it follows from $i_Y^{n-1}\circ t' = f_2^n \circ j^{n-1} \circ t' = 0$ that $t' = 0$.  Hence, $t = 0$ and we find that $\begin{psmallmatrix} m q^n \\ f_2^n \end{psmallmatrix}\colon C^n \to X^{n+1} \oplus Y^n$ is a monomorphism.
\end{proof}

\begin{proposition}\label{Proposition:tStructureOnHomotopyCategory}
Let $\EE$ be a strongly deflation-exact category with kernels.  There is a $\operatorname{t}$-structure on $\D(\EE)$ given by
\begin{align*}
\D^{\leq 0}(\EE) &= \{X^\bullet \in \D(\EE) \mid \tau^{\geq 1}X^{\bullet} \cong 0\}, \\
\D^{\geq 0}(\EE) &= \{X^\bullet \in \D(\EE) \mid \tau^{\leq -1}X^{\bullet} \cong 0\}.
\end{align*}
\end{proposition}

\begin{proof}
Let $Q\colon \K(\EE) \to \K(\EE) / \Ac(\EE)$ be the Verdier localization.  We see that $\D^{\leq 0}(\EE) = Q(\K^{\leq 0}(\EE))$ and $\D^{\geq 0}(\EE) = Q(\K^{\geq 0}(\EE))$.  Hence, to prove this proposition, it suffices to show that the conditions of \Cref{proposition:ConditionSchneiders} are satisfied for $\TT = \K(\EE)$ and $\NN = \Ac(\EE).$  As $\EE$ has kernels and satisfies axiom \ref{R3}, we know by \Cref{Proposition:BasicPropertiesDerivedCategoryNew} that $\Ac(\EE)$ is a thick subcategory of $\K(\EE)$.  The rest follows directly from \Cref{proposition:StrengtheningBazzoniCrivei}.
\end{proof}

\begin{definition}\label{Definition:LefttStructure}\label{Definition:CohomologyFunctors}
Let $\EE$ be a strongly deflation-exact category with kernels.  We call the $\operatorname{t}$-structure $(\D^{\leq 0}(\EE),\D^{\geq 0}(\EE))$ from \Cref{Proposition:tStructureOnHomotopyCategory} the \emph{left $\operatorname{t}$-structure}.  We write $\mathcal{LH}(\EE)$ for the heart $\D^\heartsuit(\EE) = \D^{\leq 0}(\EE) \cap \D^{\geq 0}(\EE)$ and $\LH^i \coloneqq \tau^{\leq 0} \circ \tau^{\geq 0} \circ \Sigma^i\colon \D(\EE) \to \mathcal{LH}(\EE)$ for the corresponding cohomology functors.
\end{definition}

\begin{remark}
As an alternative description, we have $X^\bullet \in \D^{\leq 0}(\EE)$ if and only if, for all $i \geq 1$, the sequence $\ker d^{i-1}_X \to X^{i-1} \to \ker d^{i}_X$ is a conflation.  Likewise, $X^\bullet \in \D^{\geq 0}(\EE)$ if and only if, for all $i \leq -1$, the sequence $\ker d^{i-1}_X \to X^{i-1} \to \ker d^{i}_X$ is a conflation.
\end{remark}

\subsection{Embedding into the left heart} We now turn our attention to the heart of the left $\operatorname{t}$-structure (see \Cref{Definition:LefttStructure}) on $\D(\EE).$


\begin{proposition}\label{proposition:DescriptionCohomology}
	Let $C^{\bullet}\in \D(\EE)$. The $n^{\text th}$ cohomology $\LH^n(C^{\bullet})$ is the three-term complex
	\[\dots \to 0\to \ker(d^{n-1}) \hookrightarrow C^{n-1} \to \ker(d^n)\to 0 \to \dots\]
	with $\ker(d^n)$ in degree $0$.
\end{proposition}

\begin{proof}
This follows directly from $\LH^i = \tau^{\leq 0} \circ \tau^{\geq 0} \circ \Sigma^i$.
\end{proof}

Via the embedding $i\colon \EE \to \DD(\EE)$, the category $\EE$ can be considered as a subcategory of the left heart $\mathcal{LH}(\EE)$.  We write $\phi\colon \EE \to \mathcal{LH}(\EE)$ for the corresponding embedding.

\begin{proposition}\label{proposition:EmbeddingToLeftHeartCommutesWithKernels}
The embedding $\phi\colon \EE \to \mathcal{LH}(\EE)$ commutes with kernels.
\end{proposition}

\begin{proof}
Let $f\colon X \to Y$ be any map in $\EE.$  Let $C$ be the cone of the corresponding morphism $i(f)\colon i(X) \to i(Y)$ in $\D(\EE)$.  Applying the cohomology functors $\LH^\bullet$, we find the following exact sequence in $\mathcal{\LH}(\EE):$
\[0 \to \LH^{-1}(C) \to \LH^0(iX) \to \LH^0(iY) \to \LH^0(C) \to 0\text{.}\]
As $C$ is the complex $\cdots \to 0 \to X \xrightarrow{f} Y \to 0 \to \cdots$ (with $Y$ in degree 0), we find that $\LH^{-1}(C) = \tau^{\leq 0} \circ \tau^{\geq 0} \circ \Sigma^{-1} (C) = i(\ker f).$
\end{proof}

\begin{corollary}\label{corollary:GeneralHeart}
An object $C^\bullet \in \D(\EE)$ belongs to the heart $\mathcal{LH}(\EE)$ of the left $\operatorname{t}$-structure if and only if it is isomorphic to a complex of the form
\[\cdots \to 0 \to \ker f \stackrel{k}{\hookrightarrow} X \xrightarrow{f} Y \to 0 \to \cdots\]
with $Y$ in degree 0.  For such an object $C^\bullet$, there is an exact sequence $0 \to \phi(\ker f) \xrightarrow{\phi(k)} \phi(X) \xrightarrow{\phi(f)} \phi(Y) \to C^\bullet \to 0$ in $\mathcal{LH}(\EE).$
\end{corollary}

\begin{proof}
If $C^\bullet$ belongs to the heart, then it must be isomorphic to $\LH^0(C^\bullet)$, for some $C^\bullet \in \D(\EE)$.  By \Cref{proposition:DescriptionCohomology}, it is isomorphic to a complex of the required form.  Conversely, it is easy to see that any such complex must be in the heart.

Let $D$ be the cone of the morphism $i(f)\colon i(X) \to i(Y)$ in $\D(\EE)$.  As in the previous proof, we find an exact sequence
\[0 \to \LH^{-1}(D) \to \LH^0(iX) \to \LH^0(iY) \to \LH^0(D) \to 0\]
in $\mathcal{\LH}(\EE)$.  Using the definition of the cohomology functors, we recover the exact sequence $0 \to \phi(\ker f) \xrightarrow{\phi(k)} \phi(X) \xrightarrow{\phi(f)} \phi(Y) \to C^\bullet \to 0$ in $\mathcal{LH}(\EE).$
\end{proof}

\begin{proposition}\label{Proposition:TheEmbeddingToHeartPropertiesNew}
Let $\EE$ be a deflation-exact category with kernels.  Assume that $\EE$ satisfies axiom \ref{R3}.
	\begin{enumerate}
		\item\label{Item:Proposition:TheEmbeddingToHeartProperties1x} The embedding $\phi$ is an exact and fully faithful embedding that reflects conflations.
		\item\label{Item:Proposition:TheEmbeddingToHeartProperties2x} For every object $Z\in \mathcal{LH}(\EE)$, there exists an epimorphism $Y\to Z$ with $Y\in \EE$.
		\item\label{Item:Proposition:TheEmbeddingToHeartProperties3x} The embedding $\phi$ preserves and reflects monomorphisms.
	\end{enumerate}
\end{proposition}

\begin{proof}
	\begin{enumerate}
		\item By \Cref{Theorem:BasicPropertiesDerivedCategory}, the embedding $\phi$ is fully faithful and exact.  \Cref{Proposition:BasicPropertiesDerivedCategoryNew} now shows that $\phi$ reflects exactness.
		\item This follows directly from the exact sequence in \Cref{corollary:GeneralHeart}.
		


		\item As $\phi$ is fully faithful, it reflects monomorphisms.  As $\phi$ commutes with kernels, it also preserves monomorphisms. \qedhere
	\end{enumerate}
\end{proof}

\begin{theorem}\label{Theorem:EmbeddingInHeartIsTriangleEquivalence}
Let $\EE$ be a deflation-exact category with kernels and assume that $\EE$ satisfies axiom \ref{R3}. 	The category $\EE$ is a uniformly preresolving subcategory of $\mathcal{LH}(\EE)$ with $\resdim_{\EE}(\mathcal{LH}(\EE))\leq 2$.  Consequently, the embedding lifts to a triangle equivalence $\Phi\colon\Dstar(\EE)\to \Dstar(\mathcal{LH}(\EE))$ for $*\in \{-,b,\emptyset\}$.
\end{theorem}

\begin{proof}
	By \Cref{proposition:EmbeddingToLeftHeartCommutesWithKernels}, we know that $\EE\subseteq \mathcal{LH}(\EE)$ is deflation-closed (hence, axiom \ref{PR2} is satisfied).  \Cref{corollary:GeneralHeart} implies that axiom \ref{PR1} is satisfied, as well as $\resdim_{\EE}(\mathcal{LH}(\EE))\leq 2$.  Hence, $\EE\subseteq \mathcal{LH}(\EE)$ is uniformly preresolving. That $\phi$ lifts to a derived equivalence now follows from \Cref{Theorem:PreResolvingDerivedEquivalence}.
\end{proof}

\begin{proposition}\label{proposition:CharacterizationOfHearts}
Let $\AA$ be an abelian category.  Let $\EE \subseteq \AA$ be a full subcategory satisfying condition \ref{PR1} (thus, every object in $\AA$ is a quotient of an object in $\EE$).  If for any morphism $f$ in $\EE$, we have $\ker f \in \EE$, then $\mathcal{LH}(\EE) \simeq \AA$.
\end{proposition}

\begin{proof}
As $\EE$ satisfies axiom \ref{PR1} and is closed under kernels, we find that $\EE$ is a uniformly preresolving subcategory of $\AA$ (with $\resdim_\EE(\AA) \leq 2$).  By \Cref{Proposition:DeflationClosed}, we know that $\EE$ is a strongly deflation-exact category.  It now follows from \Cref{Theorem:PreResolvingDerivedEquivalence} that the natural functor $\D(\EE) \to \D(\AA)$ is an equivalence.  In particular, every complex with terms in $\AA$ is quasi-isomorphic to a complex with terms in $\EE$.  Using the explicit form of the truncation functors on $\D(\EE)$ from \Cref{Definition:Truncations}, we see that the equivalence $\D(\EE) \to \D(\AA)$ maps the left $\operatorname{t}$-structure on $\D(\EE)$ to the standard $\operatorname{t}$-structure on $\D(\AA)$.  This now gives the equivalence $\mathcal{LH}(\EE) \simeq \AA$.
\end{proof}

\subsection{Universal properties of the left heart}  The left heart of a strongly deflation-exact category $\EE$ with kernels can be characterized via a universal property.  The first universal property we give is analogous to \cite[Proposition~12]{Rump01}.

\begin{proposition}\label{Proposition:UniversalPropertyOfLeftHeart}
	The embedding $\phi\colon \EE\to \mathcal{LH}(\EE)$ is 2-universal among conflation-exact functors to abelian categories that preserve kernels, that is to say, the functor $-\circ \phi\colon  \Funex(\mathcal{LH}(\EE), \AA) \to \Funex(\EE, \AA)$ is a fully faithful functor whose essential image consists of those functors $\EE \to \AA$ that preserve kernels.  Here, $\Funex(-,-)$ stands for the category of conflation-exact functors.
\end{proposition}

\begin{proof}
The required fully faithful functor is given by this diagram:
\[\xymatrix{\EE \ar[r] \ar[d]_{\phi} & \AA \\ {\mathcal{LH}(\EE)} \ar[ru]}\]
Let $F\colon \EE \to \AA$ be any exact functor.  Deriving this functor gives a triangle functor $\overline{F}\colon \Db(\EE) \to \Db(\AA)$, given by taking a complex $E^\bullet \in \Db(\EE)$ and applying $F$ pointwise.  As $F$ preserves kernels, it maps the heart $\mathcal{LH}(\EE)$ of the left $\operatorname{t}$-structure to the heart of the standard $\operatorname{t}$-structure on $\Db(\AA).$  That is, $\overline{F}$ restricts to an exact functor $\mathcal{LH}(\EE) \to \AA.$  Thus, for any exact $F\colon \EE \to \AA$, we find an exact functor $\mathcal{LH}(\EE) \to \Db(\EE) \to \Db(\AA) \xrightarrow{\operatorname{H}^0} \AA.$  This shows that $F$ can be lifted to $\mathcal{LH}(\EE) \to \AA$ along $\phi.$

To see that the functor $-\circ \phi\colon  \Funex(\mathcal{LH}(\EE), \AA) \to \Funex(\EE, \AA)$ is faithful, it suffices to see that every natural transformation $F \Rightarrow G$ between functors $F,G\colon \mathcal{LH}(\EE) \to \AA$ is completely determined by the restriction $F \circ \phi \Rightarrow G \circ \phi$.  This is true since, by \Cref{Theorem:EmbeddingInHeartIsTriangleEquivalence}, every object $X \in \mathcal{LH}(\EE)$ has a resolution $A \inflation B \to C \deflation X$ with $A,B, C \in \EE.$

Finally, we show that $-\circ \phi\colon  \Funex(\mathcal{LH}(\EE), \AA) \to \Funex(\EE, \AA)$ is full.  For this, we consider the arrow category $\AA^{[1]}$ of $\AA$, that is, the objects of $\AA^{[1]}$ are arrows $A \to B$ in $\AA$ and morphisms are given by commutative diagrams.  An exact functor $\EE \to \AA^{[1]}$ is given by two exact functors $F,G\colon \EE \to \AA$, together with a natural transformation $\eta\colon F \Rightarrow G$; indeed, given such $\eta\colon F \Rightarrow G$, we construct a functor $E \mapsto (\eta_E\colon F(E) \to G(E)).$  The fact that $-\circ \phi\colon  \Funex(\mathcal{LH}(\EE), \AA) \to \Funex(\EE, \AA)$ is full follows from the lifting property of $-\circ \phi\colon  \Funex(\mathcal{LH}(\EE), \AA^{[1]}) \to \Funex(\EE, \AA^{[1]}).$%
\end{proof}

\begin{remark}
	If $\EE$ is left quasi-abelian, the above proposition and \cite[Proposition~12]{Rump01} imply that $\mathcal{LH}(\EE)\simeq Q_l(\EE)$ where $Q_l(\EE)$ is the left abelian cover as defined in \cite{Rump01}. 
\end{remark}

The following proposition is a generalization of \cite[Proposition~1.2.34]{Schneiders99}.  For conflation categories $\EE, \FF$, we write $\Rex(\EE, \FF)$ for the category of right exact functors $\EE \to \FF$.

\begin{proposition}\label{proposition:RightExactUniversal}
Let $\EE$ be a strongly deflation-exact category with kernels.  For any abelian category $\AA$, the inclusion functor $\phi\colon \EE \to \mathcal{LH}(\EE)$ induces an equivalence of categories
\[\phi'\colon \Rex(\mathcal{LH}(\EE), \AA) \to \Rex(\EE, \AA).\]
Under this equivalence, conflation-exact functors correspond to conflation-exact functors.
\end{proposition}

\begin{proof}
The proof of \cite[Proposition~1.2.34]{Schneiders99} carries over to this setting.  We only note that, since $\AA$ is abelian, that a right exact functor $\EE \to \AA$ maps admissible morphisms to admissible morphisms.
\end{proof}

\begin{remark}
If $\EE$ is an exact category with kernels, then \Cref{proposition:RightExactUniversal} shows that $\phi\colon \EE \to \mathcal{LH}(\EE)$ is the right abelian envelope of $\EE$ in the sense of \cite[Definition~4.2]{BodzentaBondal20}.
\end{remark}

\subsection{The left heart as a localization}  Our final result in this section is a description of the left heart of $\EE$ as a quotient of the category $\mod(\EE)$.  To describe this quotient, we first recall the notion of an effaceable functor.

\begin{definition}\label{Definition:EffaceableFunctors}
Let $\EE$ be a deflation-exact category.  We say that an object $M \in \smod(\EE)$ is \emph{effaceable} if $M \cong \coker \Yoneda(f)$ for a deflation $f$ in $\EE$.  We write $\eff(\EE)$ for the category of effaceable functors.
\end{definition}

\begin{proposition}\label{proposition:EffaceableSerre}
Let $\EE$ be a strongly deflation-exact category.  If $\EE$ has kernels, then the category $\eff(\EE)$ is a Serre subcategory of $\mod(\EE)$.
\end{proposition}

\begin{proof}
	This is similar to \cite[Lemma~2.3]{Ogawa19}. Alternatively, using the Horseshoe lemma, one readily verifies that $\eff(\EE)$ is extension-closed in $\Mod(\EE)$.  It follows from \Cref{Lemma:TheFamousDiagramChase} that $\eff(\EE)$ is closed under subobjects and quotients in $\smod(\EE)$.
\end{proof}

We start with the embedding $\phi\colon \EE \to \mathcal{LH(\EE)}$.  By the universal property of $\mod(\EE)$, we find a natural functor $\overline{\phi}\colon \mod(\EE) \to \mathcal{LH(\EE)}.$

\begin{theorem}\label{Theorem:AlternativeDescriptionOfLeftHeart}
Let $\EE$ be strongly deflation-exact category with kernels.  The natural right exact functor $\overline{\phi}\colon \mod(\EE) \to \mathcal{LH(\EE)}$ extending $\phi\colon \EE \to \mathcal{LH(\EE)}$ induces an equivalence $\mod(\EE) / \eff(\EE) \to \mathcal{LH(\EE)}$
\end{theorem}

\begin{proof}
Write $Q\colon \smod(\EE)\to \smod(\EE)/\eff(\EE)$ for the quotient functor.  We show that $Q\circ \Upsilon\colon \EE\to \smod(\EE)/\eff(\EE)$ satisfies the universal property of $\phi\colon \EE\to \mathcal{LH}(\EE)$ given in \Cref{Proposition:UniversalPropertyOfLeftHeart}. 

Let $F\colon \EE \to \AA$ be a conflation-exact functor, preserving kernels, to an abelian category $\AA$.  We consider the lift $\overline{F}\colon \mod(\EE) \to \AA$ given by the universal property (\Cref{theorem:UniversalFreyd}).  As $F$ commutes with kernels, we know that $\overline{F}$ is exact (\Cref{proposition:UniversalLiftExact}).

Since $\overline{F}(\eff(\EE))\cong 0$, we find that $\overline{F}$ factors through $Q\colon \smod(\EE)\to\smod(\EE)/\eff(\EE)$. It remains to show that $Q\circ \Upsilon$ preserves kernels and conflations.  As both $Q$ and $\Yoneda$ commute with kernels, so does the composition $Q\circ \Upsilon$.  To see that $Q\circ \Upsilon$ preserves conflations, consider a conflation $X \inflation Y \stackrel{f}{\deflation} Z$ in $\EE.$  As the Yoneda functor is left exact, we find an exact complex $0 \to \Yoneda(X) \to \Yoneda(Y) \stackrel{\Yoneda(f)}{\deflation} \Yoneda(Z) \to \coker \Yoneda(f) \to 0$ in $\smod \EE$.  As $\coker \Yoneda(f) \in \eff(\EE)$, we have $Q(\coker \Yoneda(f)) = 0.$  As $Q\colon \mod \EE \to \smod(\EE)/\eff(\EE)$ is exact, we find that $Q\circ \Upsilon$ applied to the conflation $X \inflation Y \deflation Z$ gives a conflation (=short exact sequence) in $\smod(\EE)/\eff(\EE)$.  This shows that $Q\circ \Upsilon$ preserves conflations.%
%
\end{proof}

\begin{remark}
	In \cite[Theorem~2.9]{Ogawa19}, Ogawa shows that $\smod(\EE)/\eff(\EE)\simeq \text{lex}(\EE)$ for any extriangulated category $\EE$ with weak kernels (any exact category is extriangulated in the sense of \cite{NakaokaPalu19}). In \cite[Theorem~6.11]{Fiorot19}, Fiorot shows that \Cref{Theorem:AlternativeDescriptionOfLeftHeart} holds for any ($n$-)quasi-abelian category $\EE$. Hence, for any quasi-abelian category $\EE$, we have the following equivalent characterizations of the left heart $\mathcal{LH}(\EE)$:
	\[\mathcal{LH}(\EE)\simeq \smod(\EE)/\eff(\EE)\simeq \text{lex}(\EE)\simeq Q_l(\EE),\]
	where $Q_l(\EE)$ is the left abelian cover as defined in \cite{Rump01,Rump20}.
\end{remark}

\begin{example}
\begin{enumerate}
	\item If $\EE$ is abelian, then the left $\operatorname{t}$-structure is the standard $\operatorname{t}$-structure;  the heart of the standard $\operatorname{t}$-structure is $\EE$ itself.
	\item If $\EE$ is quasi-abelian, then the $\operatorname{t}$-structure given here is the left $\operatorname{t}$-structure from \cite[Definition~1.2.18]{Schneiders99}.
	\item If $\EE$ is equipped with the split conflation structure (thus, the only conflations are the split kernel-cokernel pairs), then $\D(\EE) = \K(\EE) \simeq \D(\smod \EE)$ where this last equivalence uses that objects in $\mod \EE$ have projective dimension at most two (as $\EE$ has kernels).  The left $\operatorname{t}$-structure is the canonical $\operatorname{t}$-structure on $\D(\smod \EE)$.  We see that the heart is equivalent to the category $\mod \EE$ of finitely presented functors (see also \Cref{Theorem:AlternativeDescriptionOfLeftHeart}).
\end{enumerate}
\end{example}
\section{Additive regular categories, admissible kernels, and the admissible intersection property}

In this section, we consider additive regular categories (see Definition \ref{definition:DeflationRegular} below).  We show that, endowing an additive regular category $\EE$ with the class of conflations consisting of all kernel-cokernel pairs, $\EE$ has the structure of a deflation-exact category.  In fact, the conflations of a deflation-exact category are given by the kernel-cokernel pairs of a regular category if and only if one of the following equivalent conditions hold: $\EE$ has admissible kernels (\Cref{Definition:AdmissibleKernels}) or $\EE$ has admissible intersections (\Cref{definition:(AI)-Category}).  This will be shown in \Cref{proposition:EquivalentCharacterizations}.

As a deflation-exact category, $\EE$ admits a derived category $\D(\EE)$ and the construction of the left heart $\mathcal{LH}(\EE)$ as in \Cref{section:LeftHeart} goes through: we show that $\EE$ is a uniformly preresolving subcategory of $\mathcal{LH}(\EE)$ so that the embedding $\EE\hookrightarrow \mathcal{LH}(\EE)$ lifts to a derived equivalence.

\subsection{Additive regular categories and their conflation structure}

We start by defining an additive regular category.  This definition is an additive version of a regular category, as defined in \cite{BarrGrilletVanOsdol71,BorceuxBourn04}.

\begin{definition}\label{definition:DeflationRegular}
An additive category is called \emph{additive regular} if
\begin{enumerate}[label=\textbf{Reg\arabic*},start=1]
	\item\label{DR1} every morphism $f$ has a factorization $f=m\circ p$ where $p$ is a cokernel and $m$ is a monomorphism, and
	\item\label{DR2} the pullback along every cokernel exists and the pullback of a cokernel is a cokernel.
\end{enumerate}
The dual of an additive regular category is called an \emph{additive coregular category}.
\end{definition}

\begin{remark}
\begin{enumerate}
	\item In \cite{Crivei12} (based on \cite{RichmanWalker77,SiegWegner11}), a cokernel $c$ was called \emph{semi-stable} if pullbacks along $c$ exist and the pullback of $c$ is a cokernel.  With this terminology, one can reformulate axiom \ref{DR2} as: all cokernels are semi-stable.
	\item\label{enumerate:RemarkRegular} In \cite[p.122]{BarrGrilletVanOsdol71}, a \emph{regular category} is a (not necessarily additive) finitely complete category where the class of regular epimorphisms satisfies the following properties: (i) every morphism $f$ has a factorization $f=m\circ p$ where $p$ is a regular epimorphism and $m$ is a monomorphism, and (ii) the pullback of a regular epimorphism is a regular epimorphism.  As additive regular categories are finitely complete (see Proposition \ref{proposition:DRProperty}\eqref{enumerate:DRProperty1} below), we see that additive categories are precisely those categories which are both additive and regular.
	
	  Note that a \emph{regular epimorphism} is an epimorphism that occurs as the coequalizer of a pair of parallel morphisms \cite[Definition~4.3.1]{BorceuxI}.  Hence, in a (pre)additive category, regular epimorphisms are precisely the cokernel maps.
	\item Let $\EE$ be an additive regular category.  We write $\fE$ for the class of cokernels and $\fM$ for the class of monomorphism.  As an additive regular category is regular (see \cref{enumerate:RemarkRegular} above), it follows from \cite{Kelly91} that the pair $(\fE, \fM)$ defines a factorization system on $\EE$ (in the sense of \cite[Definition~5.5.1]{BorceuxI}, also called a factorization \cite[Section~2.2]{FreydKelly72} or an orthogonal factorization system \cite[Section~11.2]{Riehl14}).
\end{enumerate}
\end{remark}

\begin{remark}
Not all authors require a regular category to be finitely complete (see, for example, \cite[Definition~2.1.1]{BorceuxII} and \cite[p. 4]{BarrGrilletVanOsdol71}).  These two definitions of a regular category coincide when the category is additive (see \cite[Lemma~2.6.6]{BorceuxII}).
\end{remark}

\begin{proposition}\label{proposition:DRProperty}
Let $\EE$ be an additive regular category.
\begin{enumerate} 
	\item\label{enumerate:DRProperty1} Each morphism in $\EE$ admits a kernel.
	\item Each kernel admits a cokernel.
	\item\label{enumerate:DRProperty2} Each cokernel is the cokernel of its kernel, and each kernel is the kernel of its cokernel.
	\item\label{enumerate:DRProperty3} The cokernel-monomorphism factorization in axiom \ref{DR1} is unique up to isomorphism.
\end{enumerate}
\end{proposition}

\begin{proof}
As cokernels have pullbacks in $\EE$, every cokernel $p\colon X \to Y$ admits a kernel; this kernel can be found as the pullback along $0 \to Y.$  Let $f\colon X \to Y$ be any morphism, and let $f = m \circ p$ be a cokernel-mono factorization.  We find that $\ker p = \ker f$, so that $f$ does admit a kernel.  Moreover, $p = \coker (\ker f)$ so that all kernel maps have cokernels.

The third statement is standard (see, for example, \cite[Proposition I.13.3]{Mitchell65} together with its dual).  For the last statement, let $f\colon X \to Y$ be any morphism in $\EE$ with cokernel-mono factorization $X \xrightarrow{p} I \xrightarrow{m} Y.$  By \eqref{enumerate:DRProperty2}, we see that $p = \coker(\ker p)$.  As $\ker p = \ker f$, the uniqueness follows.
\end{proof}

\begin{proposition}\label{Proposition:DeflationAICategorySatisfiesAxiomR3+}
Any additive regular category is a deflation-exact category (where the conflations are given by all kernel-cokernel pairs) satisfying axiom \ref{R3+}.
\end{proposition}

\begin{proof}
Choosing the class of all kernel-cokernel pairs as conflations, every cokernel is a deflation; this follows from \Cref{proposition:DRProperty}.  That this conflation structure satisfies axiom \ref{R2} is just axiom \ref{DR2}.  Hence, by \Cref{proposition:CriterionDeflationRegular}, this conflation structure gives a deflation-exact category.  It follows from \cite[Proposition~5.12]{Kelly69} that axiom \ref{R3+} is satisfied as well.
\end{proof}

\subsection{On admissible kernels and admissible intersections}

In the previous subsection, we started with an additive regular category and endowed it with a conflation structure.  In this subsection, we start with a deflation-exact category $\EE$ and find two properties which are equivalent to $\EE$ being additive regular.  The first property we consider has already been mentioned in \cite[\S 1.3.22]{BeilinsonBernsteinDeligne82} for exact categories.

\begin{definition}\label{Definition:AdmissibleKernels}
Let $\EE$ be a conflation category. We say that $\EE$ has \emph{admissible kernels} if every morphism admits a kernel and kernels are inflations. Having \emph{admissible cokernels} is defined dually.
\end{definition}

In \cite{Previdi12}, the admissible intersection property is introduced for exact categories (see \cite{Buhler21} for some corrections), and in \cite{HassounRoy19,BrustleHassounTattar20} for pre-abelian exact categories. It is shown in \cite[Theorem~6.1]{HassounShahWegner20} that a pre-abelian exact category satisfying the admissible intersection property is quasi-abelian. However, the admissible intersection property can be defined for general conflation categories. 

\begin{definition}\label{definition:(AI)-Category}
	Let $\EE$ be a conflation category. The category $\EE$ satisfies the \emph{admissible intersection property} if for any two inflations $f\colon X\inflation Z$ and $g\colon Y\inflation Z$, the pullback of $f$ along $g$ exists and is of the following form:
		\begin{equation*}
			\begin{tikzcd}
				P\arrow[tail]{r}{}\arrow[tail]{d}[swap]{} \commutes[\text{PB}]{dr}& Y \arrow[tail]{d}{g}\\
				X \arrow[tail]{r}[swap]{f} & Z
			\end{tikzcd}
		\end{equation*}
	The \emph{admissible cointersection property} is defined dually.
\end{definition}

The following lemma (based on \cite[Proposition~4.8]{BrustleHassounTattar20}) shows that the property of having admissible kernels and the admissible intersection property coincide for conflation categories.

\begin{lemma}\label{Lemma:InterpretationOfAI}
	Let $\EE$ be a conflation category such that all split kernel-cokernel pairs are conflations. The following are equivalent:
	\begin{enumerate}
		\item The category $\EE$ satisfies the admissible intersection property.
		\item The category $\EE$ has admissible kernels.
	\end{enumerate}
\end{lemma}

\begin{proof}
	Assume that the admissible intersection property holds. Let $g\colon Y\to Z$ be a morphism in $\EE$. As all split kernel-cokernel pairs are conflations, the sequences 
	$$Y\xrightarrow{\begin{psmallmatrix}1\\g\end{psmallmatrix}}Y\oplus Z\xrightarrow{\begin{psmallmatrix}-g & 1\end{psmallmatrix}}Z \mbox{ and } Y\xrightarrow{\begin{psmallmatrix}1\\0\end{psmallmatrix}}Y\oplus Z\xrightarrow{\begin{psmallmatrix}0 & 1\end{psmallmatrix}}Z$$
	are conflations. By the admissible intersection property, we have the following pullback diagram:
	\begin{equation*}
		\begin{tikzcd}
			P\arrow[tail]{r}{f}\arrow[tail]{d}[swap]{f'} \commutes[\text{PB}]{dr}& Y \arrow[tail]{d}{\begin{psmallmatrix}1\\g\end{psmallmatrix}}\\ Y \arrow[tail]{r}[swap]{\begin{psmallmatrix}1\\0\end{psmallmatrix}} & Y\oplus Z \arrow[two heads]{r}[swap]{(\hspace{0.5pt}0\;1\hspace{0.5pt})\:}& Z.
		\end{tikzcd}
	\end{equation*}
	As the bottom row is a kernel-cokernel pair and the square is a pullback, it follows that $f=\ker(\begin{psmallmatrix}0&1\end{psmallmatrix}\begin{psmallmatrix}1\\g\end{psmallmatrix})=\ker(g)$.
	
	The reverse implication follows immediately from \cite[Proposition I.13.2]{Mitchell65}, where it is shown that $f$ is the kernel of $\begin{psmallmatrix}0&1\end{psmallmatrix}\begin{psmallmatrix}1\\g\end{psmallmatrix}$ and $f'$ is the kernel of $\begin{psmallmatrix}-g&1\end{psmallmatrix}\begin{psmallmatrix}1\\0\end{psmallmatrix}$.
\end{proof}

\begin{remark}\label{REM}
The conflations of a conflation category $\EE$ having admissible kernels, are given by all kernel-cokernel pairs. Moreover, as every cokernel is the cokernel of its kernel, all cokernels are deflations.
\end{remark}

For deflation-exact categories, the above lemma can be extended (the proof is an adaptation of \cite[Proposition~I.1.4]{Schneiders99}).

\begin{proposition}\label{Proposition:InterpretationOfAIForDeflationExactCategories}
	Let $\EE$ be a deflation-exact category. The following are equivalent:
	\begin{enumerate}
		\item\label{lemma:InterpretationOfAIForDeflationExactCategories1} The admissible intersection property holds.
		\item\label{lemma:InterpretationOfAIForDeflationExactCategories2} The category $\EE$ has admissible kernels.
		\item\label{lemma:InterpretationOfAIForDeflationExactCategories3} Every morphism has a deflation-mono factorization, i.e.~any morphism $g\colon Y\to Z$ factors as $Y\deflation\coim(g)\hookrightarrow Z$.
	\end{enumerate}
	Moreover, a factorization as in \eqref{lemma:InterpretationOfAIForDeflationExactCategories3} is unique up to isomorphism.
\end{proposition}

\begin{proof}
	By \Cref{Remark:BasicDefinitions}, all split kernel-cokernel pairs are conflations in $\EE$. The equivalence $\eqref{lemma:InterpretationOfAIForDeflationExactCategories1}\Leftrightarrow \eqref{lemma:InterpretationOfAIForDeflationExactCategories2}$ now follows from \Cref{Lemma:InterpretationOfAI}.
	
	Assume \eqref{lemma:InterpretationOfAIForDeflationExactCategories2}. Let $g\colon Y\to Z$ be a map. As $g$ admits a kernel which is an inflation, we find a sequence $\ker(g)\stackrel{f}{\inflation}Y\stackrel{h}{\deflation}\coim(g)\xrightarrow{k}Z$ such that $k\circ h=g$. We claim that $k$ is a monomorphism, to that end, let $t\colon T\to \coim(g)$ be a map such that $k\circ t=0$. By axiom \ref{R2}, the pullback of $t$ along the deflation $h$ exists and we obtain the following commutative diagram:
	\begin{equation*}
		\begin{tikzcd}
			&\ker(g)\arrow[tail]{d}{f}&\\
			P\arrow{r}{t'}\arrow[two heads]{d}[swap]{h'}\arrow[dotted]{ru}{\exists!\hspace{0.5pt}u}& Y \arrow[two heads]{d}{h}\arrow{r}{g}& Z\arrow[equal]{d}\\
			T \arrow{r}[swap]{t} & \coim(g)\arrow{r}[swap]{k}& Z
		\end{tikzcd}
	\end{equation*}
	By commutativity of the diagram, $g\circ t'=0$ holds and thus there exists a unique map $u\colon P\to \ker(g)$ such that $f\circ u=t'$. It follows that $t\circ h'=h\circ t'= h\circ f\circ u=0$. Since $h'$ is a deflation, it is epic, and thus $t\circ h'=0$ implies that $t=0$. This shows that $k$ is monic and thus \eqref{lemma:InterpretationOfAIForDeflationExactCategories3} holds. The implication $\eqref{lemma:InterpretationOfAIForDeflationExactCategories3}\Rightarrow \eqref{lemma:InterpretationOfAIForDeflationExactCategories2}$ and the uniqueness of a deflation-mono factorization are straightforward to show.
\end{proof}

\begin{remark}
	\begin{enumerate}
		\item Every left quasi-abelian category is a deflation-exact category having admissible kernels, or equivalent, satisfying the admissible intersection property. Despite \cite[Theorem~6.1]{HassounShahWegner20}, such a category need not be quasi-abelian as it might fail to be exact. Such an example is given by the category $\mathsf{LB}$ (see \cite[Theorem~3.4]{HassounShahWegner20} or Section \ref{Section:LBSpaces}).
		\item Despite \cite[Theorem~6.1]{HassounShahWegner20}, an exact category with the admissible intersection property might fail to be quasi-abelian as well. Indeed, in \cite[Example~7.18]{HenrardvanRoosmalen19b} it shown that the exact hull $\II^{\mathsf{ex}}$ of the Isbell category $\II$ need not be pre-abelian. On the other hand, $\II^{\mathsf{ex}}$ is exact and has the admissible intersection property by \Cref{Proposition:InterpretationOfAIForDeflationExactCategories} and \Cref{Corollary:TheHullInheritsAdmissibleKernels}.
	\end{enumerate}
\end{remark}

\begin{proposition}\label{proposition:EquivalentCharacterizations}
The following are equivalent for an additive category $\EE$.
\begin{enumerate}
  \item $\EE$ is an additive regular category,
	\item $\EE$ is a deflation-exact category with admissible kernels, and
	\item $\EE$ is a deflation-exact category with admissible intersections,
\end{enumerate}
where the conflation structure is given by the class of all kernel-cokernel pairs.
\end{proposition}

\begin{proof}
This follows from \Cref{Proposition:DeflationAICategorySatisfiesAxiomR3+} and \Cref{Proposition:InterpretationOfAIForDeflationExactCategories}.
\end{proof}

\begin{remark}\label{remark:TakingConflationStructure}
Being an additive regular category is a property of an (additive) category.  In contrast, being a deflation-exact category with admissible kernels (or equivalently, admissible intersections) is a property of a conflation category.  We have shown that an additive regular category endowed with the maximal conflation structure is deflation-exact with admissible kernels.

Later in this article, for example in \Cref{Proposition:ClosedUnderSubjectsInheritsDeflationExactAndAdmissibeKernels}, we consider results which produce a deflation-exact category having admissible kernels.  This is slightly stronger than producing an additive regular category.  Indeed, the former means that we get an additive regular category with a conflation structure, and states on top, that this conflation structure is maximal.
\end{remark}

\subsection{Some examples}  We now provide some examples of deflation-exact categories with admissible kernels.  We start with an easy criterion.

\begin{proposition}\label{Proposition:ClosedUnderSubjectsInheritsDeflationExactAndAdmissibeKernels}
	Let $\EE$ be a deflation-exact category having admissible kernels. If $\FF\subseteq \EE$ is a subcategory closed under subobjects, then $\FF$ is deflation-exact and has admissible kernels.
\end{proposition}

\begin{proof}
	Assume that $\FF\subseteq \EE$ is closed under subobjects. In particular, $\FF\subseteq\EE$ is deflation-closed and thus inherits a deflation-exact structure by \Cref{Proposition:DeflationClosed}. Let $f\colon X\to Y$ be a morphism in $\FF$. By \Cref{Proposition:InterpretationOfAIForDeflationExactCategories}, $f$ admits a deflation-mono factorization $X\stackrel{f'}{\deflation} \coim(f)\stackrel{f''}{\hookrightarrow} Y$ in $\EE$. By assumption, $\ker(f),\coim(f)\in \FF$. The result then follows from \Cref{Proposition:InterpretationOfAIForDeflationExactCategories} as $\ker(f)\inflation X\deflation \coim(f)$ is a conflation in $\FF$ and the map $f''$ is a monomorphism in $\FF$. 
\end{proof}

\begin{example}
	For any category $\AA$, a \emph{preradical functor} $T$ is a subfunctor of the identity functor on $\AA$.  Let $\AA$ be a conflation category.  Consider a preradical functor $T$ with embedding $\eta\colon T \to 1_\AA$.  Assume now that for each $A \in \AA$, the given monomorphism $\eta_A\colon T(A) \to A$ is an inflation in $\AA$. To any such a preradical functor $T$, one assigns the full subcategory $\TT$ consisting of those objects $C\in \AA$ such that $\eta_C\colon T(C) \to C$ is an isomorphism.  Using the naturality of $T \to 1_\AA$, one readily verifies that $\TT\subseteq \AA$ is closed under epimorphic quotients (see, for example, \cite[Proposition~2]{BurbanHorbachuk11}).  Indeed, let $f\colon A \to B$ be an epimorphism in $\AA$ with $A \in \TT$.  Naturality of $\eta$ gives the following commutative diagram:
	\begin{equation*}
			\begin{tikzcd}
						T(A) \arrow[tail]{d}{\eta_A} \arrow{r}{T(f)} & T(B) \arrow[tail]{d}{\eta_B} \\
			A\arrow{r}{f} & B
		\end{tikzcd}
	\end{equation*}
As $\eta_A\colon T(A) \to A$ is an isomorphism, we find that the composition $f \circ \eta_A = \eta_B \circ T(f)$ is an epimorphism and, hence, so is $\eta_B\colon T(B) \inflation B$.  This shows that $\eta_B$ is an isomorphism so that $B \in \TT.$  
	
	If $\AA$ is inflation-exact with admissible cokernels and $T$ is a normal preradical functor on $\AA$ (that is, the monomorphisms $\eta_A\colon T(A) \inflation A$ are inflations), then the dual of \Cref{Proposition:ClosedUnderSubjectsInheritsDeflationExactAndAdmissibeKernels} yields that $\TT$ is an inflation-exact category having admissible cokernels.

	As a more specific example, let $R$ be a ring and let $I\triangleleft R$ be a left ideal. Let $\Mod(R)$ be the category of right $R$-modules. The functor $T$ mapping $M\in \Mod(R)$ to $T(M)=MI$ is a normal preradical functor. The corresponding subcategory $\TT = \{M \in \Mod(R)\mid M = MI\}$ of $\Mod(R)$ is an inflation-exact category having admissible cokernels.
\end{example}

\begin{example}
Let $R$ be a commutative artin ring and let $A$ be an artin $R$-algebra.  Let $M \in \smod(A)$ be a finitely generated module.  Denote by $\operatorname{fac}(M)$ the full additive subcategory of $\smod(A)$ consisting of factor modules of finite direct sums of $M$.  It follows from \Cref{Proposition:ClosedUnderSubjectsInheritsDeflationExactAndAdmissibeKernels} that $\operatorname{fac}(M)$ is an inflation-exact category with admissible cokernels.  If $\Hom_A(M, \tau M) = 0$, that is, $M$ is \emph{$\tau$-rigid} (\cite[Definition 0.1]{AdachiIyamaReiten14}), then $\operatorname{fac}(M)$ is an extension-closed subcategory of $\smod(A)$ and hence exact (see \cite[Theorem 5.10]{AuslanderSmalo81}).
\end{example}

\begin{example}
Let $\EE$ be a deflation-exact category and let $\JJ$ be any small category.  The category $\EE^\JJ \coloneqq \Fun(\JJ, \EE)$ inherits a deflation-exact structure from $\EE$ in the following way: a sequence $F \to G \to H$ in $\EE^\JJ$ is a conflation if and only if $F(J) \to G(J) \to H(J)$ is a conflation, for every $J \in \Ob(\JJ).$  If $\EE$ has admissible kernels, then so does $\EE^\JJ.$
\end{example}

\section{The left heart and the exact hull}

Let $\EE$ be a deflation-exact category with admissible kernels.  In this section, we have a closer look at the bounded derived category $\Db(\EE)$ and study two subcategories of $\Db(\EE)$: the left heart and the exact hull.

\subsection{The left heart}\label{Section:SubsectionLeftHeart}

In \Cref{section:LeftHeart}, we described the left heart of a deflation-exact category with kernels.  We did not require any compatibility between the exact structure and the kernels.  In this section, we narrow the scope and consider only those cases where the kernels are inflations.  This allows us to strengthen some results presented in \Cref{section:LeftHeart}.

Throughout this section, let $\EE$ be a deflation-exact category with admissible kernels.

The following proposition is a straightforward adaptation of \cite[Proposition~1.2.19]{Schneiders99} and strengthens \Cref{proposition:DescriptionCohomology}.

\begin{proposition}\label{Proposition:RepresentationsOfObjectsInLeftHeart}
Let $\EE$ be a deflation-exact category with admissible kernels.	Let $C^{\bullet}\in \D(\EE)$. The complex $\LH^n(C^{\bullet})$ is isomorphic to the complex
	\[\dots \to 0\to \coim(d^{n-1})\hookrightarrow \ker(d^n)\to 0 \to \dots\]
	with $\ker(d^n)$ in degree $0$.
\end{proposition}

\begin{proof}
By \Cref{proposition:DescriptionCohomology}, the complex $\operatorname{LH}^nC^{\bullet}=\tau^{\leq n}\tau^{\geq n}C^{\bullet}$ is given by 
	\[\dots \to 0\to \ker(d^{n-1})\to C^{n-1} \to \ker(d^n)\to 0\to \dots\]
	We consider the deflation-mono factorization $C^{n-1} \stackrel{\beta}{\deflation}\coim(d^{n-1})\stackrel{\gamma}{\hookrightarrow}\ker(d^n)$ from \Cref{Proposition:InterpretationOfAIForDeflationExactCategories}, giving us the following commutative diagram
		\begin{equation*}
		\begin{tikzcd}
			\cdots\arrow{r} &0\arrow{r}\arrow{d} & \ker(d^{n-1})\arrow[tail]{r}{\alpha} \arrow[equal]{d}& C^{n-1}\arrow[r,"\beta", twoheadrightarrow]\arrow[equal]{d}&\coim(d^{n-1})\arrow{r}\arrow[hook]{d}{\gamma} &0\arrow{d} \arrow{r} &\cdots\\
			\cdots\arrow{r} &0\arrow{r} & \ker(d^{n-1})\arrow[tail]{r}{\alpha}& C^{n-1}\arrow{r}&\ker(d^n)\arrow{r} &0 \arrow{r} &\cdots
		\end{tikzcd}
	\end{equation*}
	We interpret this diagram as a morphism between complexes: $f^\bullet\colon D^\bullet \to \operatorname{LH}^n C^{\bullet}$; the complexes here are given by the rows in the previous diagram.  As the top row is an acyclic complex, the morphism $\operatorname{LH}^n C^{\bullet} \to \cone(f^\bullet)$ is a quasi-isomorphism.  It is easy to see that $\cone(f^\bullet)$ is given by the complex $\dots \to 0\to \coim(d^{n-1})\hookrightarrow \ker(d^n)\to 0 \to \dots$with $\ker(d^n)$ in degree $0$, up to homotopy.
\end{proof}

\begin{proposition}\label{Proposition:TheEmbeddingToHeartProperties}Let $\EE$ be a deflation-exact category with admissible kernels.	Let $\phi\colon \EE\to \mathcal{LH}(\EE)$ be the canonical embedding.
	\begin{enumerate}
		\item\label{Item:Proposition:TheEmbeddingToHeartProperties4} The subcategory $\EE\subseteq\mathcal{LH}(\EE)$ is closed under subobjects. 
		\item\label{Item:Proposition:TheEmbeddingToHeartProperties2} For every object $Z\in \mathcal{LH}(\EE)$, there exists a short exact sequence $X\inflation Y\deflation Z$ with $X,Y\in \EE$.
	\end{enumerate}
\end{proposition}

\begin{proof}
	\begin{enumerate}
		%
		\item Let $f\colon X\hookrightarrow Y$ be a monomorphism in $\mathcal{LH}(\EE)$ with $Y\in \EE$. It follows from \Cref{Proposition:TheEmbeddingToHeartPropertiesNew}\eqref{Item:Proposition:TheEmbeddingToHeartProperties2x} that there is an epimorphism $g\colon B\to X$ in $\mathcal{LH}(\EE)$ with $B\in \EE$. Consider the deflation-mono factorization $B\deflation \coim(f\circ g)\hookrightarrow Y$ of the morphism $f\circ g$ in $\EE$.  Embedding this factorization in $\mathcal{LH}(\EE)$ gives a deflation-mono factorization of $f \circ g$ in $\mathcal{LH}(\EE)$ (this uses \Cref{Proposition:TheEmbeddingToHeartPropertiesNew}\eqref{Item:Proposition:TheEmbeddingToHeartProperties1x} and \eqref{Item:Proposition:TheEmbeddingToHeartProperties3x}).  Since such a factorization is unique in the abelian category $\mathcal{LH}(\EE)$, we find $X\cong \coim(f\circ g)\in \EE$.
		\item As $Z\in \mathcal{LH}(\EE)$, we know by \Cref{Proposition:RepresentationsOfObjectsInLeftHeart} that $Z$ can be represented by a complex
\[\dots \to 0\to X \stackrel{f}{\hookrightarrow} Y \to 0 \to \dots\]
This gives a triangle $i(X) \xrightarrow{i(f)} i(Y) \to Z \to \Sigma i(X)$ in $\D(\EE)$, where $i\colon \EE \to \D(\EE)$ is the canonical embedding.  The long exact sequence coming from the cohomology functors $\LH^i$ now give the required short exact sequence.		\qedhere
	\end{enumerate}
\end{proof}

\begin{remark}\label{Remark:DeflationMonoFactorizationCompatibleWithLeftHeart}
For any map $f\colon X\to Y$ in $\EE$, the deflation-mono factorization in $\EE$ (see \Cref{Proposition:InterpretationOfAIForDeflationExactCategories}) coincides with the epi-mono factorization of $f$ in the abelian category $\mathcal{LH}(\EE).$
\end{remark}


\begin{corollary}\label{corollary:EmbeddingInHeartIsTriangleEquivalence}
Let $\EE$ be a deflation-exact category with admissible kernels.	The category $\EE$ is a uniformly preresolving subcategory of $\mathcal{LH}(\EE)$ with $\resdim_{\EE}(\mathcal{LH}(\EE))\leq 1$. Consequently,  the embedding lifts to a triangle equivalence $\Phi\colon\Dstar(\EE)\to \Dstar(\mathcal{LH}(\EE))$ for $*\in \{-,b,\emptyset\}$.
\end{corollary}

\begin{proof}
The only improvement over \Cref{Theorem:EmbeddingInHeartIsTriangleEquivalence} is that $\resdim_{\EE}(\mathcal{LH}(\EE))\leq 1$.  This follows from \Cref{Proposition:TheEmbeddingToHeartProperties}.
\end{proof}

\begin{proposition}
Let $\EE$ be a deflation-exact category with admissible kernels and let $\AA$ be any abelian category.  For a conflation-exact functor $F\colon \EE \to \AA$, the following are equivalent:
\begin{enumerate}
	\item\label{enumerate:KernelsMono1} $F$ commutes with kernels,
	\item\label{enumerate:KernelsMono2} $F$ maps monomorphisms to monomorphisms.
\end{enumerate}
\end{proposition}

\begin{proof}
A morphism is a monomorphism if and only if the kernel is zero.  This shows the implication \eqref{enumerate:KernelsMono1} $\Rightarrow$ \eqref{enumerate:KernelsMono2}.  For the other implication, let $f\colon X \to Y$ be any morphism in $\EE.$  Let $X \stackrel{p}{\deflation} \coim f \stackrel{i}{\hookrightarrow} Y$ be the deflation-mono factorization.  We have
\[\ker F(f) = \ker (F(i) \circ F(p)) \stackrel{(*)}{=} \ker F(p) \stackrel{(**)}{=} F(\ker p) = F(\ker f),\]
where we have used that $F$ preserves monomorphisms (*) and that $F$ is conflation-exact (**).  This shows that $F$ commutes with kernels, as required.
\end{proof}

The previous proposition allows for a reformulation of the 2-universal property of the left heart (see \Cref{Proposition:UniversalPropertyOfLeftHeart}).

\begin{corollary}\label{corollary:UniversalPropertyOfLeftHeart}
Let $\EE$ be a deflation-exact category with admissible kernels.	The embedding $\phi\colon \EE\to \mathcal{LH}(\EE)$ is 2-universal among conflation-exact functors to abelian categories that preserve monomorphisms.
\end{corollary}

The following proposition is somewhat of a converse to \Cref{corollary:EmbeddingInHeartIsTriangleEquivalence}.

\begin{proposition}\label{proposition:NewEmbedding}
Let $\AA$ be an abelian category.  Let $\EE \subseteq \AA$ be a full subcategory satisfying condition \ref{PR1} (thus, every object in $\AA$ is a quotient of an object in $\EE$).  If $\EE$ is closed under subobjects, then $\EE$ has admissible kernels and $\mathcal{LH}(\EE) \simeq \AA$.
\end{proposition}

\begin{proof}
As $\EE$ is closed under subobjects, we know that $\EE$ is a uniformly preresolving subcategory of $\AA$.  By \Cref{Proposition:ClosedUnderSubjectsInheritsDeflationExactAndAdmissibeKernels}, we know that $\EE$ is a deflation-exact category with admissible kernels.  The rest follows from \Cref{proposition:CharacterizationOfHearts}.
\end{proof}


\begin{remark}
In the language of \cite[Definition~1.1]{Kvamme20}, the previous result, together with \Cref{Proposition:TheEmbeddingToHeartProperties} implies that additive regular categories axiomatize subcategories of abelian categories which are generating (that is, satisfying axiom \ref{PR1}) and are closed under subobjects.
\end{remark}

\subsection{The exact hull}\label{Section:SubsectionTheExactHull}

Recall from \Cref{proposition:EquivalentCharacterizations} that an additive regular category $\EE$ is a deflation-exact category with admissible kernels.  As a deflation-exact category, it admits an exact hull $\EE^{\mathsf{ex}}$ (see \Cref{subsection:ExactHull}).  In this subsection, we show that the exact hull $\EE^{\mathsf{ex}}$ has admissible kernels as well.  In other words, the property of having admissible kernels inherits to taking the exact hull.  This means that the exact hull of an additive regular category is still an additive regular category.

The exact hull of $\EE$ is defined as the extension-closure of $\EE$ in the derived category $\Db(\EE).$  In the following proposition, we describe the exact hull as a subcategory of the left heart of $\EE.$

\begin{proposition}\label{Proposition:PhiFactorsThroughHull}
There is a fully faithful conflation-exact functor $k\colon\EE^{\mathsf{ex}}\to \mathcal{LH}(\EE)$ for which the diagram
\[\xymatrix{\EE \ar[rr]^{\phi} \ar[rd]_{j}&& {\mathcal{LH}(\EE)} \\ &{\EE^{\mathsf{ex}}} \ar[ru]_{k} }\]
is essentially commutative.  In particular, the category $\EE^{\mathsf{ex}}$ is a full and extension-closed subcategory of $\mathcal{LH}(\EE)$. Furthermore, $\phi,j$ and $k$ all lift to derived equivalences $\Dstar(\EE) \to \Dstar(\ex{\EE}) \to \Dstar(\mathcal{LH}(\EE))$ for $\ast \in \{\varnothing, \mathrm{b}, -\}$.
\end{proposition}

\begin{proof}
By the universal property of the embedding $j\colon \EE\to \EE^{\mathsf{ex}}$ (see \Cref{Theorem:ExactHull}), the functor $\phi$ factors (essentially uniquely) as $\EE \xrightarrow{j} \ex{\EE} \xrightarrow{k} \mathcal{LH}(\EE)$ where $k$ is an exact functor.  As $\phi$ and $j$ are fully faithful, so is $k$.  To see that $\EE^{\mathsf{ex}}$ is a full and extension-closed subcategory of $\mathcal{LH}(\EE)$, it suffices to note that $\mathcal{LH}(\EE)$ is an extension-closed subcategory of $\D(\EE)$ and that $\ex{\EE}$ is the extension-closure of $\EE \subseteq \mathcal{LH}(\EE)$ in $\D(\EE).$  Furthermore, as the embeddings $\phi$ and $j$ lift to triangle equivalences $\Dstar(\EE) \to \Dstar(\mathcal{LH}(\EE))$ and $\Dstar(\EE) \to \Dstar(\ex{\EE})$ for $\ast \in \{\varnothing, \mathrm{b}, -\}$ (see \Cref{Theorem:DivisiveDeflationExactIsFinitelyPreResolvingInExactHull} and \Cref{corollary:EmbeddingInHeartIsTriangleEquivalence}), so does $k$.
\end{proof}

In the following proposition, we use the categories $\EE_n$ from Notation \ref{notation:ExactHull}.

\begin{proposition}\label{Proposition:HullLiesClosedUnderSubobjectsInLeftHeart}
Let $\EE$ be a deflation-exact category with admissible kernels.	The subcategory $\EE^{\mathsf{ex}}$ is closed under subobjects in $\mathcal{LH}(\EE)$.
\end{proposition}

\begin{proof}
Consider a monomorphism $X\hookrightarrow Y$ in $\mathcal{LH}(\EE)$ and assume that $Y\in \EE^{\mathsf{ex}}$. We need to show that $X\in \EE^{\mathsf{ex}}$. By construction, $Y\in \EE_n$ for some $n\geq 0$.  We show, by induction on $n$, that $X \in \EE_{n}$ as well.  For $n=0$, \Cref{Proposition:TheEmbeddingToHeartProperties}.\eqref{Item:Proposition:TheEmbeddingToHeartProperties4} yields that $X\in \EE_0$. Now assume that $n\geq 1$. By definition, there is a conflation $A\inflation Y\deflation B$ in $\EE^{\mathsf{ex}}$ with $A\in \EE_{n-1}$ and $B\in \EE_0$. Consider the following commutative diagram in $\mathcal{LH}(\EE)$:
\[\begin{tikzcd}
		P\arrow[r]\arrow[d, hookrightarrow] & X\arrow[d, hookrightarrow]\arrow[r] & I\arrow[d, hookrightarrow]\\
		A\arrow[r, rightarrowtail] & Y\arrow[r,twoheadrightarrow] & B
\end{tikzcd}\]
	Here, $I$ is the image of the composition $X\hookrightarrow Y\to B$ and $P \to X$ is the kernel of $X \to I.$  In particular, the top line in this diagram is an exact sequence in $\mathcal{LH}(\EE)$, and thus corresponds to a triangle in $\Db(\EE)$.  By \cite[Proposition I.13.2]{Mitchell65}, the left square is a pullback and hence the induced map $P \to A$ is a monomorphism (as the pullback of a monomorphism is a monomorphism, see \cite[Proposition I.7.1]{Mitchell65}). By the induction hypothesis, $P\in \EE_{n-1}$ and the base case yields that $I\in \EE_0$.  It follows that $X\in \EE_n$ as required.
\end{proof}

\begin{theorem}\label{theorem:ExactHullRegular}
Let $\EE$ be a deflation-exact category with admissible kernels.  The exact hull $\EE^{\mathsf{ex}}$ of $\EE$ also has admissible kernels.
\end{theorem}

\begin{proof}
This follows from \Cref{Proposition:ClosedUnderSubjectsInheritsDeflationExactAndAdmissibeKernels} and \Cref{Proposition:HullLiesClosedUnderSubobjectsInLeftHeart}.
\end{proof}

\begin{corollary}\label{Corollary:TheHullInheritsAdmissibleKernels}
	\begin{enumerate}
		\item The functor $k\colon \ex{\EE} \to \mathcal{LH}(\EE)$ maps monomorphisms to monomorphisms.
		\item	The subcategory $\EE\subseteq \EE^{\mathsf{ex}}$ is closed under subobjects.
		\item\label{enumerate:TheHullInheritsAdmissibleKernels3} The embedding $j\colon \EE \to \ex{\EE}$ commutes with kernels.
		\item\label{enumerate:TheHullInheritsAdmissibleKernels4} A morphism $X \to Y$ in $\EE$ is a deflation if and only if it is a deflation in $\ex\EE$.
		\item Exact categories with admissible kernels are precisely the extension-closed subcategories of abelian categories that are closed under subobjects.
	\end{enumerate}
\end{corollary}

\begin{proof}
	\begin{enumerate}
		\item Consider a monomorphism $f\colon X \hookrightarrow Y$ in $\ex{\EE} \subseteq \mathcal{LH}(\EE).$  Take a morphism $t\colon T \to X$ such that $f \circ t = 0.$  By \Cref{Proposition:TheEmbeddingToHeartPropertiesNew}.\eqref{Item:Proposition:TheEmbeddingToHeartProperties2x}, there is an epimorphism $p\colon Z \to T$ with $Z \in \EE \subseteq \ex{\EE}.$  As $f$ is a monomorphism in $\ex{\EE}$, it follows from $f \circ t \circ p$ that $t \circ p = 0.$  As $p$ is an epimorphism, we find that $f=0$.  This shows that $f$ is a monomorphism in $\mathcal{LH}(\EE).$
		\item	Consider a monomorphism $f\colon X \hookrightarrow Y$ in $\ex{\EE}$ with $Y \in \EE.$  We have shown that $f$ is also a monomorphism in $\mathcal{LH}(\EE).$  It follows from \Cref{Proposition:TheEmbeddingToHeartProperties}.\eqref{Item:Proposition:TheEmbeddingToHeartProperties4} that $X \in \EE.$
		\item Follows directly from the fact that $\phi\colon \EE \to \mathcal{LH}(\EE)$ preserves kernels.
		\item Consider the conflation $K \inflation X \deflation Y$ in $\ex\EE$.  As $\EE \subseteq \ex\EE$ is closed under subobjects, we find $K \in \EE$.  We can now use that $j\colon \EE \to \ex\EE$ reflects conflations (see \Cref{Theorem:ExactHull}).
		\item Clearly any extension-closed subcategory of an exact category is exact. Combining this fact with \Cref{Proposition:ClosedUnderSubjectsInheritsDeflationExactAndAdmissibeKernels} yields that any extension-closed subcategory of an abelian category closed under subobjects is an exact category with admissible kernels. Conversely, any exact category $\EE$ equals its hull $\EE\cong \EE^{\mathsf{ex}}$. Additionally, if $\EE$ has admissible kernels, then $\EE\subseteq \mathcal{LH}(\EE)$ is closed under subobjects by \Cref{Proposition:HullLiesClosedUnderSubobjectsInLeftHeart}. By construction, $\EE^{\mathsf{ex}}$ lies extension-closed in $\mathcal{LH}(\EE)$. This concludes the proof. \qedhere
	\end{enumerate}
\end{proof}

It follows from \Cref{Lemma:MonosAndEpisInExactHull} that a morphism $f\colon X \to Y$ in a deflation-exact category that becomes an inflation in $\ex{\EE}$ is necessarily a monomorphism.  However, it gives no criterion for which monomorphisms become inflations.  The following provides such a criterion for deflation-exact categories with kernels; these kernels need not be admissible.

\begin{proposition}\label{Corollary:HullOfAIIsL1Closure} Let $\EE$ be a deflation-exact category satisfying axiom \ref{R3}.  Assume that $\EE$ admits all kernels.  Any inflation $f\colon X\inflation Y$ in $\EE^{\mathsf{ex}}$ with $X,Y\in \EE$ is a finite composition of inflations in $\EE$.
\end{proposition}

\begin{proof}
 We show that for any conflation $X\inflation Y\deflation Z$ in $\EE^{\mathsf{ex}}$ with $X,Y\in \EE$, the map $X\to Y$ is a finite composition of inflations in $\EE$. As $Z\in \EE^{\mathsf{ex}}$, there is an $n\geq 0$ such that $Z\in \EE_n$. We proceed by induction on $n\geq 0$. If $n=0$, $X\inflation Y$ is an inflation in $\EE$ as the embedding $j\colon \EE\to \EE^{\mathsf{ex}}$ reflects exactness (see \Cref{Theorem:ExactHull}). If $n\geq 1$, then there exists a conflation $A\inflation Z\deflation B$ in $\EE^{\mathsf{ex}}$ such that $A\in \EE_{n-1}$ and $B\in \EE$. Consider the following commutative diagram
	\[\xymatrix{
		X\ar@{>->}[r]\ar@{=}[d] & P\ar@{->>}[r]\ar@{>->}[d] & A\ar@{>->}[d]\\
		X\ar@{>->}[r] & Y\ar@{->>}[r]\ar@{->>}[d] & Z\ar@{->>}[d]\\
		& B\ar@{=}[r] & B
	}\] in $\EE^{\mathsf{ex}}$ where the upper-right square is bicartesian.  By \cite[Proposition~5.5]{HenrardvanRoosmalen20Preresolving}, we know that $\EE$ lies deflation-closed in $\EE.$  As $Y,B \in \EE$, we find that $P \in \EE$.  The induction hypothesis now shows that $X \inflation P$ is a finite string of inflations.
\end{proof}
\section{Quotients of additive regular categories}\label{section:Quotients}

In \cite{HenrardvanRoosmalen19b,HenrardvanRoosmalen19a}, a quotient/localization theory for (one-sided) exact categories at percolating subcategories is studied. This localization theory simultaneously generalizes localization theories for exact categories developed in \cite{Cardenas98,Schlichting04} and provides new examples (even for exact categories).  As additive regular categories are deflation-exact, this framework allows to take quotients of additive regular categories.  The main result is that a quotient of an additive regular category is again regular as an additive category, and that the induced deflation-exact structure on the quotient consists of all kernel-cokernel pairs.  In addition, we provide an easy characterization of percolating subcategories for additive regular categories (see \Cref{Proposition:ClassificationOfPercolatingSubcategoriesOfExactAICategory}).

\subsection{Basic definitions and results}

We recall the basic definitions and results from \cite{HenrardvanRoosmalen19b,HenrardvanRoosmalen19a}.

\begin{definition}\label{definition:GeneralPercolatingSubcategory}
	Let $\EE$ be a conflation category. A non-empty full subcategory $\AA$ of $\EE$ is called a \emph{deflation-percolating subcategory} of $\EE$ if the following axioms are satisfied:
	\begin{enumerate}[label=\textbf{P\arabic*},start=1]
		\item\label{P1} $\AA$ is a \emph{Serre subcategory}, meaning:
		\[\mbox{ If } A'\inflation A \deflation A'' \mbox{ is a conflation in $\EE$, then } A\in \AA \mbox{ if and only if } A',A''\in \AA.\]
		\item\label{P2} For all morphisms $X\rightarrow A$ with $X \in \EE$ and $A\in \AA$, there exists a commutative diagram
			\begin{equation*}
				\begin{tikzcd}
					A'\arrow{rd} & \\
					X\arrow{r}\arrow[two heads]{u} & A
				\end{tikzcd}
			\end{equation*}
				with $A'\in \AA$ and where $X \deflation A'$ is a deflation.
		\item\label{P3} For any composition $\xymatrix@1{X\ar@{>->}[r]^i & Y\ar[r]^t & T}$ which factors through $\AA$, there exists a commutative diagram 
		\begin{equation*}
			\begin{tikzcd}
				X\arrow[tail]{r}{i}\arrow[two heads]{d}{f} & Y\arrow[two heads]{d}{f'}\arrow[bend left]{ddr}{t} &\\
				A\arrow[tail]{r}{i'}\arrow[bend right]{rrd} & P\arrow[dotted]{rd} &\\
				&&T
			\end{tikzcd}
		\end{equation*}
		with $A \in \AA$ and such that the square $XYAP$ is a pushout square.	
		\item\label{P4} For all maps $X\stackrel{f}{\rightarrow} Y$ that factor through $\AA$ and for all inflations $A\stackrel{i}{\inflation} X$ (with $A \in \AA)$ such that $f\circ i=0$, the induced map $\coker(i)\to Y$ factors through $\AA$.
	\end{enumerate}
	By dualizing the above axioms one obtains a similar notion of an \emph{inflation-percolating subcategory} or an \emph{inflation-percolating subcategory}.
\end{definition}

\begin{remark}\label{Remark:P3RedundantForExact}
	If $\EE$ is exact, axiom \ref{P3} in the above definition is redundant (see \cite[Remark~4.4]{HenrardvanRoosmalen19a}).
\end{remark}

\begin{definition}\label{definition:WeakIsomorphism}
	Let $\AA$ be a full additive subcategory of $\EE$. A morphism $f\in \Mor(\EE)$ is called a \emph{weak $\AA$-isomorphism} if it is a finite composable string of inflations with cokernels in $\AA$ and deflations with kernels in $\AA$. The weak $\AA$-isomorphisms are denoted by $S_{\AA}$.
\end{definition}

The following theorem summarizes the results of \cite{HenrardvanRoosmalen19b,HenrardvanRoosmalen19a}. We write $i\colon \EE \to \Db(\EE)$ for the canonical embedding.

\begin{theorem}\label{Theorem:LocalizationTheorem}
	Let $\EE$ be a deflation-exact category and let $\AA\subseteq \EE$ be a deflation-percolating subcategory. 
	\begin{enumerate}
		\item The set $S_{\AA}$ is a right multiplicative system.
		\item The smallest conflation structure on $\EE[S_{\AA}^{-1}]$ such that the localization functor $Q\colon \EE\to \EE[S_{\AA}^{-1}]$ is conflation-exact, is a deflation-exact structure.
		\item The functor $Q$ satisfies the $2$-universal property of the quotient $\EE/\AA$ of deflation-exact categories.
		\item The localization sequence $\AA\to \EE\to \EE/\AA$ induces a Verdier localization sequence 
		\[\DAb(\EE)\to \Db(\EE)\to \Db(\EE/\AA),\]
		here $\DAb(\EE)$ is the thick triangulated subcategory of $\Db(\EE)$ generated by $i(\AA)$.
	\end{enumerate}
	If, in addition, $\EE$ is two-sided exact, $\EE[S_{\AA}^{-1}]^{\mathsf{ex}}$ satisfies the $2$-universal property of a quotient of exact categories. We write $\EE[S_{\AA}^{-1}]^{\mathsf{ex}}=\EE\dqq\AA$ to distinguish it from the one-sided quotient.
\end{theorem}

\subsection{Percolating subcategories of deflation-exact categories having admissible kernels}

We start with the following proposition, stating that having admissible kernels is stable under quotients.

\begin{proposition}\label{Proposition:QuotientsPreserveTheAIProperty}
	Let $\EE$ be a deflation-exact category and let $\AA\subseteq \EE$ be a deflation-percolating subcategory. If $\EE$ has admissible kernels, so does $\EE/\AA$. Furthermore, $\EE[S_{\AA}^{-1}]^{\mathsf{ex}}$ has admissible kernels as well.
\end{proposition}

\begin{proof}
	By \Cref{Theorem:LocalizationTheorem} the quotient $\EE/\AA$ is a deflation-exact category. By \Cref{Proposition:InterpretationOfAIForDeflationExactCategories}, it suffices to show that $\EE/\AA$ admits kernels and that kernels are inflations. Since $S_{\AA}$ is a right multiplicative system, the localization functor $Q\colon \EE\to \EE/\AA\simeq \EE[S_{\AA}^{-1}]$ commutes with kernels. Hence every morphism in $\EE/\AA$ has a kernel, moreover, as $Q$ is a conflation-exact functor and every kernel in $\EE$ is an inflation, kernels in $\EE/\AA$ are inflations as well. The last part follows from \Cref{theorem:ExactHullRegular}.
\end{proof}

\begin{proposition}\label{Proposition:ClassificationOfPercolatingSubcategoriesOfExactAICategory}
	Let $\EE$ be a deflation-exact category having admissible kernels and let $\AA\subseteq \EE$ be a strictly full additive subcategory. If either $\EE$ is exact, or if $\EE$ is pre-abelian, the following are equivalent:
	\begin{enumerate}
		\item $\AA\subseteq \EE$ is a deflation-percolating subcategory.
		\item $\AA\subseteq \EE$ is a Serre subcategory which is closed under subobjects.
	\end{enumerate}
\end{proposition}

\begin{proof}
	Assume first that $\AA\subseteq \EE$ is a deflation-percolating subcategory. In particular, $\AA$ is a Serre subcategory. Consider a monomorphism $X\stackrel{f}{\hookrightarrow}A$ in $\EE$ with $A\in \AA$. By axiom \ref{P2}, $f$ factors as $X\deflation A' \to A$ with $A'\in \AA$. As $f$ is monic, so is $X\deflation A'$ and hence this map is an isomorphism, thus $X\in \AA$.
	
	Conversely, assume that $\AA$ is a Serre subcategory which is closed under subobjects. Axiom \ref{P1} holds by assumption. Axiom \ref{P2} follows immediately from \Cref{Proposition:InterpretationOfAIForDeflationExactCategories}.\eqref{lemma:InterpretationOfAIForDeflationExactCategories3}. 
	
	We now show axiom \ref{P4}. Let $f\colon X\to Y$ be a map which factors through an object $B\in \AA$ and let $i\colon A\inflation X$ be an inflation such that $f\circ i=0$. We first claim that we may assume $X\to B$ to be a deflation and $B\to Y$ to be a monomorphism. Indeed, by axiom \ref{P2}, the map $X\to B$ factors as $X\deflation B'\to B$ with $B'\in \AA$. By \Cref{Proposition:InterpretationOfAIForDeflationExactCategories}.\eqref{lemma:InterpretationOfAIForDeflationExactCategories3}, we find that the composition $B'\to B\to Y$ factors as $B'\deflation B''\hookrightarrow Y$. By axiom \ref{P1}, $B''\in \AA$ and by axiom \ref{R1}, the composition $X\deflation B'\deflation B''$ is a deflation. This shows the claim.  Let $p\colon X\deflation X'$ be the cokernel of $i\colon A\inflation X$. As $f\circ i=0$, we obtain a factorization $X\deflation X'\to Y$ of $f$. Again, by \Cref{Proposition:InterpretationOfAIForDeflationExactCategories}.\eqref{lemma:InterpretationOfAIForDeflationExactCategories3}, the map $X'\to Y$ factors as $X'\deflation X''\hookrightarrow Y$. By axiom \ref{R1} we obtain the deflation-mono factorization $X\deflation X''\hookrightarrow Y$ of $f$. As deflation-mono factorizations are unique, we conclude that $X''\cong B$ and thus axiom \ref{P4} holds.

	It remains to verify axiom \ref{P3}. If $\EE$ is exact, axiom \ref{P3} is automatic (see \Cref{Remark:P3RedundantForExact}) and there is nothing to prove. Now assume that $\EE$ is pre-abelian. Let $i\colon X\inflation Y$ be an inflation and let $t\colon Y\to T$ be a map such that $t\circ i$ factors as $X\to A\to T$ with $A\in \AA$. By axiom \ref{P2} we may assume that $X\to A$ is a deflation. Write $K\inflation X\deflation A$ for the corresponding conflation. As $\EE$ is pre-abelian, the cokernel $P$ of the composition $K\inflation X\inflation Y$ exists. Hence we obtain the following commutative diagram:
	\begin{equation*}
		\begin{tikzcd}
			K\arrow[equal]{r}\arrow[tail]{d} & K\arrow{d} &\\
			X\arrow[two heads]{d}\arrow[tail]{r} & Y\arrow[two heads]{r}\arrow{d} & Z\arrow[equal]{d}\\
			A\arrow{r} & P\arrow{r} & Z
		\end{tikzcd}
	\end{equation*}
	Axiom \ref{R3}, which is satisfied by \Cref{Proposition:DeflationAICategorySatisfiesAxiomR3+}, implies that $P\to Z$ is a deflation. Write $L\inflation P\deflation Z$ for the corresponding conflation. By \cite[Proposition~3.7]{HenrardvanRoosmalen19a}, the square $XYLP$ is bicartesian. In particular, $\ker(X\to L)\cong M\cong \ker(Y\to P)$. As the map $X\to L$ factors through $A$, we find that $\coim(X\to L)\in \AA$. We write $B=\coim(X\to L)$ and we write $Q=\coim(Y\to P)$. We obtain the following commutative diagram:
	\begin{equation*}
		\begin{tikzcd}
			M \arrow[equal]{r}\arrow[tail]{d}& M\arrow[tail]{d} & \\
			X\arrow[tail]{r}\arrow[two heads]{d} & Y\arrow[two heads]{r}\arrow[two heads]{d} & Z\arrow[equal]{d}\\
			B\arrow[tail]{r}\arrow[hook]{d} & Q\arrow[two heads]{r}\arrow[hook]{d} & Z\arrow[equal]{d}\\
			L\arrow[tail]{r} & P\arrow[two heads]{r} & Z
		\end{tikzcd}
	\end{equation*}
	Here $B\inflation Q\deflation Z$ is a conflation by the Nine Lemma.	It is now straightforward to check that $XYBQ$ is the desired square for axiom \ref{P3}.
\end{proof}

\begin{example}
	Consider the category $\mathsf{LCA}$ of locally compact abelian groups and let $\mathsf{LCA}_{\mathsf{D}}\subseteq \mathsf{LCA}$ be the full subcategory of discrete abelian groups. By \cite{HoffmannSpitzweck07}, $\mathsf{LCA}$ is a quasi-abelian category. As $\mathsf{LCA}_{\mathsf{D}}\subseteq \mathsf{LCA}$ is a Serre subcategory closed under subobjects, \Cref{Proposition:ClassificationOfPercolatingSubcategoriesOfExactAICategory} yields that $\mathsf{LCA}_{\mathsf{D}}\subseteq \mathsf{LCA}$ is a deflation-percolating subcategory. By \Cref{Proposition:QuotientsPreserveTheAIProperty}, the quotient $\mathsf{LCA}/\mathsf{LCA}_{\mathsf{D}}$ is a deflation-exact category having admissible kernels.  In particular, it is an additive regular category.
	
	Furthermore, \cite[Corollary~6.6]{HenrardvanRoosmalen19b} yields that $\mathsf{LCA}/\mathsf{LCA}_{\mathsf{D}}$ is in fact two-sided exact. Thus $\mathsf{LCA}/\mathsf{LCA}_{\mathsf{D}}\simeq \mathsf{LCA}\dqq \mathsf{LCA}_{\mathsf{D}}$. By Pontryagin duality, the quotient $\mathsf{LCA}/\mathsf{LCA}_{\mathsf{C}}$ is an exact category having admissible cokernels. Here, $\mathsf{LCA}_{\mathsf{C}}\subseteq \mathsf{LCA}$ is the full subcategory of compact abelian groups.
\end{example}

\subsection{Admissibly percolating subcategories}

We recall the following special kind of percolating subcategories from \cite{HenrardvanRoosmalen19a}. This type of percolating subcategories will appear in the next section.

\begin{definition}
	Let $\EE$ be a conflation category. An \emph{admissibly deflation-percolating subcategory} is a subcategory $\AA\subseteq \EE$ such that the following axioms hold:
	\begin{enumerate}[label=\textbf{A\arabic*},start=1]
		\item\label{A1} $\AA$ is a \emph{Serre subcategory} (see \Cref{definition:GeneralPercolatingSubcategory}).
		\item\label{A2} Every morphism $X\to A$ with $A\in \AA$ is admissible with image in $\AA$ (that is, the morphism $f\colon X \to A$ has a deflation-inflation factorization $X \deflation A' \inflation A$ with $A' \in \AA.$)
		\item\label{A3} If $f\colon X\deflation A$ is a deflation with $A\in \AA$ and $g\colon X\inflation Y$ is an inflation, the pushout of $f$ along $g$ exists and is of the following form:
\begin{equation*}
			\begin{tikzcd}
				X\arrow[tail]{r}{g}\arrow[two heads]{d}{f} & Y\arrow[two heads]{d}{f'}\\
				A\arrow[tail]{r}{i'} & P
				\end{tikzcd}
		\end{equation*}
	\end{enumerate}
	An \emph{admissibly inflation-percolating subcategory} is defined dually. A \emph{two-sided admissibly percolating subcategory} is both admissibly inflation-percolating and admissibly deflation-percolating.
\end{definition}

\begin{remark}\label{Remark:AnyExactCategoryYieldsA3}
	For an exact category $\EE$, any subcategory $\AA$ satisfies axiom \ref{A3} (see the dual of \cite[Proposition~2.15]{Buhler10}). 
\end{remark}

\begin{example}
	Given a filtered ring $FR$, one can consider a type of filtered representation theory called \emph{glider representations} as in \cite{CaenepeelVanOystaeyen19book}. The category $\Glid(FR)$ of glider representations is obtained as a quotient of the quasi-abelian category $\Preglid(FR)$ of pregliders by the subcategory $\Mod(R)$ (see \cite{HenrardvanRoosmalen20Glider}). Here, the subcategory $\Mod(R)\subseteq \Glid(FR)$ is an admissibly deflation-percolating subcategory. It follows that $\Glid(FR)$ is a deflation-exact category having admissible kernels.
	
	Following \cite{HenrardvanRoosmalen20Glider}, there is an embedding $\Glid(FR)\to \Mod(\FF R)$ of $\Glid(FR)$ into an abelian category $\Mod(\FF R)$ which reflects kernels and lifts to a derived equivalence (here, $\FF R$ is the filtered companion category, see \cite[Definition~3.1]{HenrardvanRoosmalen20Glider}).  It follows that this lift restricts to an equivalence on the left hearts, i.e.~$\mathcal{LH}(\Glid(FR))\simeq \Mod(\FF R)$. This recovers and generalizes \cite[Theorem~4.20]{SchapiraSchneiders16}.
\end{example}

The following proposition explains the terminology (see \cite[Section~6]{HenrardvanRoosmalen19a}).

\begin{proposition}\label{proposition:AboutStrictlyPercolating}
	Let $\EE$ be a deflation-exact category and let $\AA\subseteq \EE$ be an admissibly deflation-percolating subcategory. The following properties hold.
	\begin{enumerate}
		\item The category $\AA$ is abelian and is a deflation-percolating subcategory of $\EE$.
		\item The weak $\AA$-isomorphisms are precisely the admissible morphisms $f\in \Mor(\EE)$ with $\ker f, \coker f \in \AA$.
		\item The set $S_{\AA}$ of weak isomorphisms satisfies the 2-out-of-3-property and is saturated.
	\end{enumerate}
\end{proposition}

We conclude this section by recalling two useful properties of two-sided admissibly percolating subcategories of an exact category.

\begin{theorem}[\protect{\cite[Theorem~2.16]{HenrardKvammevanRoosmalen20}}]\label{Theorem:QReflectsAdmissibles}
	Let $\EE$ be an exact category and let $\AA\subseteq \EE$ be a two-sided admissibly percolating subcategory. A map $f\colon X\to Y$ is admissible in $\EE$ if and only if $Q(f)$ is admissible in $\EE/\AA$. In other words, the exact localization functor $Q$ reflects admissible morphisms. 
\end{theorem}

We end this section with a criterion for percolating subcategories using the language of torsion pairs in a conflation category, which is a direct adaptation from the abelian \cite{Dickson66}, the exact \cite{HenrardvanRoosmalen19a,Tattar21}, the extriangulated \cite{HeHuZhou22}, and the homological \cite{BournGran06} setting.

\begin{definition}\label{definition:TorsionPair}
	Let $\CC$ be a conflation category.  A \emph{torsion pair} or a \emph{torsion theory} is a pair $(\TT,\FF)$ of full and replete subcategories of $\CC$ such that
	\begin{enumerate}
		\item $\Hom(T,F)=0$ for all $T\in \TT$ and $F\in \FF$,
		\item every object $M\in \EE$ fits into a conflation $T\inflation M\deflation F$ with $T\in \TT$ and $F\in \FF$.
	\end{enumerate}
	A torsion pair $(\TT, \FF)$ is said to be \emph{hereditary} if $\TT$ is closed under subobjects.
\end{definition}

\begin{proposition}[\protect{\cite[Proposition~2.22]{HenrardKvammevanRoosmalen20}}]\label{Proposition:Torsion+A2IsTwoSidedPercolating}
Let $\EE$ be an exact category with a torsion pair $(\TT,\FF)$. If $\TT\subseteq \EE$ satisfies axiom \ref{A2}, then the subcategory $\TT\subseteq \EE$ is two-sided admissibly percolating.
\end{proposition}
\section{Constructions using Auslander's formula}\label{Section:ConstructionOfTheLeftHeart}

Throughout this section, let $\EE$ be a deflation-exact category with admissible kernels.  Auslander's formula (see \cite[p. 1]{Lenzing98} or \cite[Theorem~2.2]{Krause15}) states that any small abelian category $\AA$ can be recovered as the quotient $\mod (\AA) / \eff (\AA)$.  The description of the left heart given in \Cref{Theorem:AlternativeDescriptionOfLeftHeart} as $\mathcal{LH}(\EE) \simeq \mod(\EE) / \eff(\EE)$ has the same flavor.  In this section, we consider two subcategories, $\modinf(\EE)$ and $\modad(\EE)$, and consider similar quotients by the subcategory of the effaceable functors.  In addition, we show that the effaceable functors form a torsion subcategory of these categories.

More specifically, we show that $\eff \EE$ is a torsion class in $\mod \EE$; the corresponding torsionfree class is the full subcategory $\modone(\EE)$ consisting of all modules of projective dimension at most one. This observation will play a part in the next section.

The main idea is the following.  By \Cref{proposition:TorsionPart}, the deflation-mono factorization $f = m \circ p$ of a morphism $f$ in $\EE$ gives rise to a short exact sequence $0 \to \coker \Yoneda(p) \to \coker \Yoneda(f) \to \coker \Yoneda(m) \to 0$ in $\mod(\EE).$  It will follow that $\Yoneda(m) \in \modone(\EE)$ and $\Yoneda(p) \in \eff(\EE)$, so that this sequence gives the required decomposition of $\Yoneda (f)$ into a torsion submodule and a torsionfree quotient module.

We also identify two interesting subcategories of $\mod \EE$ by imposing further conditions on the presenting morphism $f = m \circ p$: we consider $\smodad \EE$ of objects of the form $\coker \Yoneda(f)$ where $m$ is an inflation in $\EE$, and the subcategory $\smodinf \EE$ consisting of those objects of the form $\coker \Yoneda(f)$ where $m$ is the composition of inflations in $\EE$.

The torsion theory $(\eff(\EE), \modone(\EE))$ in $\smod \EE$ then induces a torsion theory $(\eff(\EE), \modadone(\EE))$ in $\smodad \EE$ and a torsion theory $(\eff(\EE), \modinfone(\EE))$ in $\modinf \EE$.

Finally, the categories $\modad \EE$ and $\modinf \EE$ are not abelian, but one can nonetheless consider their quotients by the subcategory of $\eff \EE$ of effaceable functors.   By taking these quotients, the inclusions $\modad \EE \subseteq \modinf \EE \subseteq \mod \EE$ give a sequence $\EE \subseteq \ex{\EE} \subseteq \mathcal{LH}(\EE).$

\subsection{Preparatory notions} We start by formally introducing the categories $\modad(\EE)$ and $\modinf(\EE)$ mentioned before.

%


\begin{definition}\label{definition:VariousMods}
	\begin{enumerate}
	  \item A morphism in $\EE$ which is the composition of a finite string of inflations is called a \emph{weak inflation}.  A morphism $X \to Y$ is called a \emph{weakly admissible} morphism if it is the composition of a deflation $X \deflation Z$ and a weak inflation $Z \to Y.$
		\item We write $\smodad(\EE)$ for the full subcategory of $\smod(\EE)$ consisting of those functors $F\cong \coker(\Upsilon(f))$ where $f$ is admissible in $\EE$.  We write $\modadone(\EE)$ for the full subcategory of $\smodad(\EE)$ consisting of those functors $F\cong \coker(\Upsilon(f))$ where $f$ is an inflation.
		\item We write $\smodinf(\EE)$ for the full subcategory of $\smod(\EE)$ consisting of those functors $F\cong \coker(\Upsilon(f))$ such that $f$ is weakly admissible in $\EE$.  We write $\modoneinf(\EE)$ for the full subcategory of $\smodinf(\EE)$ consisting of those functors $F\cong \coker(\Upsilon(f))$ where $f$ is a weak inflation.
	\end{enumerate}
\end{definition}

\begin{remark}
We have $\eff(\EE) \subseteq \smodad(\EE) \subseteq \smodadinf(\EE) \subseteq \smod(\EE).$
\end{remark}

The following proposition explains the notation of $\modadone(\EE)$ and $\modoneinf(\EE)$.

\begin{proposition}\label{remark:AboutGlobalDimension}
For any deflation-exact category $\EE$ with admissible kernels, we have:
\begin{enumerate}
  \item $\modadone(\EE) = \modone(\EE) \cap \modad(\EE)$,
  \item $\modoneinf(\EE) = \modone(\EE) \cap \smodinf(\EE)$.
\end{enumerate}
\end{proposition}

\begin{proof}
We only show the first statement, the proof of the second statement is similar.  We start by showing the inclusion $\modadone(\EE) \subseteq \modone(\EE) \cap \modad(\EE)$.  Let $M \in \modadone(\EE)$, say $M \cong \coker \Yoneda(f)$ for an inflation $f$ in $\EE$.  As $f$ is a monomorphism, we have $M \in \modone(\EE)$ by \Cref{proposition:GlobalDimensionAtMostOne} and, as $f$ is admissible, we have $M \in \modone(\EE)$.

For the inclusion $\modone(\EE) \cap \modad(\EE) \subseteq \modadone(\EE)$, let $M \in \modad(\EE)$, say $M \cong \coker \Yoneda(f)$ for an admissible $f = m \circ p$ in $\EE$.  Here, $p$ is a deflation and $m$ is an inflation.  Since $M \in \modone(\EE)$, we know, by \Cref{proposition:GlobalDimensionAtMostOne} and the uniqueness of a deflation-inflation factorization, that $p$ is a retraction.  Hence $\coker \Yoneda(f) = \coker \Yoneda(m) \in \modadone(\EE).$
\end{proof}

The following lemma is an adaptation of \cite[Lemma~3.27]{HenrardKvammevanRoosmalen20}.

\begin{lemma}\label{Lemma:MorphismsRepresentingAnEffaceableAreDeflations}
Let $G \cong \coker \Yoneda(g) \in \smod \EE$ for some morphism $g\colon B \to C$ in $\EE$.  If $G \in \eff(\EE)$, then:
\begin{enumerate}
	\item there is a $B' \in \EE$ such that $\begin{pmatrix} g & 0\end{pmatrix}\colon B \oplus B' \deflation C$ is a deflation,
	\item if $\EE$ satisfies axiom \ref{R3}, then $g\colon B \deflation C$ is a deflation.
\end{enumerate}
\end{lemma}

\begin{proof}
As $G\in \eff(\EE)$, there is a conflation $X\inflation Y\stackrel{f}{\deflation}Z$ such that $G\cong\coker(\Upsilon(f))$. By the Comparison Theorem (\cite[Theorem 12.4]{Buhler10}) and the fact that the Yoneda embedding is fully faithful, the sequence $0\to\ker(g)\to B\to C\to 0$ is homotopy equivalent to the acyclic sequence $0\to X\inflation Y\deflation Z\to 0$.

If $\EE$ satisfies axiom \ref{R3}, then by \Cref{Proposition:BasicPropertiesDerivedCategoryNew}, we find that $g\colon B\to C$ is a deflation.  Without the assumption of axiom \ref{R3}, we may still enlarge the conflation structure on $\EE$ until it satisfies axiom \ref{R3} (as $\EE$ is weakly idempotent complete, we obtain the closure under axiom \ref{R3} as the closure under axiom $\mathbf{\text{R3}^-}$, see \cite[Proposition~3.3 and Corollary~7.14]{HenrardvanRoosmalen20Obscure}).  We can now use \Cref{Proposition:BasicPropertiesDerivedCategoryNew} to see that $0\to\ker(g)\to B\to C\to 0$ is a conflation in the new conflation structure.  By \cite[Proposition~7.18]{HenrardvanRoosmalen20Obscure}, we find that there exists an object $B' \in B$ such that $0\to\ker(g) \oplus B'\inflation B \oplus B'\deflation C\to 0$ is a conflation in the original conflation structure.
\end{proof}

\subsection{A torsion theory for exact categories}

Let $\FF$ be an exact category with admissible kernels.  Our main example will be $\FF = \ex\EE$ where $\EE$ is a deflation-exact category with admissible kernels.  We show that the subcategory $\eff(\FF)$ of effaceable functors is a torsion class in $\mod(\FF)$.  This serves as a starting point for the other torsion pairs we will give in this section.

\begin{proposition}\label{proposition:HereditaryTorsionForExact}
Let $\FF$ be an exact category with admissible kernels.  There is a hereditary torsion pair $(\eff(\FF), \modone(\FF))$ on $\mod({\FF})$.
\end{proposition}

\begin{proof}
It follows from \Cref{proposition:EffaceableSerre} that $\eff(\FF)$ is a Serre subcategory of $\mod(\FF)$.  \Cref{proposition:TorsionPart} shows that every $M \in \mod({\FF})$ is the extension of a torsion-free module by a torsion module.  We only need to show that $\Hom(\eff(\FF), \modone(\FF))=0.$  For this, consider a morphism $\eta\colon F \to G$ with $F \in \eff(\FF)$ and $G \in \modone(\FF)$.  We can choose projective presentations of $F$ and $G$ as follows:
\[
\mbox{
$\xymatrix@1{\Yoneda(A) \ar[r]^{\Yoneda(f)} & \Yoneda(B) \ar[r] & F \ar[r] & 0} \quad$ and
$\quad\xymatrix@1{\Yoneda(C) \ar[r]^{\Yoneda(g)} & \Yoneda(D) \ar[r] & G \ar[r] & 0}$}
\]
where $f\colon A \deflation B$ is a deflation and $g\colon C \hookrightarrow D$ is a monomorphism.  A morphism $\eta\colon F \to G$ lifts to a commutative diagram
	\[\begin{tikzcd}
		A\arrow{r}{\beta}\arrow[two heads]{d}{f} & C\arrow[hook]{d}{g}\\
		B\arrow{r}{\alpha} & D
	\end{tikzcd}\] 
Note that $\EE$ admits all pullbacks as it has admissible kernels. Using the notation of \Cref{Lemma:TheFamousDiagramChase}, we find that $g'$ is an isomorphism as it is both a monomorphism (as pullback of a monomorphism) and a deflation (by applying axiom \ref{R3+} to the composition $f = g' \circ \beta''$). It follows that $\im(\eta) = 0$ and hence that $\Hom(F,G) = 0$.
\end{proof}

\subsection{A torsion theory on \texorpdfstring{$\mod \EE$}{mod(E)}}

By \Cref{proposition:HereditaryTorsionForExact}, we know there is a hereditary torsion theory $(\eff(\ex\EE), \modone(\ex\EE))$ on $\mod({\ex\EE})$.  We intersect this torsion theory with $\mod(\EE) \subseteq \mod(\ex\EE)$ to find a torsion theory on $\mod(\EE)$.

\begin{remark}
For any $\EE$ deflation-exact category with admissible kernels.  As $j\colon \EE \to \ex\EE$ commutes with kernels, the natural fully faithful functor $-\otimes_\EE \ex\EE\colon \mod(\EE) \to \mod(\ex\EE)$ is exact (see \cite[Lemma~2.6.(2)]{Krause98}).
\end{remark}

\begin{lemma}\label{lemma:IntersectingWithMod}
Let $\EE$ be a deflation-exact category with admissible kernels.  We have
\begin{enumerate}
	\item $\eff(\EE) = \eff(\ex\EE) \cap \mod(\EE)$, and
	\item $\modone(\EE) = \modone(\ex\EE) \cap \mod(\EE)$.
\end{enumerate}
\end{lemma}

\begin{proof}
\begin{enumerate}
	\item The inclusion $\eff(\EE) \subseteq \eff(\ex\EE) \cap \mod(\EE)$ uses only that $j\colon \EE \to \ex\EE$ is conflation-exact.  For the inclusion $\eff(\ex\EE) \cap \mod(\EE) \subseteq \eff(\EE)$, consider an object $M \cong \coker \Yoneda(f) \in \mod(\EE)$ where $f$ is a morphism in $\EE$.  By \Cref{Lemma:MorphismsRepresentingAnEffaceableAreDeflations}, we know that $f$ is a deflation in $\ex\EE.$  Hence, by \Cref{Corollary:TheHullInheritsAdmissibleKernels}.\eqref{enumerate:TheHullInheritsAdmissibleKernels4}, we know that $f$ is a deflation in $\EE$.
	\item Again, the inclusion $\modone(\EE) \subseteq \modone(\ex\EE) \cap \mod(\EE)$ uses only that $j\colon \EE \to \ex\EE$ preserves monomorphisms (see \Cref{Corollary:TheHullInheritsAdmissibleKernels}.\eqref{enumerate:TheHullInheritsAdmissibleKernels3}).  For the other inclusion, let $M \in \modone(\ex\EE) \cap \mod(\EE)$.  As $M \in \mod(\EE)$, we know that $M\cong \coker \Yoneda(f)$ for a morphism $f$ in $\EE \subseteq \ex\EE$.  It follows from \Cref{proposition:GlobalDimensionAtMostOne} that $f = m \circ p$ in $\ex\EE$ where $p$ is a retraction and $m$ a monomorphism.  As $\EE$ is closed under subobjects in $\ex\EE$, we find that the factorization $f = m \circ p$ also holds in $\EE$.  The statement now follows from the isomorphism $\coker \Yoneda(f) \cong \coker \Yoneda(m)$. \qedhere
\end{enumerate}
\end{proof}

\begin{proposition}\label{proposition:HereditaryTorsionForExactHull}
Let $\EE$ be a deflation-exact category with admissible kernels.  There is a hereditary torsion pair $(\eff(\EE), \modone(\EE))$ on $\mod(\EE)$.
\end{proposition}

\begin{proof}
It follows from \Cref{proposition:TorsionPart} that every object in $\mod(\EE)$ is the extension of a torsion-free object by a torsion object.  The other properties (i.e.~that $\eff(\EE)$ is a Serre subcategory and that $\Hom(\eff(\EE), \modone(\EE)) = 0$) follow easily from $\eff(\EE) = \eff(\ex\EE) \cap \mod(\EE)$ and $\modone(\EE) = \modone(\ex\EE) \cap \mod(\EE)$.
\end{proof}

\subsection{The exact hull and a torsion theory \texorpdfstring{on $\modinf(\EE)$}{}}  In this subsection, we consider the category $\modinf(\EE)$, given as the full subcategory of $\mod(\EE)$ consisting of those objects with are presentable by a weak admissible morphism in $\EE$.  We will show in \Cref{Proposition:SmodinfIsExact} that this is an extension-closed subcategory, and hence exact.  Next, we show that $(\eff(\EE), \modinfone(\EE))$ is a torsion theory on $\modinf(\EE)$ by intersecting the torsion theory from \Cref{proposition:HereditaryTorsionForExact} with $\modinf(\EE)$.  Finally, we consider the quotient $\modinf(\EE) / \eff(\EE)$ and show that it is equivalent to the exact hull $\ex\EE$.

\begin{lemma}\label{lemma:IntersectingWithModinf}
Let $\EE$ be a deflation-exact category with admissible kernels.  We have
\begin{enumerate}
	\item $\eff(\EE) = \eff(\ex\EE) \cap \modinf(\EE)$, and
	\item $\modinfone(\EE) = \modone(\ex\EE) \cap \modinf(\EE)$.
\end{enumerate}
\end{lemma}

\begin{proof}
The same argument as in \Cref{lemma:IntersectingWithMod} works.
\end{proof}

In addition, we have the following description of $\modinf(\EE).$

\begin{lemma}\label{lemma:IntersectingModadEE}
Let $\EE$ be a deflation-exact category with admissible kernels.  We have $\modinf(\EE) = \modad(\ex\EE) \cap \mod(\EE)$.
\end{lemma}

\begin{proof}
The inclusion $\modinf(\EE) \subseteq \modad(\ex\EE) \cap \mod(\EE)$ only uses that $j\colon \EE \to \ex\EE$ is conflation-exact and that $\ex\EE$ satsifies axiom \ref{L1}.  For the other inclusion, take $M\in \smod(\EE)\cap\smodad(\ex\EE)$. As $M\in \smodad(\EE^{\mathsf{ex}})$, we can write $M\cong \coker(\Upsilon(f))$ where $f\in \Hom_{\EE^{\mathsf{ex}}}(X,Y)$ is an admissible morphism in $\EE^{\mathsf{ex}}$. As $M\in \smod(\EE)$, we can write $F\cong \coker(\Upsilon(g))$ for some $g\in \Hom_{\EE}(U,V)$. By the Comparison Theorem (see \cite[Theorem 12.4]{Buhler10}), the two resolutions of $M$ are homotopy equivalent in $\smod(\ex{\EE})$.  As the Yoneda embedding is fully faithful and left exact, we obtain the following commutative diagram in $\EE^{\mathsf{ex}}$ which defines a homotopy equivalence between the rows:
\[\begin{tikzcd}
\ker f \arrow[r] \arrow[d] & X \arrow[r, "f"] \arrow[d] & Y \arrow[d] \arrow[r] & \coker (f) \arrow[d] \\
\ker g \arrow[r] & U \arrow[r, "g"] & V \arrow[r] & \coker (f) \\
\end{tikzcd}\]
As the lower row is acyclic in $\EE^{\mathsf{ex}}$, \cite[Proposition~10.14]{Buhler10} (or \Cref{Proposition:BasicPropertiesDerivedCategoryNew}) implies that the upper row is acyclic in $\EE^{\mathsf{ex}}$ as well (this uses that $\ex{\EE}$ has kernels, see \Cref{theorem:ExactHullRegular} and hence is weakly idempotent complete).  In particular, $g$ is an admissible morphism in $\EE^{\mathsf{ex}}$.  As $\EE \subseteq\EE^{\mathsf{ex}}$ is closed under subobjects (see \Cref{Corollary:TheHullInheritsAdmissibleKernels}), one sees that $\ker(g), \coim(g)\in \EE$. By \Cref{Corollary:HullOfAIIsL1Closure}, the map $\coim(g)\to V$ is a finite composition of inflations in $\EE$. This shows that $M\cong \coker(\Upsilon(g))\in \smodinf(\EE)$.
\end{proof}

\begin{proposition}\label{Proposition:SmodinfIsExact}
	The category $\smodinf(\EE)$ lies extension-closed in $\smod(\EE)$. In particular, $\smodinf(\EE)$ is an exact category.
\end{proposition}

\begin{proof}
	By \cite[Proposition~3.5]{HenrardKvammevanRoosmalen20}, the category $\smodad(\ex\EE)$ lies extension-closed in $\Mod(\ex\EE)$.  The statement now follows from the equality $\modinf(\EE) = \modad(\ex\EE) \cap \mod(\EE)$ (see \Cref{lemma:IntersectingModadEE}).
\end{proof}

\begin{proposition}\label{Proposition:TorsionTheoryInSmodinf}
	The pair $(\eff(\EE),\FFinf(\EE))$ is a hereditary torsion pair in $\smodinf(\EE)$. The category $\eff(\EE)$ is a two-sided admissibly percolating subcategory of $\smodinf(\EE)$.
\end{proposition}

\begin{proof}
It follows from \Cref{proposition:TorsionPart} that every object in $\modinf(\EE)$ is the extension of an object in $\FFinf(\EE)$ by an object in $\eff(\EE)$.  The other properties of a torsion pair follow easily from combining \Cref {proposition:HereditaryTorsionForExact} and \Cref{lemma:IntersectingWithModinf} (taking $\FF = \ex\EE$).

	To show that $\eff(\EE)$ is a two-sided percolating subcategory of $\modinf(\EE)$, it suffices to check that it satisfies axiom \ref{A2} (see \Cref{Proposition:Torsion+A2IsTwoSidedPercolating}).  To that end, consider a map $\eta\colon F\to G$ in $\smodinf(\EE)$ with $G\in \eff(\EE)$. By \cite[Proposition~3.6]{HenrardKvammevanRoosmalen20}, $\eff(\EE^{\mathsf{ex}})\subseteq \smodad(\EE^{\mathsf{ex}})$ satisfies axiom \ref{A2}. Hence we obtain a sequence 
	\[\ker(\eta)\inflation F\deflation \im(\eta)\inflation G\deflation \coker(\eta)\]
in $\smodad(\EE^{\mathsf{ex}})$ with $\coker(\eta),\im(\eta)\in \eff(\EE^{\mathsf{ex}})$. Note that $\coker(\eta)\in \smod(\EE)$ as $\smod(\EE)$ is closed under cokernels and hence, by \Cref{lemma:IntersectingModadEE}, we know that $\coker(\eta)\in \modinf(\EE)$. As $\smod(\EE)$ is closed under kernels in $\mod(\ex\EE)$, we find that $\im(\eta),\ker(\eta)\in \smod(\EE) \cap \eff(\ex\EE)$ and hence, by \Cref{lemma:IntersectingWithMod}, $\im(\eta),\ker(\eta)\in \eff(\EE)$. This shows axiom \ref{A2}.  By \Cref{Proposition:Torsion+A2IsTwoSidedPercolating}, we know that $\eff(\EE)$ is a two-sided admissibly percolating subcategory of $\modinf(\EE)$.  Specifically, we know that that $\eff(\EE)$ is a Serre subcategory of $\modinf(\EE)$ and thus $(\eff(\EE),\FFinf(\EE))$ is a hereditary torsion theory. This concludes the proof.
\end{proof}

\begin{corollary}\label{Corollary:HullAsQuotient}
	The quotient $\smodinf(\EE)/\eff(\EE)$ is an exact category. Moreover, the exact categories $\smodinf(\EE)/\eff(\EE)$ and $\EE^{\mathsf{ex}}$ are equivalent.
\end{corollary}

\begin{proof}
	By \Cref{Proposition:TorsionTheoryInSmodinf}, \Cref{Theorem:LocalizationTheorem} and its dual, $\smodinf(\EE)/\eff(\EE)$ is an exact category. We write $Q\colon \smodinf(\EE)\to \smodinf(\EE)/\eff(\EE)$ for the corresponding localization functor.
	
	We first claim that the composition $\EE \stackrel{\Upsilon}{\rightarrow}\smodinf(\EE)\stackrel{Q}{\rightarrow}\smodinf(\EE)/\eff(\EE)$ is a conflation-exact functor. To that end, let $X\stackrel{f}{\inflation}Y\stackrel{g}{\deflation}Z$ be a conflation in $\EE$. As the Yoneda embedding is left exact, we obtain an exact sequence $\Upsilon(X)\inflation \Upsilon(Y) \to \Upsilon(Z)\deflation \coker(\Upsilon(g))$ in $\smodinf(\EE)$. Applying $Q$ to this sequence, we obtain the conflation $Q\Upsilon(X)\inflation Q\Upsilon(Y) \deflation Q\Upsilon(Z)$ as $\coker(\Upsilon(g))\in \eff(\EE)$. This shows that $Q\circ \Upsilon$ is conflation-exact.
	
	We show that $Q\circ \Upsilon$ satisfies the universal property of $j\colon \EE\to \EE^{\mathsf{ex}}$ and thus the desired equivalence. Let $\Phi\colon \EE\to \FF$ be a conflation-exact functor to an exact category $\FF$. We construct an exact functor $\overline{\Phi}\colon \smodinf(\EE)\to \FF$ as follows. By the universal property of the exact hull, there is a unique (up to isomorphism) exact functor $\Phi^{\mathsf{ex}}\colon \EE^{\mathsf{ex}}\to \FF$ such that $\Phi^{\mathsf{ex}}\circ j=\Phi$. By \cite[Theorem~3.9]{HenrardKvammevanRoosmalen20}, there is a unique (up to isomorphism) exact functor $\overline{\Phi^{\mathsf{ex}}}\colon \smodad(\EE^{\mathsf{ex}})\to \FF$ such that $\overline{\Phi^{\mathsf{ex}}}\circ \Upsilon^{\mathsf{ex}}=\Phi^{\mathsf{ex}}$.  Here, $\Upsilon^{\mathsf{ex}}\colon \EE^{\mathsf{ex}} \to \smod(\EE^{\mathsf{ex}})$ is the Yoneda embedding of $\EE^{\mathsf{ex}}$.  Clearly $\overline{\Phi^{\mathsf{ex}}}(\eff(\EE^{\mathsf{ex}})) \cong 0$. The restriction of $\overline{\Phi^{\mathsf{ex}}}$ to $\smodinf(\EE)$ is still an exact functor and maps $\eff(\EE)$ to zero. Hence this functor further factors through $Q\colon \smodinf(\EE)\to \smodinf(\EE)/\eff(\EE)$ via an exact functor $\overline{\Phi}\colon \smodinf(\EE)/\eff(\EE)\to \FF$ as required. This concludes the proof.
\end{proof}

\subsection{One-sided Auslander's formula} Let $\EE$ be a deflation-exact category with admissible kernels.  We start with the following straightforward observation.

\begin{lemma}\label{Lemma:AdmissiblesStableUnderPullbacks}
\begin{enumerate}
	\item Inflation and admissible morphisms are stable under pullbacks in $\EE.$
	\item Weak inflation and weak admissible morphisms are stable under pullbacks in $\EE.$
\end{enumerate}
\end{lemma}

\begin{proof}
Since kernels are stable under pullbacks and $\EE$ has admissible kernels, the pullback of an inflation is an inflation.  That weak inflations are stable under pullbacks follows from the first statement together with the Pullback Lemma.  That (weak) admissible morphisms are stable under pullbacks then follows from the Pullback Lemma, axiom \ref{R2}, and the first statement. 
\end{proof}

\begin{proposition}\label{Proposition:smodad(EE)IsDeflationExactWithAdmissibleKernels}
The subcategory $\smodad(\EE)$ of $\smod(\EE)$ is closed under subobjects. In particular, $\smodad(\EE)$ inherits a deflation-exact structure having admissible kernels.
\end{proposition}

\begin{proof}
	Let $\eta\colon F\hookrightarrow G$ be a monomorphism in $\smod(\EE)$ and assume that $G\in \smodad(\EE)$. Let $f\colon A\to B$ and $g\colon C\to D$ be morphisms in $\EE$ such that $F\cong \coker(\Upsilon(f))$ and $G\cong \coker(\Upsilon(g))$. We may assume that $g$ is admissible in $\EE$. The map $\eta\colon F\hookrightarrow G$ induces a commutative square 
	\[\xymatrix{
		A\ar[r]\ar[d]^{f} & C\ar[d]^g\\
		B\ar[r]^{h} & D
	}\] in $\EE$. By \Cref{Lemma:TheFamousDiagramChase}, $F\cong \coker \Upsilon(f')$ where $f'$ is the pullback of $g$ along $h$. The result now follows from \Cref{Lemma:AdmissiblesStableUnderPullbacks}.
\end{proof}

%

\begin{lemma}\label{Lemma:ExtensionOfEffaceableAndSmodadIsSmodad}
	Let $E\inflation P \deflation H$ be a short exact sequence in $\smod(\EE)$. If $E\in \eff(\EE)$ and $H\in \smodad(\EE)$, then $P\in \smodad(\EE)$.
\end{lemma}

\begin{proof}
	Let $p\colon A\to B$ and $h\colon C\to D$ be maps in $\EE$ such that $P\cong \coker(\Upsilon(p))$, $H\cong \coker(\Upsilon(h))$ and $h$ is admissible. By \Cref{Lemma:TheFamousDiagramChase}, the map $P \deflation H$ induces the following commutative diagram in $\EE$:
	\[\xymatrix{
		\ker(h)\oplus A\ar@{->>}[r]^-{\begin{psmallmatrix}0&1\end{psmallmatrix}}\ar[d]^{\alpha} & A\ar[r]\ar[d]^{p} & Q\ar[r]\ar[d]^{h'} & C\ar[d]^{h}\\
		Q\ar[r]^{h'} & B\ar@{=}[r] & B\ar[r] & D
	}\]
	where the right square is a pullback square and $E\cong \coker(\Upsilon(\alpha))$. As $E \in \eff(\EE)$, \Cref{Lemma:MorphismsRepresentingAnEffaceableAreDeflations} shows that $\alpha$ is a deflation. Since $h'$ is obtained from the admissible morphism $h$ via a pullback, $h'$ itself is admissible (see \Cref{Lemma:AdmissiblesStableUnderPullbacks}).  Hence, using axiom \ref{R1}, we see that $h' \circ \alpha$ is admissible and hence so is the composition $\ker(h)\oplus A\xrightarrow{\begin{psmallmatrix}0&1\end{psmallmatrix}} A\xrightarrow{p} B$.  Hence, this composition is equal to a composition $ker(h)\oplus A\xrightarrow{[0,p']}B'\xrightarrow{p''} B$ where $[0,p']$ is a deflation and $p''$ is an inflation.  Since $\EE$ satisfies axiom \ref{R3+} by \Cref{Proposition:DeflationAICategorySatisfiesAxiomR3+}, it follows from \cite[Theorem 1.2]{HenrardvanRoosmalen20Obscure} that $p'\colon A\to B'$ is a deflation. Since $p=p'\circ p''$, this shows that $p$ is admissible.
\end{proof}

\begin{proposition}\label{Proposition:TorisionPairInModad}
	The pair $(\eff(\EE),\modadone(\EE))$ defines a torsion pair in $\smodad(\EE)$ and $\eff(\EE)$ is an admissibly deflation-percolating subcategory of $\smodad(\EE)$.
\end{proposition}

\begin{proof}
Since $\modadone(\EE) \subseteq \modinfone(\EE)$, it follows from \Cref{Proposition:TorsionTheoryInSmodinf} that $\Hom(\eff(\EE), \modadone(\EE)) = 0.$  The existence of a torsion/torsion-free sequence follows again from \Cref{proposition:TorsionPart}.

It remains to show that $\eff(\EE)\subseteq \smodad(\EE)$ is an admissibly deflation-percolating subcategory. Axiom \ref{A1} follows directly from \Cref{Proposition:TorsionTheoryInSmodinf}.  For axiom \ref{A2}, consider a morphism $f\colon F \to E$ with $F \in \smodad(\EE)$ and $E \in \eff(\EE).$  As $\eff(\EE)$ satisfies axiom \ref{A2} in $\smodinf(\EE)$, it suffices to show that $\ker f \in \smodad(\EE).$  This is automatic since $\smodad(\EE)$ is closed under subobjects in $\mod(\EE)$, see \Cref{Proposition:smodad(EE)IsDeflationExactWithAdmissibleKernels}.

It remains to show axiom \ref{A3}. To that end, consider a conflation $F\inflation G\deflation H$ in $\smodad(\EE)$ and a map $F\deflation E$ with $E\in \eff(\EE)$. We obtain the following commutative diagram in $\Mod(\EE)$
	\[\xymatrix{
		F\ar@{>->}[r]\ar@{->>}[d] & G\ar@{->>}[r]\ar@{>->}[d] & H\ar@{=}[d]\\
		E\ar@{>->}[r] & P\ar@{->>}[r] & H
	}\] where the left square is a pushout. By \Cref{Lemma:ExtensionOfEffaceableAndSmodadIsSmodad}, $P\in \smodad(\EE)$ as required. This completes the proof.
\end{proof}

\begin{theorem}\label{Theorem:UniversalPropertyDeflationExactsmodad}
	\begin{enumerate}
		\item The Yoneda embedding $\Upsilon\colon \EE\to \smodad(\EE)$ is left conflation-exact (see \Cref{definition:Conflation}) and maps admissible morphisms to admissible morphisms;
		\item If $\FF$ is a deflation-exact category having admissible kernels and $\Phi\colon \EE\to \FF$ is a left exact functor that preserves admissible morphisms, then there exists a functor $\overline{\Phi}\colon \smodad(\EE)\to \FF$, unique up to isomorphism, which is exact and satisfies $\overline{\Phi}\circ \Upsilon=\Phi$.
	\end{enumerate} 
\end{theorem}

\begin{proof}
We show that the Yoneda embedding $\Upsilon\colon \EE\to \smodad(\EE)$ maps admissible morphisms to admissible morphisms; the remainder of the proof is then a straightforward adaptation of \cite[Theorem~3.9]{HenrardKvammevanRoosmalen20}.  Let $f\colon X \to Y$ be any admissible morphism in $\EE$, and let $k\colon K \inflation X$ be the kernel.  Using \Cref{Lemma:TheFamousDiagramChase}, starting from the commutative square
\[
\xymatrix{
0 \ar[r] \ar[d] & 0\ar[d] \\
X \ar[r]^f & Y}\]
gives a diagram 
\[\xymatrix{0 \ar[r] & \Yoneda(K) \ar[r] &\Yoneda(X) \ar[rr]^{\Yoneda(f)} \ar[rd] && \Yoneda(Y) \ar[r] & \coker \Yoneda(f) \ar[r] & 0 \\
&&&\coker \Yoneda(k) \ar[ru]&}\]
in $\Mod \EE.$  As the objects in this diagram all lie in $\smodad(\EE)$, we see that the morphism $\Yoneda(f)\colon \Yoneda(X) \to \Yoneda(Y)$ factors as $\Yoneda(X) \deflation \coker \Yoneda(k) \inflation \Yoneda(Y)$.  This establishes that $\Yoneda(f)$ is admissible.
\end{proof}

\begin{corollary}
	Let $\EE$ be a deflation-exact category having admissible kernels.  The Yoneda embedding $\Upsilon\colon \EE\to \smodad(\EE)$ has a left adjoint, sending an object $M \cong \coker \Yoneda(f) \in \modad(\EE)$ to $\coker(f)$, for each admissible morphism $f \in \EE$.
\end{corollary}

\begin{proof}
The proof is as in \cite[Corollary~3.10]{HenrardKvammevanRoosmalen20}.  The left adjoint $L\colon \modad(\EE) \to \EE$ is obtained by applying \Cref{Theorem:UniversalPropertyDeflationExactsmodad} to the identity $\EE \to \EE$, that is, $L \circ \Yoneda \cong 1$.  The explicit description is obtained using that $L$ commutes with cokernels.
\end{proof}

\begin{theorem}
Let $\EE$ be a deflation-exact category having admissible kernels.  The Yoneda embedding $\Upsilon\colon \EE\to \smodad(\EE)$ induces an equivalence $\EE\simeq \smodad(\EE)/\eff(\EE)$.
\end{theorem}

\begin{proof}
It follows from the above description of $L\colon \modad(\EE) \to \EE$ that $L(\eff(\EE)) = 0$.  Hence, by the universal property of the quotient, $L$ factors as $\modad(\EE) \xrightarrow{Q} \modad(\EE) / \eff(\EE) \xrightarrow{\overline{L}} \EE$.  We find that $Q \circ \Yoneda$ is left adjoint to $\overline{L}$ (see \cite[Lemma~1.3.1]{GabrielZisman67} with $F = \overline{L}$, $G = Q$, and $D = \Yoneda$, or \cite[Proposition~2.11.(1a)]{HenrardvanRoosmalen20Glider} with $F = \Yoneda$, $G = Q$, and $H = \overline{L}$).
		
		As $\overline{L} \circ Q \circ \Yoneda \cong 1_\EE$, we know that $Q \circ \Yoneda$ is fully faithful.  We only need to show that $Q \circ \Yoneda$ is essentially surjective.  For this, take an arbitrary $M \in \Ob(\modad(\EE)) = \Ob(\modad(\EE) / \eff(\EE))$.  Let $f\colon X \to Y$ be an admissible morphism in $\EE$ with $M \cong \coker \Yoneda (f)$.
		
				From the deflation-inflation factorization $X \stackrel{p}{\deflation} X' \stackrel{m}{\inflation} Y$ of $f$, we obtain the following conflation (see \Cref{proposition:TorsionPart}):
		\[0 \to \coker \Yoneda(p) \to M \xrightarrow{g} \coker \Yoneda(m) \to 0.\]
		As $\coker \Yoneda(p) \in \eff(\EE)$, the map $g\colon M \to \coker \Yoneda(m)$ is a weak isomorphism, i.e. $Q(g)$ is an isomorphism.  Consider now the conflation $X' \stackrel{m}{\inflation} Y \stackrel{q}{\deflation} Z$ in $\EE$.  As the Yoneda embedding is left conflation-exact, we get the following diagram
		\[\begin{tikzcd}
		  \Yoneda(X') \arrow[r, rightarrowtail, "\Yoneda(m)"] & \Yoneda(Y) \arrow[rr, "\Yoneda(q)"]\arrow[rd, twoheadrightarrow] && \Yoneda(Z) \\
			&&\coker(m) \arrow[ru, rightarrowtail, "h"]
		\end{tikzcd}\]
As $\coker(h) \cong \coker \Yoneda(q) \in \eff(\EE)$, we see that $h$ is a weak isomorphism as well.  We find that $Q(M) \cong Q(\coker \Yoneda(m)) \cong Q \circ \Yoneda (Z).$  Hence, $Q \circ \Yoneda\colon \EE \to \modad(\EE) / \eff(\EE)$ is essentially surjective, as required.
		\end{proof}

\begin{corollary}\label{Corollary:OneSidedAuslanderForAdmissibleKernels}
	Let $\EE$ be a deflation-exact category having admissible kernels. 
	\begin{enumerate}
		\item The Yoneda embedding $\Upsilon\colon \EE\to \smodad(\EE)$ induces triangle equivalences $\Kstar(\EE)\to  \Dstar(\smodad(\EE))$ for $\ast \in \{\text{b}, \text{-}, \varnothing\}$.
		\item There is a natural commutative diagram
		\[\xymatrix{
			\mathbf{D}_{\eff(\EE)}^b(\smodad(\EE))\ar[r]\ar[d]^{\simeq} & \Db(\smodad(\EE))\ar[r]\ar[d]^{\simeq} & \Db(\smodad(\EE)/\eff(\EE))\ar[d]^{\simeq}\\
			\Acb(\EE)\ar[r] & \Kb(\EE)\ar[r] & \Db(\EE)
		}\]
	\end{enumerate}
\end{corollary}

\begin{proof}
	\begin{enumerate}
		\item As every object of $\smod(\EE)$ has projective dimension at most two, we can apply \Cref{Theorem:PreResolvingDerivedEquivalence}.
		\item By \Cref{Theorem:LocalizationTheorem} one obtains the upper row. The right equivalences follow from the above. By \Cref{Proposition:BasicPropertiesDerivedCategoryNew}, $\Acb(\EE)$ is a thick triangulated subcategory of $\Kb(\EE)$ and thus $\mathbf{D}_{\eff(\EE)}^b(\smodad(\EE))$ is equivalent to $\Acb(\EE)$ as both categories are obtained as the kernel of the same Verdier localization.\qedhere
	\end{enumerate}
\end{proof}

\subsection{Some derived equivalences}  Let $\EE$ be a deflation-exact category with admissible kernels.  In \Cref{definition:VariousMods}, we introduced $\modad(\EE)$ and $\modinf(\EE)$, as well as their subcategories $\modadone(\EE)$ and $\modinfone(\EE)$ of objects of projective dimension at most one.  We now show that these categories are derived equivalent.  We start with the following observation.

\begin{lemma}\label{lemma:ModadinfClosedUnderSubjects}
The subcategory $\modinf(\EE)$ of $\mod(\EE)$ is closed under subobjects.
\end{lemma}

\begin{proof}
As in \Cref{Proposition:smodad(EE)IsDeflationExactWithAdmissibleKernels}, now using that the pullback of a weakly admissible morphism is weakly admissible (see \Cref{Lemma:AdmissiblesStableUnderPullbacks}).
\end{proof}

\begin{proposition}
Let $\EE$ be a deflation-exact category with admissible kernels.  For each $* \in \{\text{b}, \text{-}, \varnothing\}$, there are triangle equivalences
\[\begin{tikzcd}
\Dstar(\modad(\EE)) \arrow[r, "\simeq"] &\Dstar(\modinf(\EE)) \arrow[r, "\simeq"]& \Dstar(\mod(\EE)) \\
\Dstar(\modadone(\EE)) \arrow[u, "\simeq"] &\Dstar(\modinfone(\EE)) \arrow[u, "\simeq"]& \Dstar(\modone(\EE)) \arrow[u, "\simeq"]
\end{tikzcd}\]
\end{proposition}

\begin{proof}
Note that $\mod(\EE)$ has enough projective objects.  As these projective objects are contained in $\modinf(\EE)$, it follows that the embedding $\modinf(\EE) \to \mod(\EE)$ satisfies axiom \ref{PR1}.  By \Cref{lemma:ModadinfClosedUnderSubjects}, the category $\modinf(\EE)$ is closed under subobjects in $\mod(\EE)$, hence axiom \ref{PR2} holds.  Evenmoreso, $\resdim_{\modinf(\EE)} \mod(\EE) \leq 1$.  The required triangle equivalence $\Dstar(\modinf(\EE)) \to \Dstar(\mod(\EE))$ holds by \Cref{Theorem:PreResolvingDerivedEquivalence}.

The other equivalences are shown in a similar way.  For $\Dstar(\modad(\EE)) \to \Dstar(\modinf(\EE))$, we use \Cref{Proposition:smodad(EE)IsDeflationExactWithAdmissibleKernels}.  For the vertical maps, we use, from left to right, \Cref{Proposition:TorisionPairInModad}, \Cref{Proposition:TorsionTheoryInSmodinf}, and \Cref{proposition:HereditaryTorsionForExactHull}.%
%
%
%
\end{proof}
\section{The left heart as a localization of \texorpdfstring{$\hMon(\EE)$}{the monomorphisms}}\label{Section:AsLocalizationOfMonomorphisms}

Let $\EE$ be an additive regular category.  In \Cref{Section:ConstructionOfTheLeftHeart}, we showed that the left heart $\mathcal{LH}(\EE)$ can be obtained as the quotient $\smod(\EE) / \eff(\EE).$  When $\EE$ is quasi-abelian, then it has been shown in \cite{Schneiders99, Wegner17} that the left heart of $\EE$ can be described as a localization of the category $\hMon(\EE)$ of monomorphisms in $\EE$ (up to homotopy).  In this section, we give a similar description of the left heart of a deflation-exact category with admissible kernels.

Our approach is the following.  Let $(\TT, \FF)$ be a hereditary torsion theory in an abelian category $\AA$.  It follows from \Cref{Proposition:LocalizationAtTorsion} below that the quotient $\AA / \TT$ can be described as a localization of the torsionfree class $\FF$; specifically, one formally inverts all bimorphisms in $\FF$.

Applying this to the torsion pair $(\eff \EE, \modone \EE)$ in $\mod \EE$ shows that the quotient $\mod \EE / \eff \EE (\simeq \mathcal{LH}(\EE))$ can be obtained as a localization of $\modone(\EE)$ at the class of all \emph{bimorphisms} (= morphisms that are both epimorphisms and monomorphisms).  All that is left, is then to study the map $\Mon(\EE) \to \modone(\EE).$


The following observation allows us to obtain \Cref{Theorem:SIsMSinhMon} from the results in Section \ref{Section:ConstructionOfTheLeftHeart}.

\begin{proposition}\label{Proposition:LocalizationAtTorsion}
Let $(\TT, \FF)$ be a hereditary torsion theory in an abelian category $\AA$. Let $\Sigma_\TT \subseteq \AA$ be the set of all morphisms $f$ such that $\ker f, \coker f \in \TT$. 
\begin{enumerate}
	\item $\Sigma_\TT \cap \Mor \FF$ is a multiplicative system in $\FF$,
	\item $f \in \Mor(\FF)$ lies in $\Sigma_\TT$ if and only if $f$ is a bimorphism in $\FF$,
	\item the functor $\FF \to \AA$ induces an equivalence $\FF[(\Sigma_\TT \cap \Mor \FF)^{-1}] \xrightarrow{\simeq} \AA[\Sigma_\TT^{-1}]$.
	\end{enumerate}
\end{proposition}

\begin{proof}
	As $\TT$ is a Serre subcategory of $\AA$, we know that $\Sigma_\TT$ is a multiplicative system. By \cite[Proposition~3.1]{GabrielZisman67}, the localization functor $Q\colon \AA\to \AA[\Sigma_\TT^{-1}]$ commutes with kernels and cokernels (and thus is exact). Now write $\mathfrak{t}\colon \AA\to \TT$ for the torsion functor and write $\ff\colon \AA\to \FF$ for the torsion-free functor. For any object $A\in \AA$, the short exact sequence $\mathfrak{t}(A)\inflation A\deflation \ff(A)$ is mapped to $0\inflation Q(A)\deflation (Q\circ \ff)(A)$ under $Q$. This shows that the natural transformation $Q\to Q\circ \ff$ is a natural isomorphism. Note that $\FF$ has kernels (these coincide with kernels in $\AA$) and cokernels (these are given by $\ff \circ \coker_\AA$). Hence, $Q|_\FF\colon \FF \to \AA[\Sigma_\TT^{-1}]$ commutes with kernels and cokernels. It now follows from \cite[Proposition I.3.4]{GabrielZisman67} that $\Sigma_\TT \cap \Mor \FF$ is a multiplicative system in $\FF$.

Note that a morphism $f \in \Mor(\FF)$ lies in $\Sigma_\TT$ if and only if $\ker_\AA(f), \coker_\AA(f) \in \TT$, which is equivalent to $\ker_\FF(f), \coker_\FF(f) = 0$. This is then equivalent to $f$ being both a monomorphism and an epimorphism in $\FF$.

The last statement follows from \cite[Corollary 7.2.2]{KashiwaraSchapira06}.
\end{proof}

\begin{definition}\label{definition:Mon}
	We write $\Mon(\EE)$ for the category of monomorphisms in $\EE$, i.e.~the objects are monomorphisms $\delta_E\colon E^{-1}\hookrightarrow E^0$ in $\EE$ and morphisms are commutative squares.  Consider the ideal $\II$ in $\Mon(\EE)$ consisting of all squares
	\[\begin{tikzcd}
		{E^{-1}}\arrow[r, hookrightarrow, "\delta_E"] \arrow[d, "u_{-1}"] & {E^0} \arrow[d,"{u_0}"]\\
		F^{-1}\arrow[r, hookrightarrow, "\delta_F"] & F^0
		\end{tikzcd}\]
	for which there exists a morphism $t\colon E^0 \to F^{-1}$ satisfying $t \circ \delta_E = u_{-1}$ and $\delta_F \circ t = u_0.$  We define the category $\hMon(\EE)$ as $\Mon(\EE) / \II.$
\end{definition}

\begin{remark}\label{remark:MonEmbeddings}
There is a natural full embedding $\Mon(\EE) \to \C(\EE)$, mapping a monomorphism $\delta_E\colon E^{-1}\hookrightarrow E^0$ in $\EE$ to a complex with $E^{-1}$ and $E^0$ in degrees $-1$ and $0$, respectively, and zero elsewhere.  For the category $\hMon(\EE)$, there is a similar full embedding into $K(\EE).$
\end{remark}

\begin{proposition}\label{Proposition:hMonModOneEquivalent}
	The functor $\coker \circ \Upsilon \colon \hMon(\EE) \to \smod(\EE)$, mapping a monomorphism $\delta_E\colon E^{-1}\hookrightarrow E^0$ to $\coker \Yoneda(\delta_E) \in \smod(\EE)$, induces an equivalence $\hMon(\EE) \to \modone(\EE)$.
\end{proposition}

\begin{proof}
	This follows from the Comparison Theorem (see \cite[Theorem 12.4]{Buhler10}).
\end{proof}

\begin{remark}
By \Cref{proposition:HereditaryTorsionForExactHull}, $\modone(\EE)$ is a cotilting torsion-free class in $\smod(\EE)$ and thus \cite[Proposition~B.3]{BondalvandenBergh03} yields that $\hMon(\EE)\simeq \modone(\EE)$ is a quasi-abelian category. By \Cref{Theorem:PreResolvingDerivedEquivalence}, we have $\Dstar(\hMon(\EE)) \simeq \Dstar(\smod(\EE)) \simeq \Kstar(\EE)$, for $* \in \{\varnothing, \text{b}, -\}$.
\end{remark}

\begin{proposition}\label{Proposition:BicartesianMeansRegular}
A morphism $\delta_E \to \delta_F$ in $\Mon(\EE)$ is a bimorphism in $\hMon(\EE)$ if and only if it is a bicartesian square.
\end{proposition}

\begin{proof}
By \Cref{Proposition:hMonModOneEquivalent}, it suffices to show that bicartesion squares in $\hMon(\EE)$ correspond to bimorphisms in $\modone(\EE)$ under the functor $\coker \circ \Upsilon \colon \hMon(\EE) \to \smod(\EE).$  Since $(\eff(\EE), \modone(\EE))$ is a torsion pair in $\smod(\EE)$, a bimorphism in $\modone(\EE)$ is a morphism $\phi\colon F \to G$ such that $\ker \phi, \coker \phi \in \eff(\EE)$, equivalently, such that $\ker \phi = 0$ and $\coker \phi \in \eff(\EE)$.

Consider first a morphism $\phi\colon F \to G$ in $\modone(\EE)$.  Let $f\colon A \hookrightarrow B$ and $g\colon C \hookrightarrow D$ be monomorphisms in $\EE$ such that $\coker \Yoneda(f) \cong F$ and $\coker \Yoneda(g) \cong G$.  The morphism $\phi$ can be lifted to a commutative diagram 
	\[\xymatrix{
		A\ar[r]^{\beta}\ar[d]^f & C\ar[d]^{g}\\
		B\ar[r]^{\alpha} & D
	}\]
	in $\EE$.  We show that this square is bicartesian. Since $\coker \phi\in \eff(\EE)$, it follows from \Cref{Lemma:MorphismsRepresentingAnEffaceableAreDeflations} and \Cref{Lemma:TheFamousDiagramChase} that $\begin{pmatrix} g & \alpha \end{pmatrix} \colon C \oplus B \to D$ is a deflation. Next we take the pullback of $g$ along $\alpha$ and use the notation from \Cref{Lemma:TheFamousDiagramChase}.  Using that $\phi$ is a monomorphism and that $\ker(g)=0$, we have that $\beta''\colon A \to E$ is a retraction.  Furthermore, using that $g' \circ \beta'' = f$ is a monomorphism, we see that $\beta''$ is an isomorphism.  This shows that the square $ABCD$ is a pullback.  Hence, $A \to C \oplus B$ is the kernel of the deflation $C \oplus B \deflation D$, so that $A \inflation C \oplus B \deflation D$ is a conflation and may conclude that the square $ABCD$ is both a pullback and a pushout.
	
	The other implication, that a bicartesion square in $\EE$ corresponds to a bimorphism in $\modone(\EE)$, follows easily from \Cref{Lemma:TheFamousDiagramChase}.
\end{proof}

The previous proposition motivates introducing the following notation.

\begin{notation}\label{notation:BicartesionSquares}
We write $S$ for those morphisms $u\colon \delta_E\to\delta_F$ such that 
	\[\begin{tikzcd}
		{E^{-1}}\arrow[r, hookrightarrow, "\delta_E"] \arrow[d, "u_{-1}"] & {E^0} \arrow[d,"{u_0}"]\\
		F^{-1}\arrow[r, hookrightarrow, "\delta_F"] & F^0
		\end{tikzcd}\] is a bicartesian square. Furthermore, we write $\theta\colon \hMon(\EE)\to \K(\EE)$ for the embedding functor in \Cref{remark:MonEmbeddings}, mapping a monomorphism $\delta_E\colon E^{-1}\hookrightarrow E^0$ in $\EE$ to a complex with $E^{-1}$ and $E^0$ in degrees $-1$ and $0$, respectively, and zero elsewhere.
\end{notation}

By \Cref{Proposition:RepresentationsOfObjectsInLeftHeart}, there is a natural functor $\hMon(\EE)\to \mathcal{LH}(\EE)$.

\begin{theorem}\label{Theorem:SIsMSinhMon}
In the category $\hMon(\EE)$, the class $S$ of all bicartesian squares is a left and right multiplicative system.  The natural functor $\hMon(\EE) \to \mathcal{LH}(\EE)$ induces an equivalence $\hMon(\EE)[S^{-1}] \to \mathcal{LH}(\EE)$.
\end{theorem}

\begin{proof}
By  \Cref{proposition:HereditaryTorsionForExactHull}, the pair $(\eff(\EE), \modone(\EE))$ is a hereditary torsion pair in $\smod(\EE).$  Let $S'$ be the class of all bimorphisms in $\modone(\EE).$  It follows from \Cref{Proposition:LocalizationAtTorsion} that $S'$ is a multiplicative system in $\modone(\EE)$ and that $\modone(\EE)[(S')^{-1}] \simeq \smod(\EE) / \eff(\EE)$.

By \Cref{Proposition:hMonModOneEquivalent,Proposition:BicartesianMeansRegular}, we have $\modone(\EE)[(S')^{-1}] \simeq \hMon(\EE)[S^{-1}]$ and by \Cref{Theorem:AlternativeDescriptionOfLeftHeart} we have $\mathcal{LH}(\EE) \simeq \smod(\EE) / \eff(\EE)$.  This finishes the proof.
\end{proof}

\begin{lemma}\label{Lemma:BicartesianSquaresCorrespondToQuasiIsomorphisms}
	Let $u\colon \delta_E\to \delta_F$ be a morphism in $\hMon(\EE)$. Then $\theta(u)$ is a quasi-isomorphism if and only if $u\in S$. 
\end{lemma}

\begin{proof}
	Consider a morphism $u\colon \delta_E\to\delta_F$ given by the commutative diagram:
	\[\begin{tikzcd}
		E^{-1}\arrow[hook]{r}{\delta_E}\arrow{d}{u_{-1}} & E^0\arrow{d}{u_0}\\
		F^{-1}\arrow[hook]{r}{\delta_F} & F^0
	\end{tikzcd}\] The cone of $\theta(u)$ is given by 
	\[\xymatrix{
		\cdots\ar[r] & 0\ar[r] & E^{-1}\ar[r]^-{\begin{psmallmatrix}-\delta_E\\u_{-1}\end{psmallmatrix}} & E^{0}\oplus F^{-1} \ar[r]^-{\begin{psmallmatrix}u_0&\delta_F\end{psmallmatrix}}& F^0\ar[r] & 0\ar[r] & \cdots
	}\] Combining \Cref{Proposition:DeflationAICategorySatisfiesAxiomR3+} and \Cref{Proposition:BasicPropertiesDerivedCategoryNew}, we see that $u$ is a quasi-isomorphism if and only if $\cone(u)\in\Ac(\EE)$.  Hence, $u$ is a quasi-isomorphism if and only if $\xymatrix{E^{-1}\ar[r]^-{\begin{psmallmatrix}-\delta_E\\u_{-1}\end{psmallmatrix}} & E^0\oplus F^{-1} \ar[r]^-{\begin{psmallmatrix}u_0&\delta_F\end{psmallmatrix}} &F^0}$ is a conflation (equivalently, a kernel-cokernel pair by \Cref{REM}). The latter is clearly equivalent to requiring the above square to be bicartesian. This completes the proof.
\end{proof}

\begin{remark}
	We turn our attention back to the category $\Mon(\EE)$.  Let $\NN$ be the class of all objects $X \hookrightarrow Y$ which are isomorphisms, and let $[\NN]$ be the ideal of $\Mon(\EE)$ consisting of all morphisms factoring through an object of $\NN$.  It is straightforward to verify that $\hMon(\EE) \simeq \Mon(\EE) / [\NN]$.  It follows from \Cref{Proposition:BicartesianMeansRegular} that the set $S \subseteq \Mor (\Mon(\EE))$ of bicartesian squares is precisely the set of all bimorphisms in $\hMon(\EE)$.  In this case, the localization $\Mon(\EE)[S^{-1}]$ has also been denoted by $\Mon(\EE) / \NN$ in \cite{Rump07} (this notion differs from the one used in Section \ref{section:Quotients} as $\NN$ is neither an inflation- nor deflation-percolating subcategory of $\Mon(\EE)$).
\end{remark}

With a small abuse of notation, we write $S$ for the class of bicartesian squares in both $\Mon(\EE)$ and $\hMon(\EE)$, cf.~Notation \ref{notation:BicartesionSquares}.

\begin{proposition}\label{proposition:hMonVsMon}
The quotient functor $\Mon(\EE) \to \hMon(\EE)$ induces an equivalence $\Mon(\EE)[S^{-1}] \xrightarrow{\simeq} \hMon(\EE)[S^{-1}].$
\end{proposition}

\begin{proof}
	Consider a map $u\colon\delta_E\to \delta_F$ in $\Mon(\EE)$ and assume that $u$ is null-homotopic. Then there exists a map $h\colon E^{0}\to F^{-1}$ such that the diagram
	\[\begin{tikzcd}
	{E^{-1}} \arrow[r, "u_{-1}"] \arrow[d, hookrightarrow, "\delta_E"] & {F^{-1}} \arrow[d, hookrightarrow, "\delta_F"] \\
	{E^{0}} \arrow[r, "u_{0}"] \arrow[ru, "h"] & {F^{0}}
	\end{tikzcd}\]
 commutes. It follows that $u$ factors as follows:
	\[\begin{tikzcd}
	{E^{-1}} \arrow[r, "u_{-1}"] \arrow[d, hookrightarrow, "\delta_E"] & {F^{-1}} \arrow[d, equal] \arrow[r, equal]& {F^{-1}} \arrow[d, hookrightarrow, "\delta_F"] \\
	{E^{0}} \arrow[r, "h"] & {F^{-1}} \arrow[r, hookrightarrow, "\delta_F"] & {F^0}
	\end{tikzcd}\]
	As the square 
	\[\xymatrix{F^{-1}\ar@{=}[d]\ar[r] & 0\ar[d]\\ F^{-1}\ar[r] & 0}\]
	is bicartesian, $u=0$ in $\Mon(\EE)[S^{-1}]$. From this one readily deduces that $\Mon(\EE)[S^{-1}]\simeq \hMon(\EE)[S^{-1}]$, and the result follows.
\end{proof}

Similar results now hold for the full subcategories $\hWInf(\EE)$ (or $\hInf(\EE)$) of $\hMon(\EE)$ consisting of objects $\delta_E\colon E^{-1} \hookrightarrow E^0$ which are weak inflations (or inflations).

\begin{corollary}
	\begin{enumerate}
		\item The set $S_{\hWInf(\EE)} \coloneqq S \cap \Mor(\WInf(\EE))$ is a right multiplicative system in $\hWInf(\EE).$  Moreover, we have $\hWInf(\EE)[S^{-1}_{\hWInf(\EE)}] \simeq \ex{\EE}.$
		\item The set $S_{\hInf(\EE)} \coloneqq S \cap \Mor(\Inf(\EE))$ is a right multiplicative system in $\hInf(\EE).$  Moreover, we have $\hInf(\EE)[S^{-1}_{\hInf(\EE)}] \simeq {\EE}.$
	\end{enumerate}
\end{corollary}

\begin{proof}
Following \Cref{Lemma:AdmissiblesStableUnderPullbacks}, we know that weak inflations are stable under pullbacks.  Hence, for any morphism $f\colon \delta_E \to \delta_F$ in $\hMon(\EE)$, if $f$ is a pullback square and $\delta_F \in \hWInf(\EE)$, we know that $\delta_E \in \hWInf(\EE).$  It now follows from \cite[Proposition 7.2.1]{KashiwaraSchapira06} that $S_{\hWInf(\EE)}$ is a right multiplicative set and the induced functor $\hWInf(\EE)[S^{-1}_{\hWInf(\EE)}] \to \hMon(\EE)[S^{-1}]$ is fully faithful.

It follows from \Cref{Theorem:ExactHull} that every object $Z \in \ex{\EE}$ fits in a conflation $X \inflation Y \deflation Z$ in $\ex{\EE}$ with $X,Y\in \EE$.  It follows from \Cref{Corollary:HullOfAIIsL1Closure} that the inflation $X \inflation Y$ in $\ex{\EE}$ is a finite composition of inflations in $\EE$.  Therefore the restriction of the functor $\coker \circ \Yoneda\colon \hMon(\EE) \to\mathcal{LH}(\EE)$ to the subcategory $\hWInf(\EE)$ gives an equivalence between $\hWInf(\EE)[S^{-1}_{\hWInf(\EE)}]$ and the subcategory $\ex{\EE}$ of $\mathcal{LH}(\EE).$

The second statement is proven in a similar fashion.  
\end{proof}

%
\section{The heart of the LB-spaces}\label{Section:LBSpaces}

For this section, let $k$ be either the field of real or the field of complex numbers.  Let us denote by $\LB$ the category of LB-spaces. Its objects are by definition all those Hausdorff locally convex topological $k$-vector spaces $(X,\tau)$ that can be represented by an $\bN$-indexed direct limit of Banach spaces, meaning that there are Banach spaces $X_0\hookrightarrow X_1\hookrightarrow X_2\hookrightarrow\cdots$ with continuous injective linking maps such that $X=\bigcup_{n=1}^{\infty}X_n$ holds as linear spaces and $\tau$ is the finest linear topology that makes all inclusion maps $X_n\hookrightarrow(X,\tau)$ continuous. A morphism between LB-spaces is by definition a linear and continuous map.

It is well-known that $\LB$ is a pre-abelian category. Indeed, given a morphism $f\colon X\rightarrow Y$, then its cokernel is given by $\coker(f)\colon Y\rightarrow Y/\overline{f(X)}$ where $\overline{f(X)}$ is the topological closure of $f(X)$ and where $Y/\overline{f(X)}$ is endowed with the locally convex quotient topology (this is then again of the $\operatorname{LB}$-type explained above). Its kernel is given by $\ker(f)\colon f^{-1}(0)\rightarrow X$
where $f^{-1}(0)$ carries the direct limit topology of the sequence $X_0\cap f^{-1}(0)\hookrightarrow X_1\cap f^{-1}(0)\hookrightarrow X_2\cap f^{-1}(0)\hookrightarrow\cdots$ of Banach spaces. The latter can be strictly finer than the subspace topology, see \cite[Example 6.8.13]{BPC} for an example. To indicate that we are not using the subspace topology, we will write $f^{-1}(0)^{\flat}$ for the kernel in $\LB$. From this discussion it follows that a pair of composable morphisms
\[
X\stackrel{f}{\longrightarrow}Y\stackrel{g}{\longrightarrow}Z
\]
is a kernel-cokernel pair in $\LB$ if and only if $f$ is injective, $g$ is surjective and $f(X)=g^{-1}(0)$ holds as linear spaces. Observe that, in this case, $f(X)\subseteq Y$ is automatically closed, but that $f(X)$ (or, equivalently, $g^{-1}(0)$) endowed with the induced topology of $Y$ is in general \emph{not} an LB-space.  We write $\bC_{\mathsf{all}}$ for the class of all kernel-cokernel pairs in $\LB$.

\begin{theorem}\label{LB-THM-1} The category $\LB$ is a deflation-exact category with respect to the conflation structure $\mathbb{C}_{\mathsf{all}}$ of all kernel-cokernel pairs.  In particular, $(\LB,\mathbb{C}_{\mathsf{all}})$ has admissible kernels. The conflation structure $\mathbb{C}_{\mathsf{all}}$ is not exact.
\end{theorem}

\begin{proof} By \cite[Theorem~3.4]{HassounShahWegner20}, which had been mentioned without proof in \cite[p.~540]{KopylovWegner12}, the category $\LB$ is deflation quasi-abelian but not inflation quasi-abelian. This means explicitly, that in every pullback diagram
\begin{equation*}
\begin{tikzcd}
A\arrow{r}{a}\arrow{d}[swap]{b} \commutes[\text{PB}]{dr}& B \arrow{d}{c}\\ C \arrow{r}[swap]{d} & D
\end{tikzcd}
\end{equation*}
$a$ is a cokernel whenever this is true for $d$, and that there exists a pushout diagram
\begin{equation*}
\begin{tikzcd}
A\arrow{r}{a}\arrow{d}[swap]{b} \commutes[\text{PO}]{dr}& B \arrow{d}{c}\\ C \arrow{r}[swap]{d} & D
\end{tikzcd}
\end{equation*}
in which $a$ is a kernel but $d$ is not. The latter statement implies immediately that $\mathbb{C}_{\mathsf{all}}$ cannot be an exact structure.
\end{proof}

\begin{remark}
At first sight, and in light of \Cref{Lemma:InterpretationOfAI}, the previous result might appear to be inconsistent with \cite[Theorem~6.1]{HassounShahWegner20} which reads `every pre-abelian category with the admissible intersection property is quasi-abelian'. Notice however, that in \cite[Theorem~6.1]{HassounShahWegner20} the admissible intersection property is required with respect to a conflation structure which is exact.
\end{remark}

Applying our results from the previous sections, the category $\LB$ admits a heart which is by \Cref{Theorem:SIsMSinhMon} equivalent to the localization of its monomorphism category modulo homotopy (denoted earlier in this paper by $\hMon(\LB)$) by the class of bicartesian squares. Writing the latter down explicitly for the LB-spaces gives
\[
\mathcal{LH}(\LB)\simeq ({\operatorname{hMon}\LB})\bigl[\{\text{bicartesian squares}\}^{-1}\bigr]
\]
where the right hand side coincides with the category that was defined in an ad hoc fashion and established to be abelian in \cite[Theorem~10 and Proposition 14]{Wegner17} (see also \cite{Schneiders99}). In addition to recovering this ad hoc approach, our results show that the category defined in \cite{Wegner17} is indeed derived equivalent to the category we started with.

\begin{theorem}\label{LB-COR-1} With respect to the conflation structure $\mathbb{C}_{\mathsf{all}}$, the embeddings $\LB\to\LB^{\mathsf{ex}}\to \mathcal{LH}(\LB)$ lift to triangle equivalences $\Dstar(\LB)\to \Dstar(\LB^{\mathsf{ex}})\to\Dstar(\mathcal{LH}(\LB))$ with $*\in \{-,b,\emptyset\}$.
\end{theorem}

\begin{proof} By \Cref{LB-THM-1} the category $\LB$ is deflation-exact and has admissible kernels. Thus, \Cref{Theorem:EmbeddingInHeartIsTriangleEquivalence} and \Cref{Proposition:PhiFactorsThroughHull} imply the result.
\end{proof}

\begin{remark}
It is shown in \Cref{Theorem:SIsMSinhMon} that the class $S$ of all bicartesian squares is a multiplicative system in $\hMon(\LB)$.  It follows from \Cref{proposition:hMonVsMon} that $\mathcal{LH}(\LB) \simeq \Mon(\LB)[S^{-1}]$, so one can opt to describe $\mathcal{LH}(\LB)$ starting from $\Mon(\LB)$ instead of $\hMon(\LB).$  However, the class of bicartesian squares $S$ is not a multiplicative system in $\Mon(\LB)$.  Indeed, the localization $\Mon(\LB) \to \Mon(\LB)[S^{-1}]$ does not commute with kernels as can be seen from the following example.  Let $E \in \LB$ be a nonzero object and consider the following two objects in $\Mon(\LB)$: the zero morphism $\delta\colon 0 \to E$ and the identity $\delta'\colon E \to E.$  The morphism $f\colon \delta \to \delta'$ given by
\[
\xymatrix{ 0 \ar[r]^{\delta} \ar[d]^{f^{-1}} & E \ar@{=}[d]^{f^0} \\
E \ar[r]^{\delta'} & E}\]
is a monomorphism in $\Mon(\LB)$ (as both components of $f\colon \delta \to \delta'$ are monomorphisms), but not in $\Mon(\EE)[S^{-1}]$ (as $\delta'$ is zero in $\Mon(\LB)[S^{-1}]$ but $\delta$ is not).  This implies that the localization $\Mon(\LB) \to \Mon(\LB)[S^{-1}]$ does not commute with kernels.
\end{remark}

In addition to the natural, but non-exact, conflation structure $\mathbb{C}_{\mathsf{all}}$, the category $\LB$ admits at least two natural conflation structures that are exact. Let us denote by $\mathbb{E}_{\mathsf{top}}$ the class of \emph{topologically exact sequences} which consists by definition of all pairs $(f,g)$ of composable morphisms that form an exact sequence of vector spaces in which $f$ is closed and $g$ is open. Notice that, due to the Open Mapping Theorem for LB-spaces, the second condition is satisfied automatically.  On the other hand, let us write $\mathbb{E}_{\max}$ for the conflation structure given by all kernel-cokernel pairs $(f,g)$ in which every pushout of $f$ is again a kernel and every pullback of $g$ is again a cokernel.  By \cite{RichmanWalker77, SiegWegner11}, the latter is the maximal exact structure.

\begin{proposition} Consider the category $\LB$ of LB-spaces.
\begin{enumerate}
	
\item \cite[Proposition~3.3]{DierolfSieg12} The conflation category $(\LB, \mathbb{E}_{\mathsf{top}})$ is exact and we have
\[
\mathbb{E}_{\mathsf{top}}=\bigl\{(f,g)\in\mathbb{C}_{\mathsf{all}}\:\big|\: g^{-1}(0)^{\flat} = g^{-1}(0) \text{ as topological spaces}\bigr\}
\]
where we understand that $g^{-1}(0)$ carries the subspace topology.

\item \cite[Proposition~2.2.4 and Remark 2.2.6]{Dierolf14} The conflation category $(\LB,\mathbb{E}_{\mathsf{max}})$ is exact, we have
\[
\mathbb{E}_{\mathsf{max}}=\bigl\{(f,g)\in\mathbb{C}_{\mathsf{all}}\:\big|\: \Hom(g^{-1}(0)^{\flat},k) = \Hom(g^{-1}(0),k) \text{ as vector spaces}\bigr\}
\]
and $\mathbb{E}_{\mathsf{top}}\subset\mathbb{E}_{\mathsf{max}}$ is a proper subclass.\qedhere
\end{enumerate}
\end{proposition}

Let us mention that the exact structure $\mathbb{E}_{\mathsf{top}}$ is inherited by $\LB$ from the category of all Hausdorff locally convex spaces, see \cite{DierolfSieg12}. The latter category is quasi-abelian and its topologically exact sequences are precisely all kernel-cokernel pairs.  Our final theorem shows however, that no exact structure on $\LB$ does induce a derived equivalence with $(\LB,\mathbb{C}_{\mathsf{all}})$. 

\begin{theorem}\label{theorem:NoDerivedEquivalenceForExact}
 Let $\mathbb{E}$ be any exact structure on $\LB$. Then $(\LB,\mathbb{E})\to (\LB,\mathbb{C}_{\mathsf{all}})$ does \emph{not} lift to a triangle equivalence $\Dstar(\LB,\mathbb{E})\to \Dstar(\LB,\mathbb{C}_{\mathsf{all}})$. Consequently, none of the natural functors $\Dstar(\LB,\mathbb{E}_{\mathsf{top/max}})\to \Dstar(\mathcal{LH}(\LB,\mathbb{C}_{\mathsf{all}}))$ is a triangle equivalence, either.
\end{theorem}

\begin{proof}
	As $\mathbb{E}\subseteq \mathbb{C}_{\mathsf{all}}$, the identity $(\LB,\mathbb{E})\to (\LB,\mathbb{C}_{\mathsf{all}})$ lifts to a triangle functor $\Dstar(\LB,\mathbb{E})\to \Dstar(\LB,\mathbb{C}_{\mathsf{all}})$. As $(\LB,\mathbb{C}_{\mathsf{all}})$ is not exact, there is a conflation $X\inflation Y\deflation Z$ in $(\LB,\mathbb{C}_{\mathsf{all}})$ which is not a conflation in $(\LB,\mathbb{E})$. Extending the above conflation to a complex $U^{\bullet}$, \Cref{Proposition:BasicPropertiesDerivedCategoryNew} implies that $U^{\bullet}\in \Ac(\LB,\mathbb{C}_{\mathsf{all}})$ but $U^{\bullet}\notin \Ac(\LB,\mathbb{E})$ by \Cref{Proposition:BasicPropertiesDerivedCategoryNew}. It follows that $\Dstar(\LB,\mathbb{E})\to \Dstar(\LB,\mathbb{C}_{\mathsf{all}})$ is not faithful.
\end{proof}

\begin{remark}
As the proof indicates, the statement of Theorem \ref{theorem:NoDerivedEquivalenceForExact} holds after replacing $\LB$ with any additive regular category which is not an exact category (when endowed with the maximal conflation structure).  The dual statement holds for additive coregular categories.
\end{remark}

We conclude this article by outlining that the dual situation, i.e.~inflation-exact categories having admissible cokernels (or, thus, additive coregular categories), appear naturally in the functional analytic context as well.

\begin{example} The category $\mathsf{COM}$ of complete Hausdorff locally convex spaces, furnished with linear and continuous maps as morphisms, is inflation quasi-abelian and not deflation quasi-abelian, see \cite[Theorem 3.3]{HassounShahWegner20}. As in the proof of Theorem \ref{LB-THM-1} it follows that the latter category is inflation-exact and has admissible cokernels if endowed with the conflation structure consisting of all kernel-cokernel pairs. The latter contains the maximal exact structure as a proper subclass. Consequently, the embedding $(\mathsf{COM},\mathbb{C}_{\mathsf{all}})\to\mathcal{RH}(\mathsf{COM})$ lifts to an equivalence of bounded derived categories, whereas the functor $(\mathsf{COM},\mathbb{E}_{\mathsf{max}})\to\mathcal{RH}(\mathsf{COM})$ does not.	
\end{example}

\begin{example}
Let $\operatorname{Top}_\bZ^{\text sc}$ be the category of complete and separated topological groups with linear topology.  Likewise, for a field $k$, we write $\operatorname{Top}_k^{\text sc}$ for the category of complete and separated topological $k$-vector spaces with linear topology.  It is shown in \cite[Proposition 8.3]{Positselski20} that the categories $\operatorname{Top}_\bZ^{\text sc}$ and $\operatorname{Top}_k^{\text sc}$ are inflation quasi-abelian (thus, inflation-exact categories with admissible cokernels, or equivalently, satisfying the admissible cointersection property) but not quasi-abelian.
\end{example}

\begin{example}
Let $\PLS$ be the category of countable projective limits of Silva spaces (see \cite{Domanski04}); examples of $\PLS$-spaces include the space of distributions and the space of real analytic functions.  It is shown in \cite[Theorem 7]{LawsonWegner21} that the category $\PLS$ is inflation quasi-abelian (but not quasi-abelian).  Hence, $\PLS$ is an additive coregular category.  Similar statements hold for the categories $\PLN$ and $\PLSW$.
\end{example}

\section*{Compliance with ethical standards}

\subsection*{Conflict of interest} On behalf of all authors, the corresponding author states that there is no conflict of
interest.

\subsection*{Data availability} Data sharing is not applicable to this article as no new data were created or
analyzed in this study.






\def\bysame{\leavevmode\hbox to3em{\hrulefill}\thinspace}

\end{document}